\documentclass[11pt,a4paper]{article}
\usepackage{amsmath, amsfonts, amsthm, amssymb, graphicx, color, linegoal,
hyperref, comment}
\usepackage[numbers,square, comma, sort&compress]{natbib}
\usepackage{enumerate}
\usepackage{enumitem}
\usepackage{bm}
\usepackage{tikz}
\usetikzlibrary{calc}
\usetikzlibrary{shapes.geometric}
\usepackage{float}
\usepackage[export]{adjustbox}
\usepackage{algorithm}
\usepackage[noend]{algpseudocode}

\usepackage[small,bf,hang]{caption} % Improved caption
\usepackage{subcaption}

\usepackage[english]{babel}

\pgfdeclarelayer{bg}    % declare background layer
\pgfsetlayers{bg,main}

\newtheorem{theorem}{Theorem}
\newtheorem{corollary}{Corollary}
\newtheorem{lemma}{Lemma}

\newtheorem{observation}{Observation}

\theoremstyle{remark} %Change to plain if statement should be in italics.

\def\:{{\scalebox{1}[.75]{$\, \vphantom{\int^0}\smash[t]{\vdots} \,$}}}

\newcommand{\Carol}[1]{{\color{cyan} Carol: #1}}

\begin{document}

\title
{$K_2$-Hamiltonian Graphs: II}
\author{{\sc Jan GOEDGEBEUR\footnote{Department of Computer Science, KU Leuven Campus Kulak-Kortrijk, 8500 Kortrijk, Belgium}\;\footnote{Department of Applied Mathematics, Computer Science and Statistics, Ghent University, 9000 Ghent, Belgium}\;,
Jarne RENDERS\footnotemark[1]\;,
G\'abor WIENER\footnote{Department of Computer Science and Information Theory, 
Faculty of Electrical Engineering and Informatics, Budapest University of 
Technology and Economics, Budapest, Hungary}\;, }\\[1mm]
and {\sc Carol T. ZAMFIRESCU\footnotemark[2]\;\footnote{Department of Mathematics, Babe\c{s}-Bolyai University, Cluj-Napoca, Roumania}}\;\footnote{E-mail addresses: jan.goedgebeur@kuleuven.be; jarne.renders@kuleuven.be;  wiener@cs.bme.hu; czamfirescu@gmail.com}}

\date{}

\maketitle
\begin{center}
\begin{minipage}{125mm}
{\bf Abstract.} In this paper we use theoretical and computational tools to continue our investigation of \emph{$K_2$-hamiltonian} graphs, that is, graphs in which the removal of any pair of adjacent vertices yields a hamiltonian graph, and their interplay with \emph{$K_1$-hamiltonian} graphs, that is, graphs in which every vertex-deleted subgraph is hamiltonian. Perhaps surprisingly, there exist graphs that are both $K_1$- and $K_2$-hamiltonian, yet non-hamiltonian, for example, the Petersen graph. Gr\"unbaum conjectured that every planar $K_1$-hamiltonian graph must itself be hamiltonian; Thomassen disproved this conjecture. Here we show that even planar graphs that are both $K_1$- and $K_2$-hamiltonian need not be hamiltonian, and that the number of such graphs grows at least exponentially.
%In particular, we show that there exist infinitely many planar $K_1$- and $K_2$-hamiltonian yet non-hamiltonian graphs, and that their number grows at least exponentially.
%Motivated by a conjecture of Gr\"unbaum we show that there exist infinitely many planar $K_2$- and $K_1$-hamiltonian yet non-hamiltonian graphs.
Motivated by results of Aldred, McKay, and Wormald, we determine for every 
integer $n$ that is not 14 or 17 whether 
there exists a \textit{$K_2$-hypohamiltonian}, that is, non-hamiltonian and 
$K_2$-hamiltonian, graph of order~$n$, and 
characterise all orders for which such cubic graphs and such snarks 
exist. We also describe the smallest cubic planar graph which is 
$K_2$-hypohamiltonian,
as well as the smallest planar $K_2$-hypohamiltonian graph of girth $5$. 
We conclude with open problems and by correcting two inaccuracies from the 
first article.
%\Jan{TODO: mention more results from Section 3 in the abstract (i.e.\ all non-hamiltonian $K_2$-hamiltonian up to a given order for specific families + infinite constructions)?!}

\bigskip


\bigskip


\end{minipage}
\end{center}

\vspace{1cm}

\section{Introduction}

We shall call a graph \emph{$K_1$-hamiltonian} if all of its vertex-deleted subgraphs are hamiltonian and \emph{$K_2$-hamiltonian} if the removal of any pair of adjacent vertices yields a hamiltonian graph. Non-hamiltonian $K_1$-hamiltonian graphs are called \emph{hypohamiltonian} and, in the same vein, we say that a graph is \emph{$K_2$-hypohamiltonian} if it is $K_2$-hamiltonian yet non-hamiltonian. Here we will mainly focus on properties of $K_2$-hypohamiltonian graphs, but we will also discuss their interplay with hamiltonicity and $K_1$-hamiltonicity. This paper is a continuation of~\cite{Za21}; for a motivation of our interest in $K_2$-hamiltonian graphs we refer to that article's introduction, and references therein.

An $n$-vertex $K_2$-hamiltonian graph may contain a hamiltonian cycle but no
$(n-1)$-cycle (e.g.\ a $d$-dimensional cube for $d \ge 3$), but it may also
contain no hamiltonian cycle yet be $K_1$-hamiltonian, e.g.\ the Petersen graph.
As we shall see, the Petersen graph is in fact the smallest $K_2$-hypohamiltonian
graph. Coxeter's graph, another famous hypohamiltonian graph, is not
$K_2$-hamiltonian. Furthermore, every planar 4-connected graph is hamiltonian,
$K_1$-hamiltonian, $K_2$-hamiltonian, and indeed even \emph{$K_3$-hamiltonian}
(i.e.\ the removal of any triangle yields a hamiltonian graph) by work of
Tutte~\cite{Tu56}, Nelson~\cite{Ne73}, Thomas and Yu~\cite{TY94}, and
Sanders~\cite{Sa96}. The reader might wonder if a $K_2$-hamiltonian graph on
$n$~vertices and containing neither $n$- nor $(n-1)$-cycles exists;
%\Jan{@Jarne: did you verify if all of the $K_2$-hypohamiltonian graphs we
%generated have an $n-1$ cycle? We should probably mention in
%Section~\ref{sect:orders} that we verified this}
this is a good question to which we do not know the answer. It is a special
but interesting case of a conjecture of Gr\"unbaum~\cite{Gr74} which states 
that for any integer $k \ge 2$,
there exists no graph whose order and circumference differ by $k$ and in which
any set of $k$ vertices is avoided by some longest cycle. A relaxation of this
problem of Gr\"unbaum is also still open: Katona, Kostochka, Pach, and
Stechkin~\cite{KKPS89} asked whether an $n$-vertex graph in which any $n - 2$
vertices induce a hamiltonian graph, must itself be hamiltonian. (In fact,
Katona et al.\ had raised this problem far more generally; one obtains their
original question by replacing ``$n-2$'' with $k$ for $n/2 < k < n - 1$, where
$n$ denotes the graph's order.) %The construction of Meredith~\cite{Me73} gives
%a non-hamiltonian $K_{k-1,k}$-hamiltonian $k$-regular $k$-connected graph for
%any $k \ge 3$.
Van Aardt, Burger, Frick, Llano, and Zuazua raised in~\cite{ABFLZ14} an 
equivalent problem; see their Question~1.

It follows from a result of Thomason~\cite{Th78-2} that if a cubic graph contains a vertex-deleted subgraph which has an odd number of hamiltonian cycles, then the graph itself must be hamiltonian; and that if a cubic graph has an odd number of hamiltonian cycles, it is $K_1$-hamiltonian. The first statement is not true if we replace $K_1$ by $K_2$: simply take the Petersen graph. It is cubic and non-hamiltonian, but removing adjacent vertices yields a graph containing exactly one hamiltonian cycle. The second statement's $K_2$-variant is not true either: take any graph obtained from $K_4$ by subsequently replacing vertices by triangles. The resulting graph contains exactly three hamiltonian cycles, but also a triangle whose vertices are cubic, so it cannot be $K_2$-hamiltonian (see~\cite[Proposition 1]{Za21}). 

In Section~\ref{sect:variations_grunbaum} we address, as announced 
in~\cite{Za21}, the natural question whether planar $K_2$-hypohamiltonian 
graphs exist. Gr\"unbaum had conjectured~\cite{Gr74} that every planar 
$K_1$-hamiltonian graph is hamiltonian, but Thomassen~\cite{Th76} disproved 
this conjecture. On the other hand, in a subsequent paper, Thomassen showed 
that every planar $K_1$-hamiltonian graph \emph{without cubic vertices} is 
hamiltonian~\cite{Th78}, thereby giving an elegant amendment of Gr\"unbaum's 
conjecture so that it does hold. It is natural to wonder about other ways in 
which this conjecture can be ``saved''. We will prove that by replacing $K_1$ 
with $K_2$ one \emph{cannot} save the conjecture, i.e.\ it is not true that 
every planar graph in which every $K_2$-deleted subgraph is hamiltonian, must 
itself be hamiltonian. Moreover, as we shall see, in a planar graph not even 
the hamiltonicity of every $K_1$-deleted \emph{and} $K_2$-deleted subgraph 
necessarily implies the hamiltonicity of the graph itself.
In fact, we will even show that the number of planar hypohamiltonian 
$K_2$-hamiltonian graphs grows at least exponentially, giving an alternative 
proof of a result of Collier and Schmeichel \cite{CS77}.

In Section~\ref{sect:orders} we determine all $K_2$-hypohamiltonian graphs up to a given order for specific classes of graphs (i.e.\ general graphs, cubic graphs, snarks, planar graphs, cubic planar graphs) and characterise the orders for which they exist using a combination of computational methods and theoretical arguments.

Finally, in Section~\ref{sect:last}, we conclude the paper with an erratum 
concerning \cite{Za21} and discuss potential future directions of research.

We end this section with the introduction of definitions and notations that will be used throughout the remainder of the article. We refer to~\cite{graph_theory_diestel} for any standard graph theory terminology which is not explicitly defined, but note that here we write $K_n$ (not $K^n$) for the $n$-vertex complete graph.
%\Jan{I referred to Diestel to avoid having to define basic concepts such as connectivity.}

For a set $S$, we say that $A, B \subset S$ \emph{partition} $S$ if $A \cap B 
= 
\emptyset$ and $A \cup B = S$. In a graph $G$, consider $S \subset V(G)$. 
Then 
the \emph{induced subgraph} $G[S]$ is the graph whose vertex set is $S$ and 
whose edge set consists of all of the edges in $E(G)$ with both endpoints in 
$S$. Let $G$ be a non-complete graph of connectivity $k$ and order greater 
than 
$k$, $X$ a $k$-cut in $G$, and $C$ a component of $G - X$. Then $G[V(C) \cup 
X]$ is an \emph{$X$-fragment} of $G$ with \emph{attachments} $X$. An 
$X$-fragment is \emph{trivial} if it contains exactly $|X| + 1$ vertices. 
Sometimes we wish to emphasise the cardinality of the set $X$ of attachments, 
so we also say that $G[V(C) \cup X]$ is an \emph{$|X|$-fragment}. Let $F,F'$ 
be 
disjoint 3-fragments of graphs of connectivity~3, and let $F$ have 
attachments 
$x_1,x_2,x_3$ and $F'$ have attachments $x'_1,x'_2,x'_3$. Identifying $x_i$ 
with $x'_i$ for all $i$, we obtain the graph $(F, \{ x_1,x_2,x_3 \}) \: (F', 
\{ 
x'_1,x'_2,x'_3 \})$. If each of two disjoint fragments admits a planar 
embedding with cofacial attachments, it is easy to see that the 
identifications 
from the aforementioned operation $\:$ can be applied such that the resulting 
graph is planar, as well.
%\Carol{Someone check this. We were using it tacitly which I found 
%problematic.}. \Gabor{I checked it and it's fine. Something like "It is easy 
%to 
%see that..." might be included.}
 A cut $X$ of $G$ is \emph{trivial} if $G - X$ has exactly two components, 
 one 
 of which is $K_1$. A path with end-vertex $v$ is a \emph{$v$-path}, and a 
 $v$-path with end-vertex $w \ne v$ is a \emph{$vw$-path}.

%\section{A new construction method for hypohamiltonian and \boldmath$K_2$-hypohamiltonian graphs}
\section{Variations of a conjecture of Gr\"unbaum}
\label{sect:variations_grunbaum}

Motivated by Gr\"unbaum's conjecture that every planar $K_1$-hamiltonian graph 
is hamiltonian~\cite{Gr74} and Thomassen's elegant way to save it~\cite{Th78}, 
we now give two proofs %\Jan{Can you maybe give a small hint to the reader of 
%why it is useful to have two proofs for this?}
of the fact that even if in a planar graph the deletion of any single vertex \emph{and} the deletion of any pair of adjacent vertices yields a hamiltonian graph, the presence of a hamiltonian cycle is not guaranteed---a property shared by the Petersen graph. We give infinitely many examples, and note that they in fact satisfy the stronger property that any complete graph we remove, we get a hamiltonian graph (simply because they contain no $K_t$ for all integers $t \ge 3$). For our first proof we introduce a new method inspired by an old idea of Chv\'atal---his so-called \emph{flip-flop} graphs~\cite{Ch73}---while the second proof is based on a lemma from the first article in this series~\cite{Za21}. This second proof is shorter, but the first one builds a framework which we believe could be useful in attacking related problems concerning longest paths and longest cycles, in particular regarding graphs of minimum degree 4 and maybe even the notoriously difficult 4-connected case.

Let $G$ be a graph and let $a,b,c,d \in V(G)$ be pairwise distinct. The pair of vertices $(a,b)$ is said to be \emph{good} if there is a hamiltonian $ab$-path in $G$, and the pair of pairs of vertices $((a,b),(c,d))$ is \emph{good} if there exist an $ab$-path and a $cd$-path in $G$ whose vertex sets partition $V(G)$. Otherwise, in either case, the pair is said to be \emph{bad}. The quintuple $(G,a,b,c,d)$ is a \emph{cell} if $G$ is a graph and $a,b,c,d \in V(G)$ are pairwise distinct. These four vertices shall be called \emph{outer vertices} and every vertex that is not an outer vertex is an \emph{inner vertex}. A cell is \emph{suitable} if all of the following properties hold.

%Itemize seems less nice because of 1.5(a) and 1.5(b).
%\begin{itemize}\setlength\itemsep{-1.5em}
%\item[1.1.] The pairs $(a,b), (a,c), (b,d)$, and $(c,d)$ are good.\\[1mm]
%\item[1.2.] The pairs $(a,d)$ and $(b,c)$ are bad.\\[1mm]
%\item[1.3.] The pair of pairs $((a,b),(c,d))$ is good.\\[1mm]
%\item[1.4.] The pairs of pairs $((a,d),(b,c))$ and $((a,c),(b,d))$ are bad.\\[1mm]
%\item[1.5(a).] For any outer vertex $v$, the pairs $(a,b), (a,c), (b,d)$, and  $(c,d)$ are bad in $G - v$.\\[1mm]
%\item[1.5(b).] For any outer vertex $v$, the pairs $(a,d)$ and $(b,c)$ are bad in $G - v$.\\[1mm]
%\item[1.6.] Deleting any of the pairs of vertices $(a,c)$, $(a,d)$, $(b,c)$, $(b,d)$, the remaining pair of outer vertices is good in the resulting graph (e.g.~the pair $(b,d)$ is good in $G-a-c$, etc.).
%\end{itemize}

\smallskip

\noindent 1.1. The pairs $(a,b), (a,c), (b,d)$, and $(c,d)$ are good.\\[1mm]
1.2. The pairs $(a,d)$ and $(b,c)$ are bad.\\[1mm]
1.3. The pair of pairs $((a,b),(c,d))$ is good.\\[1mm]
1.4. The pairs of pairs $((a,d),(b,c))$ and $((a,c),(b,d))$ are bad.\\[1mm]
1.5(a). For any outer vertex $v$, the pairs $(a,b), (a,c), (b,d)$, and  $(c,d)$ which do not contain $v$ as one of the vertices are bad in $G - v$.\\[1mm]
1.5(b). For any outer vertex $v$, the pairs $(a,d)$ and $(b,c)$ which do not contain $v$ as one of the vertices are bad in $G - v$.\\[1mm]
1.6. Deleting any of the pairs of vertices $(a,c)$, $(a,d)$, $(b,c)$, $(b,d)$, the remaining pair of outer vertices is good in the resulting graph (e.g.~the pair $(b,d)$ is good in $G-a-c$, etc.).

%For any pair $\wp \in \{ (a,c), (a,d), (b,c), (b,d) \}$, the pair $\{ a, b, c, d \} - \wp$ is good in $G - \wp$ (e.g.~$(b,d)$ is good in $G - a - c$). \emph{CZ: This is perhaps not the best way to write it---feel free to improve.}

\smallskip

Property 1.5 states that, for any outer vertex $v$, all pairs of outer vertices are bad in $G - v$. The distinction between cases (a) and (b) is needed in the following definitions. A suitable cell $(G,a,b,c,d)$ is a \emph{$K_1$-cell} if by deleting any inner vertex of $G$ we obtain a graph in which at least one of the bad pairs stated in Properties 1.2, 1.4, 1.5(a) becomes good. A suitable cell $(G,a,b,c,d)$ is a \emph{$K_2$-cell} if all of the following properties hold.

\smallskip

\noindent 2.1. By deleting any two neighbouring inner vertices of $G$ we obtain a graph in which at least one of the bad pairs stated in Properties 1.2, 1.4, 1.5(a) becomes good.\\[1mm]
2.2. If $x \in N(a)$, then $(b,c)$ is good in $G-a-x$.\\[1mm]
2.3. If $x \in N(d)$, then $(b,c)$ is good in $G-d-x$.\\[1mm]
2.4. If $x \in N(b)$, then $(a,d)$ is good in $G-b-x$.\\[1mm]
2.5. If $x \in N(c)$, then $(a,d)$ is good in $G-c-x$.

\smallskip

It is far from obvious whether $K_1$- or $K_2$-cells actually exist; we shall later see examples of $K_1$- and $K_2$-cells, and even cells that are both $K_1$- and $K_2$-cells. On the other hand, it is easy to verify that the underlying graph of a suitable cell must have at least five vertices. %(THIS SUFFICES HERE, BUT SHOULD PROVE MORE - CAN'T BE HARD). \Carol{As you write we don't need more here, so I'd just leave this it at 5 here.}

Let $k \ge 3$ be a fixed but arbitrary integer. For $i=0, \ldots, k-1$, 
consider pairwise disjoint cells $H_i=(G_i,a_i,b_i,c_i,d_i)$. Henceforth, 
indices are taken mod~$k$. For a fixed but arbitrary $j \in\{1, 2\}$, if 
$H_i$ is a $K_j$-cell for every $i$, then the graph $\Gamma_j^k$ is defined 
as 
$\bigcup_i G_i$ in which we identify $b_i$ with $a_{i+1}$ and $c_i$ with 
$d_{i+1}$, for $i=0, \ldots, k-1$. Note that if $H_i$ admits a planar 
embedding 
in which $a,b,c,d$ occur in the same facial cycle and in this order, for 
$i=0, 
\ldots, k-1$, then the aforementioned identifications can be performed such 
that $\Gamma_j^k$ is planar, as well.
%\Carol{Someone double-check this sentence carefully}. \Gabor{Seems OK.} 
%\Carol{I haven't said anything about the connectivity here because (a) it's 
%cumbersome and (b) we don't really use it at this point.} \Gabor{Agreed.}
Put $V_i := V(G_i)$, $E_i := E(G_i)$, and $P_i := \{ a_i, b_i, c_i, d_i \}$; 
vertices in $P_i$ are \emph{outer vertices} (of $G_i$), while vertices in 
$V_i 
\setminus P_i$ are \emph{inner vertices} (of $G_i$).
%\Jan{The operation is of course pretty straightforward, but I think it can't hurt to include a figure with a sketch of $\Gamma_j^k$.} \Carol{I think it's sufficiently clear without a figure (this will already be a long, figure-laden paper).} %Jan: OK, fine with me.

\begin{theorem}\label{thm:gamma}
For every odd integer $k \ge 3$ the graph $\Gamma_1^k$ is hypohamiltonian and the graph $\Gamma_2^k$ is $K_2$-hypohamiltonian.
\end{theorem}
%\Jarne{I changed the structure of the original proof back to how it was as 
%Carol and G\'abor preferred it.}
\begin{proof}
We will first show that both $\Gamma_1^k$ and $\Gamma_2^k$ are 
non-hamiltonian. 
Let us assume that $\Gamma_1^k$ has a hamiltonian cycle ${\frak h}$---the proof will be the same for $\Gamma_2^k$, since at this point we shall only use properties of suitable cells---and let us consider ${\frak h}_i := {\frak h}[V_i]$. All vertices of ${\frak h}_i$ have degree at most~2 and the inner vertices (of $G_i$) must have degree~2 in ${\frak h}_i$, while at least two of the outer vertices have degree~1. It is also straightforward to see that each component of ${\frak h}_i$ contains an outer vertex. Therefore, exactly one of the following cases must hold.

\smallskip \noindent Case 1. Two outer vertices have degree 1 and the other two 
outer vertices have degree 0 in ${\frak h}_i$. Then ${\frak h}_i$ consists of a 
path ${\frak p}_i$ between two outer vertices and two isolated (outer) vertices.

\smallskip \noindent Case 2. Two outer vertices have degree 1, the others have degree 0 and 2, respectively. Then ${\frak h}_i$ consists of a path ${\frak p}_i$ between two outer vertices and an isolated (outer) vertex.

\smallskip \noindent Case 3. Two outer vertices have degree 1, the others have degree 2. Then ${\frak h}_i$ is a path between two outer vertices.

\smallskip \noindent Case 4. All four outer vertices have degree 1. Then ${\frak h}_i$ is the disjoint union of two paths, both between two outer vertices.

\smallskip

\noindent We now treat these four cases. Case 2 is easy to handle: it cannot occur due to Property~1.5.
%only if the endvertices of $C'_i$ are either $a_i$ and $d_i$ or $b_i$ and $c_i$, because of Property 5 of V-cells. W.l.o.g let us assume that the endvertices of $C'_i$ are $b_i$ and $c_i$. Let $j$ be the smallest positive integer, such that $C_{i+j}$ does not correspond to Case 4 (such a $j$ must clearly exist, since either $a_i$ or $d_i$ is an isolated vertex in $C_i$). Then $b_{i+j-1}$ and $c_{i+j-1}$ has degree 1 in $C_{i+j-1}$ and therefore $a_{i+j}$ and $d_{i+j}$ have degree 1 in $C_{i+j}$ and there is a path between them in $C_{i+j}$. Thus there exists a cycle in $C_i \cup C_{i+1} \cup \ldots \cup C_{i+j}$ that does not contain either $a_i$ or $d_i$, a contradiction showing that Case 2 is actually impossible.

\smallskip

\noindent In Case~4, ${\frak h}_i$ must be the union of an $a_ib_i$- and a $c_id_i$-path, because of Property~1.4. Now we show that Case 4 is also impossible. Assume to the contrary that Case~4 holds for some ${\frak h}_i$. Then there exists an index $j$ such that Case~4 holds for ${\frak h}_j$, but not for ${\frak h}_{j+1}$. Indeed, if such a $j$ did not exist then ${\frak h}$ would not be connected, but the union of two cycles, since every ${\frak h}_i$ is the union of an $a_ib_i$- and a $c_id_i$-path. For ${\frak h}_{j+1}$ either Case~1 or Case~3 must hold and the degree 1 vertices of ${\frak h}_{j+1}$ are $a_{j+1}$ and $d_{j+1}$. Case~3 now cannot hold, because then $(a_{j+1},d_{j+1})$ would be a good pair in $G_{j+1}$, contradicting Property~1.2. Thus Case~1 holds, that is $b_{j+1}$ and $c_{j+1}$ are isolated vertices in ${\frak h}_{j+1}$. Then both $a_{j+2}$ and $d_{j+2}$ have degree 2 in ${\frak h}_{j+2}$, which is also impossible, since then ${\frak h}_{j+2}$ would be a hamiltonian $b_{j+2}c_{j+2}$-path in $G_{j+2}$, contradicting Property~1.2 (${\frak h}_{j+2}$ exists, since $k \ge 3$).

Now we prove that if ${\frak h}_i$ corresponds to Case~1 then ${\frak h}_{i+1}$ corresponds to Case~3 and vice versa. Let us assume first that ${\frak h}_i$ corresponds to Case~1. The end-vertices of ${\frak p}_i$ cannot be $a_i$ and $d_i$, as otherwise $(b_{j+1},c_{j+1})$ would be a good pair in $G_{j+1}$, contradicting Property~1.2. Similarly, the end-vertices of ${\frak p}_i$ cannot be $b_i$ and $c_i$, either. Thus, either $b_i$ or $c_i$ must be an isolated vertex of ${\frak h}_i$, so either $a_{i+1}= b_i$ or $d_{i+1} = c_i$ has degree~2 in ${\frak h}_{i+1}$. Therefore ${\frak h}_{i+1}$ cannot correspond to Case~1 and we have already seen that it cannot correspond to Case~2 or~4. Now let us assume that ${\frak h}_i$ corresponds to Case~3. Then the end-vertices of ${\frak h}_i$ cannot be $a_i$ and $d_i$, and they cannot be $b_i$ and $c_i$, by Property~1.2. Thus, either $b_i$ or $c_i$ has degree~2 in ${\frak h}_i$, whence, either $a_{i+1} = b_i$ or $d_{i+1} = c_i$ are isolated vertices in ${\frak h}_{i+1}$. Therefore ${\frak h}_{i+1}$ cannot correspond to Case~3 and we have already seen that it cannot correspond to Case~2 or~4.

We have proven that Cases 1 and 3 alternate among the ${\frak h}_i$'s, thus the 
number $k$ of ${\frak h}_i$'s should be even. Since this is not the case, we 
have finished the proof of the non-hamiltonicity of $\Gamma_1^k$ (and also 
$\Gamma_2^k$). 
%This completes the proof of Claim~\ref{claim:nonham}.
%\end{claimproof}

\medskip
%\begin{claim}\label{claim:hypoham}
%For every odd integer $k \ge 3$, $\Gamma_1^k$ is hypohamiltonian.
%\end{claim}
%\begin{claimproof}{\ref{claim:hypoham}}
We proceed by proving that $\Gamma_1^k$ is hypohamiltonian.
Since $\Gamma_1^k$ is non-hamiltonian, it remains to show that for every 
vertex 
$x$, the graph $\Gamma_1^k - x$ is hamiltonian.
Let $x$ be an arbitrary vertex of $\Gamma_1^k$ and let us assume without loss of generality that $x \in V_0$. Suppose first that $x$ is an inner vertex of $G_0$. By the definition of $K_1$-cells, for $G_0 - x$ one of the pairs appearing in Properties~1.2, 1.4, 1.5(a) must be good.

\smallskip

\noindent Case 1. One of the pairs appearing in Property~1.2 becomes good after deleting $x$ from $G_0$. We may assume without loss of generality that $(b_0,c_0)$ is good in $G_0 - x$, i.e.\ there exists a hamiltonian $b_0c_0$-path ${\frak p}_0$ in $G_0 - x$. For $i = 1, \ldots, k - 2$ let ${\frak u}_i$ be the union of
%the two paths in $G_i$ guaranteed by Property~1.3; these paths are either an
the $a_ib_i$- and $c_id_i$-paths in $G_i$ guaranteed by Property~1.3
. Let furthermore ${\frak p}_{k - 1}$ be the hamiltonian $a_{k-1}d_{k-1}$-path of $G_{k-1} - b_{k-1} - c_{k-1}$ guaranteed by Property~1.6. Then ${\frak p}_0 \cup {\frak u}_1 \cup \ldots \cup {\frak u}_{k-2} \cup {\frak p}_{k-1}$ is a hamiltonian cycle of $\Gamma_1^k - x$.

\smallskip \noindent Case 2. One of the pair of pairs $((a_0,d_0),(b_0,c_0))$ 
and $((a_0,c_0),(b_0,d_0))$ appearing in Property~1.4 becomes good after 
deleting $x$ from $G_0$, i.e.\ either there exists an $a_0d_0$-path $P$ and a 
$b_0c_0$-path $Q$ which together partition $V_0 - x$ or there exists an 
$a_0c_0$-path $P$ and a $b_0d_0$-path $Q$ which together partition $V_0 - x$. 
For both cases let us define ${\frak u}_i$ for $i = 1, \ldots, k - 1$ as the 
union of the $a_ib_i$- and $c_id_i$-paths in $G_i$ guaranteed by Property~1.3. 
Then $P \cup Q \cup {\frak u}_1 \cup \ldots \cup {\frak u}_{k-1}$ is a 
hamiltonian cycle of $\Gamma_1^k - x$.

\smallskip \noindent Case 3. One of the pairs appearing in Property~1.5(a) becomes good after deleting $x$ from $G_0$. We may assume without loss of generality that $(a_0,b_0)$ or $(a_0,c_0)$ is good in $G_0 - d_0 - x$.

\noindent Case 3a. There exists a hamiltonian $a_0c_0$-path ${\frak p}_0$ in 
$G_0 - d_0 - x$. For all odd $j = 1, \ldots, k-2$, let ${\frak p}_j$ be a 
hamiltonian $b_jd_j$-path of $G_j - a_j - c_j$, guaranteed by Property~1.6.
%(LEHET, HOGY KEVESEBB IS ELEG a 6-NAL??????  ESETLEG AZ 1-NEL (IS): MOST AZ 1-BOL HASZNALJUK AZ A-C-t, DE NEM FELTETLEN KELL, LEHETNE A-B ES C-D VALTAKOZVA HELYETTE)
If $k \ge 5$, for all even $j = 2, \ldots, k - 3$ let ${\frak p}_j$ be a 
hamiltonian $a_jc_j$-path of $G_j$, guaranteed by Property~1.1. Finally, let 
${\frak p}_{k-1}$ be a hamiltonian $a_{k-1}b_{k-1}$-path of $G_{k-1}$, again 
guaranteed by Property~1.1. Now ${\frak p}_0 \cup \ldots \cup {\frak p}_{k-1}$ 
is a hamiltonian cycle of $\Gamma_1^k - x$.

\noindent Case 3b. There exists a hamiltonian $a_0b_0$-path ${\frak p}_0$ in 
$G_0 - d_0 - x$. For all odd $j = 1, \ldots, k - 2$, let ${\frak p}_j$ be a 
hamiltonian $a_jc_j$-path of $G_j - b_j - d_j$, guaranteed by Property~1.6. For 
all even $j =2, \ldots, k - 1 $, let ${\frak p}_j$ be a hamiltonian 
$b_jd_j$-path of $G_j$, guaranteed by Property~1.1.
%Finally, let ${\frak p}_{k-1}$ be a hamiltonian $b_{k-1}d_{k-1}$-path of $G_{k-1}$, again guaranteed by Property~1.1. - included in the previous sequence
Now ${\frak p}_0 \cup \ldots \cup {\frak p}_{k-1}$ is a hamiltonian cycle of $\Gamma_1^k - x$.
%(ITT VISZONT KELL A B-D UT!!!!!!!!!!)

\medskip

Let us suppose now that $x$ is an outer vertex of $G_0$; without loss 
of generality assume $x=a_0$. There is a hamiltonian $b_0d_0$-path ${\frak 
p}_0$ in $G_0 - a_0 - c_0$ by Property~1.6. Let ${\frak p}_1$ be a hamiltonian 
$a_1b_1$-path of $G_1$ (which exists by Property~1.1) and for all even $j = 2, 
\ldots, k - 1$ let ${\frak p}_j$ be a hamiltonian $a_jc_j$-path of $G_j - b_j - 
d_j$, guaranteed again by Property~1.6. If $k \ge 5$, for all odd $j = 3, 
\ldots, k-2 $, let furthermore ${\frak p}_j$ be a hamiltonian $b_jd_j$-path 
of $G_j$, guaranteed again by Property~1.1. Then ${\frak p}_0 \cup \ldots \cup 
{\frak p}_{k-1}$ is a hamiltonian cycle of $\Gamma_1^k - a_0$, finishing the 
proof of the hypohamiltonicity of $\Gamma_1^k$.
%This completes the proof of Claim~\ref{claim:hypoham}.
%\end{claimproof}

%\begin{claim}\label{claim:k2ham}
%For every odd integer $k \ge 3$, $\Gamma_2^k$ is $K_2$-hypohamiltonian.
%\end{claim}
%\begin{claimproof}{\ref{claim:k2ham}}
%Now we prove that $\Gamma_2^k$ is $K_2$-hypohamiltonian.
\medskip
Since $\Gamma_2^k$ is non-hamiltonian, we still have to show that for any 
neighbouring vertices $x$ and $y$, the graph $\Gamma_2^k - x - y$ is 
hamiltonian. By the construction of $\Gamma_2^k$ there exists a $G_i$ which 
contains both $x$ and $y$, so we may assume without loss of generality that 
$x,y \in V_0$.

If both $x$ and $y$ are inner vertices of $G_0$, then by Property~2.1, one of the pairs appearing in Properties~1.2, 1.4, 1.5(a) is good in $G_0 - x - y$ and the proof of the hamiltonicity of $\Gamma_2^k - x - y$ is the same as the proof of the hamiltonicity of $\Gamma_1^k - x$ (where $x$ is an inner vertex of $G_0$) we have just seen.

Let us suppose now that $y$ is an outer vertex of $G_0$ and $x\in V_0$ is an 
arbitrary, but fixed neighbour of $y$. Without loss of generality we may 
assume 
$y = a_0$. By Property~2.2 there exists a hamiltonian $b_0c_0$-path ${\frak 
p}_0$ in $G_0 - a_0 - x$. Now the proof is essentially the same as for Case~1 
of the proof of the hypohamiltonicity of $\Gamma_1^k$, but we include it 
here, 
since the setting is a slightly different one.

For $i = 1, \ldots, k-2$ let ${\frak u}_i$ be the union of the $a_ib_i$- and  
$c_id_i$-paths of $G_i$ guaranteed by Property~1.3. Let furthermore ${\frak 
p}_{k-1}$ be a hamiltonian $a_{k-1}d_{k-1}$-path of $G_{k-1} - b_{k-1} - 
c_{k-1}$ whose existence is guaranteed by Property~1.6. Now ${\frak p}_0 \cup 
{\frak u}_1 \cup \ldots \cup {\frak u}_{k-2} \cup {\frak p}_{k-1}$ is a 
hamiltonian cycle of $\Gamma_2^k - a_0 - x$, finishing the proof of the 
$K_2$-hypohamiltonicity of $\Gamma_1^k$. %\hfill $\Box$
%\smallskip \noindent Case 2. ${\frak p}_0$ is a hamiltonian path of $G_0 - d_0 - a_0 - x$. For $i = 1, \ldots, k-2$ let ${\frak u}_i$ be the union of the two paths in $G_i$ guaranteed by Property~1.3, again. These paths are either an $a_ib_i$- and a $c_id_i$-path or an $a_ic_i$- and a $b_id_i$-path, as before. Let finally ${\frak p}_{k-1}$ be a hamiltonian $a_{k-1}d_{k-1}$-path of $G_{k-1} - b_{k-1} - c_{k-1}$ given by Property~1.6. Then ${\frak p}_0 \cup {\frak u}_1 \cup \ldots \cup {\frak u}_{k-2} \cup {\frak p}_{k-1}$ is a hamiltonian cycle of $\Gamma_2^k - a_0 - x$. \emph{CZ: There is an issue here, since we are not visiting $d_0 = c_{k-1}$, so the cycle omits three vertices in $\Gamma_2^k$. We would need a hamiltonian $a_{k-1}d_{k-1}$-path in $G_{k-1} - b_{k-1}$, but this is explicitly forbidden by Property~1.5(b).}
%\emph{CZ: Can both removed vertices be outer vertices, i.e.\ can outer vertices be adjacent? (We seem to assume in the proofs that they cannot.)}
%This completes the proof of Claim~\ref{claim:k2ham}.
%\end{claimproof}

%Together Claims~\ref{claim:nonham}--\ref{claim:k2ham} prove 
%Theorem~\ref{thm:gamma}.
\end{proof}

%\Jan{TODO: As Carol suggested in his email of 4 March it would be useful to ``mention in Theorem 1 something about when Gamma is planar. Along the same lines, it's worth mentioning how the connectivity of the cell affects the connectivity of Gamma''.}

%\subsection{\boldmath$K_1$- and \boldmath$K_2$-cells}

We will now present a suitable cell that is both a $K_1$- and $K_2$-cell: the 
quintuple $(J_{18},a,b,c,d)$ of Figure~\ref{fig:j18}. It is obtained by 
removing two adjacent vertices from the dodecahedron. This particular graph was 
already used by Faulkner and Younger in the context of determining hamiltonian 
properties of planar 3-connected cubic graphs~\cite{FY74}. Although they 
mention two of the properties used below (more precisely, points 1 and 5), no 
proof is given; we therefore give all details here. %First we show that 
%$(J_{18},a,b,c,d)$ is indeed a suitable cell.

\begin{figure}[!htb]
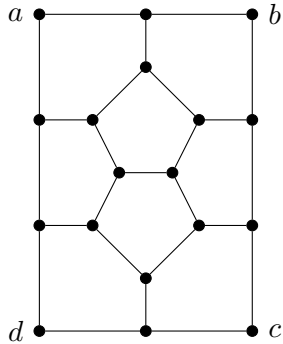

\begin{center}
\tikz[fo/.style={draw, fill=black, circle, minimum size={0.16cm}, inner sep=0cm, font=\bf, align=center, scale=0.88},
scale=0.7]
{
\node [fo, label=left:$d$] (1) at (0,0) {};
\node [fo] (2) at (2,0) {};
\node [fo, label=right:$c$] (3) at (4,0) {};
\node [fo] (4) at (0,2) {};
\node [fo] (5) at (0,4) {};
\node [fo, label=left:$a$] (6) at (0,6) {};
\node [fo] (7) at (4,2) {};
\node [fo] (8) at (4,4) {};
\node [fo, label=right:$b$] (9) at (4,6) {};
\node [fo] (10) at (2,6) {};
\node [fo] (11) at (2,1) {};
\node [fo] (12) at (2,5) {};
\node [fo] (13) at (1,2) {};
\node [fo] (14) at (1,4) {};
\node [fo] (15) at (3,2) {};
\node [fo] (16) at (3,4) {};
\node [fo] (17) at (1.5,3) {};
\node [fo] (18) at (2.5,3) {};

\draw

(1) edge node  {} (2)
    edge node {} (4)

(2) edge node {} (3)
    edge node {} (11)
		
(3) edge node {} (7)

(4) edge node {} (5)
    edge node {} (13)

(5) edge node {} (6)
    edge node {} (14)
	
(6) edge node {} (10)

(7) edge node {} (8)
    edge node {} (15)
		
(8) edge node {} (9)
    edge node {} (16)
		
(9) edge node {} (10)

(10) edge node {} (12)
     		
(11) edge node {} (13)
     edge node {} (15)

(12) edge node {} (14)
     edge node {} (16)
		
(13) edge node {} (17)

(14) edge node {} (17)

(15) edge node {} (18)

(16) edge node {} (18)

(17) edge node {} (18)
;}

\caption{The graph $J_{18}$.}
\label{fig:j18}
\end{center}
\end{figure}

%\bigskip

%\noindent \textbf{Lemma 1.} \label{lem:egy} \emph{The quintuple $(J_{18},a,b,c,d)$ of Figure~1 is a suitable cell.}
\filbreak
\begin{lemma}\label{lem:egy}
The quintuple $(J_{18},a,b,c,d)$ of Figure~\ref{fig:j18} is a suitable cell.
\end{lemma}
\begin{proof}
%\noindent \textit{Proof.}
We have to verify Properties 1.1--1.6. Some of the upcoming arguments are most 
easily verified by using figures, which we have provided in the Appendix; all 
figures referenced in this proof are to be found there. Ignoring symmetric 
cases, for Property~1.1 it suffices to provide a hamiltonian $ab$-path and a 
hamiltonian $ac$-path in $J_{18}$, see Figure~\ref{fig:J18spanningpaths} in the 
Appendix. For Property~1.3 it is enough to present an $ab$-path and a $cd$-path 
whose vertex sets partition $V(J_{18})$; an easy task, verified in 
Figure~\ref{fig:J18spanningpaths}. Again by symmetry, for Property~1.6 it 
suffices to give a hamiltonian $bc$-path in $J_{18} - a - d$ and a hamiltonian 
$bd$-path in $J_{18} - a - c$; these are depicted in 
Figure~\ref{fig:J18bcpath} in the Appendix. For Properties~1.2, 1.4, and 1.5, 
we shall use the 
planarity of $J_{18}$ and Grinberg's Criterion~\cite{Gr68} stating that if a 
plane graph has a hamiltonian cycle such that there are $f_i$ faces of size $i$ 
in the exterior of the cycle and $f'_i$ faces of size $i$ in the interior of 
the cycle, then
\[ \sigma := \sum_i (i-2)(f_i-f'_i) = 0.\]
Let us call a face of a plane graph an \emph{$i$-face} if dividing its size by 3, we obtain the remainder~$i$ (thus $i \in \{ 0, 1, 2\}$). Now we are ready to verify Properties~1.2, 1.4, and 1.5. By symmetry, for these it suffices to show that
\begin{enumerate}
\item there is no hamiltonian $ad$-path in $J_{18}$ (Property~1.2),
\item there is no hamiltonian $bc$-path in $J_{18}-a$ (Property~1.5),
\item there is no hamiltonian $bd$-path in $J_{18}-a$ (Property~1.5),
\item there is no hamiltonian $cd$-path in $J_{18}-a$ (Property~1.5),
\item there exist no $ad$- and $bc$-path partitioning $V(J_{18})$ (Property~1.4), and
\item there exist no $ac$- and $bd$-path partitioning $V(J_{18})$ (Property~1.4).
\end{enumerate}

\noindent To prove point 1, let us assume to the contrary that there exists a hamiltonian $ad$-path ${\frak p}$ in $J_{18}$. Then we have a hamiltonian cycle ${\frak h} := {\frak p} + ad$ in the plane graph $J_{18} + ad$, which is a so-called Grinbergian graph (see e.g.\ \cite{JMOPZ17}, \cite{Wi18}), because all but one of its faces are 2-faces. These graphs are easily seen to be non-hamiltonian: let $f_i$ and $f_i'$ be as in Grinberg's theorem, using ${\frak h}$ as the hamiltonian cycle; now $\sigma$ is not divisible by 3, thus it cannot be $0$, a contradiction.

\smallskip \noindent Now we prove points 2, 3, and 4. The plane graph $J_{18} - a$ has only 2-faces and by adding any (one) of the edges $bc, bd, cd$ we obtain a graph with one 0-face and one 1-face that are on the opposite sides of the new edge, while all the other faces are still 2-faces. If there is a hamiltonian $bc$-path ${\frak p}'$ in $J_{18} - a$, then ${\frak h}' := {\frak p}' + bc$ is a hamiltonian cycle in $(J_{18} - a) + bc$. Let $f_i$ and $f_i'$ be as in Grinberg's theorem again, using ${\frak h}'$ as the hamiltonian cycle. The 0-face and the 1-face must be on different sides of ${\frak h}'$, since they both contain the edge $bc \in E({\frak h}')$, thus $\sigma$ is again indivisible by $3$, a contradiction. The proof is the same for points 3 and 4.

\smallskip \noindent Let us now prove point 5. Assume to the contrary that 
there exists the disjoint union ${\frak u}$ of an $ad$- and a $bc$-path 
spanning $J_{18}$. Then ${\frak h}'' := {\frak u} + ab + cd$ is a hamiltonian 
cycle of the plane graph $J_{18} + ab + cd$, which has two 0-faces and all 
other faces are 2-faces, one of which is an octagon. Let $f_i$ and $f_i'$ be as 
in Grinberg's theorem, using ${\frak h}''$ as the hamiltonian cycle. Then both 
0-faces are on a different side of ${\frak h}''$ than the octagon, since the 
octagon contains both $ab$ and $cd$ and these edges lie in ${\frak h}''$. So 
the 0-faces are on the same side of ${\frak h}''$, whence $\sigma$ is again 
indivisible by $3$, a contradiction.

\smallskip \noindent Finally, we prove point 6.  It is easy to see that the graph $J'_{18}$ obtained by adding a new vertex $v$ and the edges $ab, bc, cd, da, va, vb, vc, vd$ to $J_{18}$ is also a planar graph. Now assume to the contrary that there exists the disjoint union of an $ac$- and a $bd$-path in $J_{18}$. Then $J'_{18}$ would contain a subdivision of $K_5$ (consider the vertices $a,b,c,d,v$), a contradiction by Kuratowski's Theorem.% \hfill $\Box$
\end{proof}

%\bigskip

%\noindent \textbf{Theorem 2.} \emph{The quintuple $(J_{18},a,b,c,d)$ of Figure~\ref{fig:j18} is both a $K_1$- and a $K_2$-cell.}

\begin{lemma}\label{lem:j18}
The quintuple $(J_{18},a,b,c,d)$ of Figure~\ref{fig:j18} is both a $K_1$- and a $K_2$-cell.
\end{lemma}

%\bigskip
\begin{proof}
%\noindent \emph{Proof.}
%We have verified with a computer program that the quintuple $(J_{18},a,b,c,d)$ is both a $K_1$-cell and a $K_2$-cell.
%\Carol{Gabor, if I remember correctly, these two TODO's should be done computationally, right?} %\hfill $\Box$
We implemented a straightforward computer program which we used to verify
that the quintuple $(J_{18},a,b,c,d)$ indeed has the desired properties and
thus is both a $K_1$-cell and a $K_2$-cell. The source code of this program can
be found on GitHub \cite{Re22}
% from \url{https://github.com/TODO} \Jan{Once the code is cleaned,
%we should upload it on a public github repository and foresee a short manual
%for it.}
and in Section~\ref{app:j18} of the Appendix we included drawings so
these results can also be verified by hand.

In order to show that $(J_{18},a,b,c,d)$ is a $K_1$-cell, by Lemma~\ref{lem:egy} it suffices to prove that if by deleting any inner vertex of
$J_{18}$ we obtain a graph in which at least one of the bad pairs of Properties~1.2, 1.4, 1.5(a) becomes good. Figure~\ref{fig:j18k1goodpairs} in
the Appendix shows all inner vertex-deleted subgraphs of $J_{18}$ up to symmetry and highlights a new good pair. Below each figure we mention which previously bad pair became good in this subgraph and to which property of suitable cells the bad pair belonged.

We now prove that $(J_{18},a,b,c,d)$ is a $K_2$-cell. It suffices to show that Properties~2.1--2.5 hold. Figure~\ref{fig:j182.1-2.5} illustrates the proof of these properties. This figure shows all subgraphs of $J_{18}$ up to symmetry obtained by deleting a pair of adjacent inner vertices and highlights a new good pair, i.e.\ all subgraphs relating to Property~2.1. If this new good pair belonged to Property~1.5, we show the subgraph in which we also deleted the appropriate outer vertex. The figure also shows the proof of Property~2.2, i.e.\ all subgraphs in which we deleted $a$ and one of its neighbours, together with a hamiltonian path highlighting the new good pair. Properties~2.3--2.5 are symmetric.
\end{proof}

From Theorem~\ref{thm:gamma} and Lemma~\ref{lem:j18} we can immediately infer the following result.

\begin{corollary}\label{cor:infinite_fam}
There exists an infinite family ${\cal G}$ of non-hamiltonian planar 
$K_1$-hamiltonian $K_2$-hamiltonian graphs.
\end{corollary}

The smallest member of ${\cal G}$, which we will call $G_{48}$, is depicted in Figure~\ref{fig:48v}.

\begin{figure}[!htb]
\begin{center}
\begin{tikzpicture}[scale=0.05]
    \definecolor{marked}{rgb}{0.25,0.5,0.25}
    \node [circle,fill,scale=0.25] (48) at (34.827472,50.053745) {};
    \node [circle,fill,scale=0.25] (47) at (42.771147,50.064494) {};
    \node [circle,fill,scale=0.25] (46) at (29.710844,39.949478) {};
    \node [circle,fill,scale=0.25] (45) at (22.433622,44.001934) {};
    \node [circle,fill,scale=0.25] (44) at (22.433622,56.008813) {};
    \node [circle,fill,scale=0.25] (43) at (29.700094,60.125764) {};
    \node [circle,fill,scale=0.25] (42) at (51.445768,72.670106) {};
    \node [circle,fill,scale=0.25] (41) at (41.911209,64.156722) {};
    \node [circle,fill,scale=0.25] (40) at (41.900460,35.983016) {};
    \node [circle,fill,scale=0.25] (39) at (51.424269,27.372890) {};
    \node [circle,fill,scale=0.25] (38) at (40.976027,24.019133) {};
    \node [circle,fill,scale=0.25] (37) at (32.000427,29.178759) {};
    \node [circle,fill,scale=0.25] (36) at (17.134256,37.541652) {};
    \node [circle,fill,scale=0.25] (35) at (8.653121,49.967751) {};
    \node [circle,fill,scale=0.25] (34) at (17.102008,62.404599) {};
    \node [circle,fill,scale=0.25] (33) at (32.000427,70.896483) {};
    \node [circle,fill,scale=0.25] (32) at (40.943779,76.109855) {};
    \node [circle,fill,scale=0.25] (31) at (58.626247,76.916047) {};
    \node [circle,fill,scale=0.25] (30) at (57.669567,63.189292) {};
    \node [circle,fill,scale=0.25] (29) at (53.735352,56.331290) {};
    \node [circle,fill,scale=0.25] (28) at (53.735352,43.711704) {};
    \node [circle,fill,scale=0.25] (27) at (57.648068,36.832204) {};
    \node [circle,fill,scale=0.25] (26) at (58.583250,23.105449) {};
    \node [circle,fill,scale=0.25] (25) at (55.659462,15.312265) {};
    \node [circle,fill,scale=0.25] (24) at (38.697192,12.614208) {};
    \node [circle,fill,scale=0.25] (23) at (23.282811,21.514564) {};
    \node [circle,fill,scale=0.25] (22) at (0.000000,49.999999) {};
    \node [circle,fill,scale=0.25] (21) at (23.272060,78.463936) {};
    \node [circle,fill,scale=0.25] (20) at (38.772436,87.439535) {};
    \node [circle,fill,scale=0.25] (19) at (55.734707,84.709231) {};
    \node [circle,fill,scale=0.25] (18) at (69.009995,70.917983) {};
    \node [circle,fill,scale=0.25] (17) at (68.934749,62.576587) {};
    \node [circle,fill,scale=0.25] (16) at (66.440930,49.999999) {};
    \node [circle,fill,scale=0.25] (15) at (68.924001,37.434160) {};
    \node [circle,fill,scale=0.25] (14) at (68.988496,29.092765) {};
    \node [circle,fill,scale=0.25] (13) at (70.643876,14.151348) {};
    \node [circle,fill,scale=0.25] (12) at (25.002684,6.702140) {};
    \node [circle,fill,scale=0.25] (11) at (25.002684,93.297859) {};
    \node [circle,fill,scale=0.25] (10) at (70.751368,85.805652) {};
    \node [circle,fill,scale=0.25] (9) at (77.233147,72.250886) {};
    \node [circle,fill,scale=0.25] (8) at (77.157903,55.191872) {};
    \node [circle,fill,scale=0.25] (7) at (77.136406,44.786627) {};
    \node [circle,fill,scale=0.25] (6) at (77.190151,27.695366) {};
    \node [circle,fill,scale=0.25] (5) at (74.997310,6.702140) {};
    \node [circle,fill,scale=0.25] (4) at (74.997310,93.297859) {};
    \node [circle,fill,scale=0.25] (3) at (88.046863,58.921852) {};
    \node [circle,fill,scale=0.25] (2) at (88.036115,41.045898) {};
    \node [circle,fill,scale=0.25] (1) at (99.999999,49.999999) {};
    \draw [black] (48) to (46);
    \draw [black] (48) to (43);
    \draw [black] (48) to (47);
    \draw [black] (47) to (40);
    \draw [black] (47) to (41);
    \draw [black] (46) to (37);
    \draw [black] (46) to (45);
    \draw [black] (45) to (36);
    \draw [black] (45) to (44);
    \draw [black] (44) to (34);
    \draw [black] (44) to (43);
    \draw [black] (43) to (33);
    \draw [black] (42) to (32);
    \draw [black] (42) to (31);
    \draw [black] (42) to (30);
    \draw [black] (41) to (32);
    \draw [black] (41) to (29);
    \draw [black] (41) to (33);
    \draw [black] (40) to (37);
    \draw [black] (40) to (28);
    \draw [black] (40) to (38);
    \draw [black] (39) to (26);
    \draw [black] (39) to (38);
    \draw [black] (39) to (27);
    \draw [black] (38) to (24);
    \draw [black] (37) to (23);
    \draw [black] (36) to (35);
    \draw [black] (36) to (23);
    \draw [black] (35) to (22);
    \draw [black] (35) to (34);
    \draw [black] (34) to (21);
    \draw [black] (33) to (21);
    \draw [black] (32) to (20);
    \draw [black] (31) to (18);
    \draw [black] (31) to (19);
    \draw [black] (30) to (29);
    \draw [black] (30) to (17);
    \draw [black] (29) to (16);
    \draw [black] (28) to (16);
    \draw [black] (28) to (27);
    \draw [black] (27) to (15);
    \draw [black] (26) to (25);
    \draw [black] (26) to (14);
    \draw [black] (25) to (13);
    \draw [black] (25) to (24);
    \draw [black] (24) to (12);
    \draw [black] (23) to (12);
    \draw [black] (22) to (11);
    \draw [black] (22) to (12);
    \draw [black] (21) to (11);
    \draw [black] (20) to (11);
    \draw [black] (20) to (19);
    \draw [black] (19) to (10);
    \draw [black] (18) to (9);
    \draw [black] (18) to (17);
    \draw [black] (17) to (8);
    \draw [black] (16) to (7);
    \draw [black] (16) to (8);
    \draw [black] (15) to (14);
    \draw [black] (15) to (7);
    \draw [black] (14) to (6);
    \draw [black] (13) to (5);
    \draw [black] (13) to (6);
    \draw [black] (12) to (5);
    \draw [black] (11) to (4);
    \draw [black] (10) to (9);
    \draw [black] (10) to (4);
    \draw [black] (9) to (3);
    \draw [black] (8) to (3);
    \draw [black] (7) to (2);
    \draw [black] (6) to (2);
    \draw [black] (5) to (1);
    \draw [black] (4) to (1);
    \draw [black] (3) to (1);
    \draw [black] (2) to (1);
\end{tikzpicture}
\qquad
\begin{tikzpicture}[scale=0.05]
    \definecolor{marked}{rgb}{0.25,0.5,0.25}
    \node [circle,fill,scale=0.25] (48) at (17.876847,40.631206) {};
    \node [circle,fill,scale=0.25] (47) at (22.930919,23.857865) {};
    \node [circle,fill,scale=0.25] (46) at (9.766053,51.765613) {};
    \node [circle,fill,scale=0.25] (45) at (20.657689,67.038180) {};
    \node [circle,fill,scale=0.25] (44) at (27.002867,63.264179) {};
    \node [circle,fill,scale=0.25] (43) at (24.994481,50.551753) {};
    \node [circle,fill,scale=0.25] (42) at (41.425733,42.827189) {};
    \node [circle,fill,scale=0.25] (41) at (30.103728,33.028027) {};
    \node [circle,fill,scale=0.25] (40) at (25.005516,6.698299) {};
    \node [circle,fill,scale=0.25] (39) at (71.595673,16.011915) {};
    \node [circle,fill,scale=0.25] (38) at (75.005515,6.698299) {};
    \node [circle,fill,scale=0.25] (37) at (0.000000,49.999999) {};
    \node [circle,fill,scale=0.25] (36) at (28.492604,84.054292) {};
    \node [circle,fill,scale=0.25] (35) at (42.098873,82.487310) {};
    \node [circle,fill,scale=0.25] (34) at (36.989625,71.264621) {};
    \node [circle,fill,scale=0.25] (33) at (32.222466,48.554402) {};
    \node [circle,fill,scale=0.25] (32) at (38.501433,46.413594) {};
    \node [circle,fill,scale=0.25] (31) at (48.101963,50.949016) {};
    \node [circle,fill,scale=0.25] (30) at (40.851907,34.142572) {};
    \node [circle,fill,scale=0.25] (29) at (35.985432,25.656586) {};
    \node [circle,fill,scale=0.25] (28) at (40.951223,13.529021) {};
    \node [circle,fill,scale=0.25] (27) at (57.978370,17.534758) {};
    \node [circle,fill,scale=0.25] (26) at (79.430587,33.039063) {};
    \node [circle,fill,scale=0.25] (25) at (90.244977,48.344735) {};
    \node [circle,fill,scale=0.25] (24) at (99.999999,49.999999) {};
    \node [circle,fill,scale=0.25] (23) at (25.005516,93.301700) {};
    \node [circle,fill,scale=0.25] (22) at (59.126019,86.448907) {};
    \node [circle,fill,scale=0.25] (21) at (42.220259,65.912601) {};
    \node [circle,fill,scale=0.25] (20) at (47.186050,61.531670) {};
    \node [circle,fill,scale=0.25] (19) at (51.776648,60.825423) {};
    \node [circle,fill,scale=0.25] (18) at (51.754578,48.885454) {};
    \node [circle,fill,scale=0.25] (17) at (48.090928,39.020083) {};
    \node [circle,fill,scale=0.25] (16) at (45.376295,24.387551) {};
    \node [circle,fill,scale=0.25] (15) at (63.021407,28.735377) {};
    \node [circle,fill,scale=0.25] (14) at (73.008164,36.768923) {};
    \node [circle,fill,scale=0.25] (13) at (82.101081,59.401896) {};
    \node [circle,fill,scale=0.25] (12) at (75.005515,93.301700) {};
    \node [circle,fill,scale=0.25] (11) at (54.469211,75.446921) {};
    \node [circle,fill,scale=0.25] (10) at (59.037738,65.747075) {};
    \node [circle,fill,scale=0.25] (9) at (58.452879,57.029352) {};
    \node [circle,fill,scale=0.25] (8) at (52.648421,38.280732) {};
    \node [circle,fill,scale=0.25] (7) at (57.702492,34.010150) {};
    \node [circle,fill,scale=0.25] (6) at (74.939307,49.459279) {};
    \node [circle,fill,scale=0.25] (5) at (77.002867,76.175237) {};
    \node [circle,fill,scale=0.25] (4) at (63.926285,74.244094) {};
    \node [circle,fill,scale=0.25] (3) at (61.388214,53.442947) {};
    \node [circle,fill,scale=0.25] (2) at (67.689251,51.390420) {};
    \node [circle,fill,scale=0.25] (1) at (69.796953,66.839549) {};
    \draw [black] (48) to (46);
    \draw [black] (48) to (43);
    \draw [black] (48) to (47);
    \draw [black] (47) to (40);
    \draw [black] (47) to (41);
    \draw [black] (46) to (37);
    \draw [black] (46) to (45);
    \draw [black] (45) to (36);
    \draw [black] (45) to (44);
    \draw [black] (44) to (34);
    \draw [black] (44) to (43);
    \draw [black] (43) to (33);
    \draw [black] (42) to (32);
    \draw [black] (42) to (31);
    \draw [black] (42) to (30);
    \draw [black] (41) to (32);
    \draw [black] (41) to (29);
    \draw [black] (41) to (33);
    \draw [black] (40) to (37);
    \draw [black] (40) to (28);
    \draw [black] (40) to (38);
    \draw [black] (39) to (26);
    \draw [black] (39) to (38);
    \draw [black] (39) to (27);
    \draw [black] (38) to (24);
    \draw [black] (37) to (23);
    \draw [black] (36) to (35);
    \draw [black] (36) to (23);
    \draw [black] (35) to (22);
    \draw [black] (35) to (34);
    \draw [black] (34) to (21);
    \draw [black] (33) to (21);
    \draw [black] (32) to (20);
    \draw [black] (31) to (18);
    \draw [black] (31) to (19);
    \draw [black] (30) to (29);
    \draw [black] (30) to (17);
    \draw [black] (29) to (16);
    \draw [black] (28) to (16);
    \draw [black] (28) to (27);
    \draw [black] (27) to (15);
    \draw [black] (26) to (25);
    \draw [black] (26) to (14);
    \draw [black] (25) to (13);
    \draw [black] (25) to (24);
    \draw [black] (24) to (12);
    \draw [black] (23) to (12);
    \draw [black] (22) to (11);
    \draw [black] (22) to (12);
    \draw [black] (21) to (11);
    \draw [black] (20) to (11);
    \draw [black] (20) to (19);
    \draw [black] (19) to (10);
    \draw [black] (18) to (9);
    \draw [black] (18) to (17);
    \draw [black] (17) to (8);
    \draw [black] (16) to (7);
    \draw [black] (16) to (8);
    \draw [black] (15) to (14);
    \draw [black] (15) to (7);
    \draw [black] (14) to (6);
    \draw [black] (13) to (5);
    \draw [black] (13) to (6);
    \draw [black] (12) to (5);
    \draw [black] (11) to (4);
    \draw [black] (10) to (9);
    \draw [black] (10) to (4);
    \draw [black] (9) to (3);
    \draw [black] (8) to (3);
    \draw [black] (7) to (2);
    \draw [black] (6) to (2);
    \draw [black] (5) to (1);
    \draw [black] (4) to (1);
    \draw [black] (3) to (1);
    \draw [black] (2) to (1);
\end{tikzpicture}

%\includegraphics[height=48mm]{48v2} \includegraphics[height=48mm]{48}\\
%Fig.~4: Two depictions of the planar hypohamiltonian $K_2$-hamiltonian graph $G_{48}$ on 48 vertices obtained from three copies of $J_{18}$.
\caption{Two depictions of the planar hypohamiltonian $K_2$-hamiltonian graph $G_{48}$ on 48 vertices obtained from three copies of $J_{18}$.}
\label{fig:48v}
\end{center}
\end{figure}
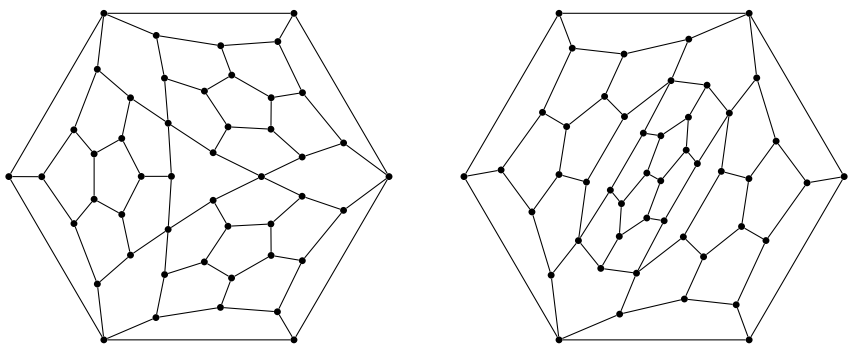

We provide a short alternative proof of Corollary~\ref{cor:infinite_fam} based only on the hypohamiltonicity and $K_2$-hypohamiltonicity of $G_{48}$ (which could be determined easily using a computer, without the framework laid out in Theorem~\ref{thm:gamma} and Lemma~\ref{lem:j18}). We require the following definition, a result of Thomassen, and a lemma from~\cite{Za21}. In a 2-connected non-hamiltonian graph $G$, we call ${\rm exc}(G) \subset V(G)$, which contains every vertex $v$ of $G$ such that $G - v$ is non-hamiltonian, the set of \emph{exceptional vertices} of $G$.

\begin{lemma}[Thomassen \cite{Th78}]\label{lem:thom}
Let $F_1$ and $F_2$ be disjoint $3$-fragments of hypohamiltonian graphs, not both trivial, with $F_i$ having attachments $X_i$. Then $(F_1,X_1) \: (F_2,X_2)$ is hypohamiltonian.
\end{lemma}

\begin{lemma}[\cite{Za21}]\label{lem:carol}
Let $G_1$ and $G_2$ be disjoint $K_2$-hypohamiltonian graphs. For $i\in \{1, 
2\}$, 
let $G_i$ contain a $3$-cut $X_i$ and $X_i$-fragments $F_i$ and $F'_i$ such 
that for each $x \in X_i$ there is a hamiltonian path in $F_i - x$ and in $F'_i 
- x$ between the two vertices of $X_i - x$. This is fulfilled e.g.\ when $X_i$ 
is non-trivial, or ${\rm exc}(G_i) \cap X_i = \emptyset$. Then, if both $F_1$ 
and $F_2$ are non-trivial, or both $F_i$ and $F'_{3-i}$ are trivial, $(F_1,X_1) 
\: (F_2,X_2)$ is $K_2$-hypohamiltonian.
\end{lemma}

\begin{proof}[Second proof of Corollary~\ref{cor:infinite_fam}.]
As $G_{48}$ is hypohamiltonian, ${\rm exc}(G_{48}) = \emptyset$. Consider a cubic vertex $v$ in $G_{48}$ and let $F$ and $F'$ denote two copies of the non-trivial $N(v)$-fragment of $G_{48}$ isomorphic to $G_{48} - v$. Let $X$ ($X'$) be the 3-cut in $G_{48}$ corresponding to $N(v)$ in $F$ ($F'$). By Lemma~\ref{lem:carol}, $G := (F,X) \: (F',X')$ is $K_2$-hypohamiltonian; here we make use of the fact that, since $G_{48}$ is hypohamiltonian, Lemma~\ref{lem:thom} implies that $G$ is hypohamiltonian so ${\rm exc}(G) = \emptyset$. By construction $G$ contains many cubic vertices, so we may iterate this procedure ad infinitum.
%\hfill $\Box$
\end{proof}

%\bigskip

%\noindent \emph{Third proof of Corollary~1.} $G_{48}$ is $K_2$-hypohamiltonian. It contains an extendable 5-cycle, see Fig.~A in the Appendix. Applying Lemma~3, we are done. \hfill $\Box$

%\Carol{We should discuss whether keeping this alternative proof is worthwhile. We can show that there are infinitely many $K_2$-hypohamiltonian graphs by applying the ``inserting a dodecahedron'' operation from~\cite{Za21} to $G_{48}$ (the required extendable 5-cycle is shown in Fig.~A in the Appendix), but hypohamiltonicity is not preserved, so this does not work as a proof of Corollary~1.}

%\Carol{An important computational result here would be to find further cells
%that are both $K_1$- and $K_2$-cells (not necessarily planar). In particular:
%can we find $K_1$-cells in which all vertices except $a,b,c,d$ have degree at
%least 4? This would solve an old question of Thomassen. With this
%cells-approach we even have a chance of finding 4-connected hypohamiltonian
%graphs (another Q of Thomassen).}
%

Now that we know of infinitely many planar $K_2$-hypohamiltonian graphs, one 
can ask how quickly the number of pairwise non-isomorphic such graphs grows as 
a function of the graphs' order. Indeed, even leaving planarity aside, no 
non-trivial lower bound was known hitherto. Using the aforementioned lemmas, we 
can prove the following. 
%\Carol{This proof is, structurally, very similar to the proof of Thm 3 in 
%https://arxiv.org/abs/2008.03173 -- mention this?} \Jan{I think it is ok to 
%mention it, but also ok not to mention it. As you wrote it, I think you should 
%decide. (Unless if Gabor or Jarne have a more pronounced opinion on this).}

\begin{theorem}
The number of planar hypohamiltonian $K_2$-hamiltonian graphs grows at least exponentially.
\end{theorem}

\begin{proof}
Let $A$ and $B$ be the two graphs shown in 
Figure~\ref{fig:expGrowthBuildingBlocks}.
%\Carol{What's missing here are two non-isomorphic graphs, each obtained from a 
%hypohamiltonian $K_2$-hamiltonian graph by removing two cubic vertices $v,w$ 
%at 
%distance at least 3 (I'm not completely sure whether this distance req.\ is 
%really necessary, but the point is to avoid any funny business with the 
%overlap 
%of the identifying triples); $N(v)$ and $N(w)$ should be emphasised somehow; 
%it 
%would be nice if embeddings could be given in which the vertices of $N(v)$ are 
%collinear and on the left and lying on the boundary of the infinite face, and 
%vertices of $N(w)$ are collinear and on the right and lying on the boundary of 
%the infinite face. I think a good candidate is $G_{52}$ from 
%Figure~\ref{fig:52v}. Therein take two distinct pairs of vertices, all four 
%lying on the outer face, each pair at distance 3}. \Jan{Jarne confirmed that 
%$G_{52}$ is a good candidate for this and will take care of it when he has the 
%token.}
They are non-isomorphic, of the same order, each obtained by removing two cubic vertices $v,w$ from the planar hypohamiltonian $K_2$-hamiltonian graph $G_{52}$ shown in Figure~\ref{fig:52v} from Section~\ref{subsect:planar}, where $v$ and $w$ lie on the boundary of the infinite face and are at distance 3. Consider
$$G := K_{1,3} \: F_1 \: F_2 \: \ldots \: F_k \: K_{1,3},$$
where each $F_i$ is $A$ or $B$, the operation $\:$ is performed such that the corresponding copies of $N(v)$ and $N(w)$ are identified, and in the case of $K_{1,3}$, its three leaves are used for the identification. We abbreviate this convention by $(\dagger)$.

Since $A$ and $B$ as well as $K_{1,3}$ are planar, it is clear that the 
identifications behind the application of $\:$ can be performed such that $G$ 
is also planar. We now show that $G$ is hypohamiltonian and 
$K_2$-hypohamiltonian. By $(\dagger)$, we will suppress mentioning explicitly 
the triples of vertices to which the identifications behind $\:$ are applied, 
in order to improve readability. Clearly, since both $K_{1,3} \: F_1$ and 
$F_2 
\: K_{1,3}$ are vertex-deleted subgraphs of $G_{52}$, which itself is a 
hypohamiltonian $K_2$-hamiltonian graph, by Lemmas~\ref{lem:thom} 
and~\ref{lem:carol}, the graph $K_{1,3} \: F_1 \: F_2 \: K_{1,3}$ is 
hypohamiltonian and $K_2$-hypohamiltonian. Using the same argument, applying 
$\:$ to $K_{1,3} \: F_1 \: F_2$ and $F_3 \: K_{1,3}$ yields the 
hypohamiltonian 
$K_2$-hamiltonian graph $K_{1,3} \: F_1 \: F_2 \: F_3 \: K_{1,3}$. Iterating 
this procedure yields that $G$ must be hypohamiltonian and 
$K_2$-hypohamiltonian. 
%\Gabor{What follows here is a detailed version of the 
%rest of the original proof, please check it.} \Jarne{Looks good to me.}

Now we show that if $$G_1:=K_{1,3} \: F^1_1 \: F^1_2 \: \ldots \: F^1_k \: 
K_{1,3} \qquad {\rm and} \qquad G_2:=K_{1,3} \: F^2_1 \: F^2_2 : \ldots \: 
F^2_k 
\: K_{1,3}$$ are isomorphic, then we either have $F^1_i \cong F^2_i$ for all 
$i=1,\ldots, k$ or $F^1_i \cong F^2_{k-i+1}$  for all $i = 1, \ldots, k$, 
yielding an exponential number of pairwise distinct $G$'s.

$G_{52}$ has exactly one non-cubic vertex. For a non-trivial $3$-cut to exist, 
we need at least three such vertices (see Proposition 1 (ii) of \cite{Za21}).
By construction and the fact that $G_{52}$ has only trivial 3-cuts,
%\Gabor{we 
%should give some kind of proof for the claim about $G_{52}$}
$G$ has exactly $k - 1$ non-trivial 3-cuts, namely the ones between the $F_i$'s.
Let $f$ be an isomorphism that maps $G_1$ onto $G_2$ and denote the 3-cut between $F^1_i$ and $F^1_{i+1}$ by $X_i$ and the 3-cut between $F^2_i$ and $F^2_{i+1}$ by $Y_i$. For a subset $Z$ of vertices of $G_1$ let us denote by $f(Z)$ the set of vertices onto which the vertices of $Z$ are mapped by $f$.   
It is obvious that $f(X_1)=Y_i$ for some $i$. Moreover, the possible values 
for 
$i$ are only $1$ and $k$, since the deletion of the vertices of $Y_i$ must 
give 
a two component graph with components of $48$ and $47(k-1)+1$ vertices. If 
$i=1$, then $f(X_2)=Y_2$, because the deletion of the vertices of 
$f(X_1)=Y_1$ 
and $f(X_2)$ must give a three component graph with components of $48$, $44$, 
and $47(k-2)+1$ vertices. The proof that $f(X_j)=Y_j$ in the $i=1$ case is 
similar. Now it is obvious that $F^1_i \cong F^2_i$ for all 
$i=1,\ldots, k$. If $i=k$, then the same argument shows that 
$F^1_i \cong F^2_{k-i+1}$ for all $i = 1, \ldots, k$, finishing the proof.
% So any isomorphism must map these onto each other. Hence, if
%$$K_{1,3} \: F^1_1 \: F^1_2 \: \ldots \: F^1_k \: K_{1,3} \qquad {\rm and} \qquad K_{1,3} \: F^2_1 \: F^2_2 : \ldots \: F^2_k \: K_{1,3}$$
%are isomorphic but are composed of different building blocks (i.e.\ a different sequence of $A$'s and $B$'s), we must have $F^1_1 \cong F^2_1$ or $F^1_1 \cong F^2_k$, so $F^1_i \cong F^2_i$ or $F^1_i \cong F^2_{k+1-i}$ for all $i \in \{ 1, \ldots, k \}$. This yields an exponential number of pairwise distinct $G$'s.
\end{proof}

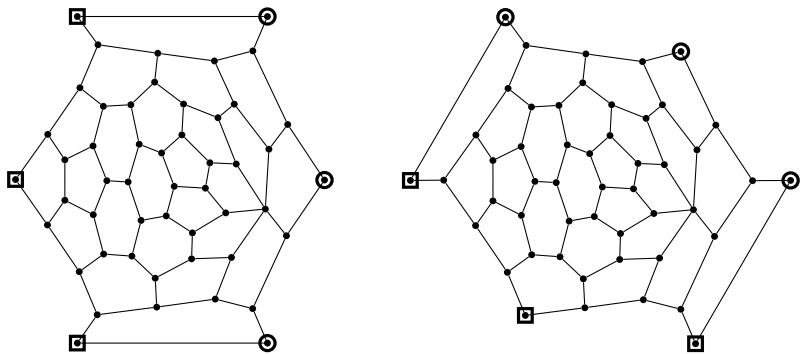
\begin{figure}[!htb]
	\begin{center}
		\begin{tikzpicture}[scale=0.05]
			\definecolor{marked}{rgb}{0.25,0.5,0.25}
			\node [circle,fill,scale=0.25] (52) at (38.356969,49.427872) {};
			\node [circle,fill,scale=0.25] (51) at (47.158925,57.114914) {};
			\node [circle,fill,scale=0.25] (50) at (41.770171,39.227383) {};
			\node [circle,fill,scale=0.25] (49) at (32.762836,49.760391) {};
			\node [circle,fill,scale=0.25] (48) at (41.271394,59.432763) {};
			\node [circle,fill,scale=0.25] (47) at (52.498778,61.916870) {};
			\node [circle,fill,scale=0.25] (46) at (50.493888,48.264058) {};
			\node [circle,fill,scale=0.25] (45) at (48.371639,40.430317) {};
			\node [circle,fill,scale=0.25] (44) at (39.344744,29.760391) {};
			\node [circle,fill,scale=0.25] (43) at (29.193154,40.743276) {};
			\node [circle,fill,scale=0.25] (42) at (21.721271,44.586797) {};
			\node [circle,fill,scale=0.25] (41) at (21.711492,55.295843) {};
			\node [circle,fill,scale=0.25] (40) at (29.124695,58.982885) {};
			\node [circle,fill,scale=0.25] (39) at (39.061125,69.907091) {};
			\node [circle,fill,scale=0.25] (38) at (52.938876,70.083129) {};
			\node [circle,fill,scale=0.25] (37) at (59.872862,54.513447) {};
			\node [circle,fill,scale=0.25] (36) at (58.679707,47.784841) {};
			\node [circle,fill,scale=0.25] (35) at (55.266504,35.921760) {};
			\node [circle,fill,scale=0.25] (34) at (55.051345,29.046455) {};
			\node [circle,fill,scale=0.25] (33) at (45.486553,23.911980) {};
			\node [circle,fill,scale=0.25] (32) at (31.911981,30.308068) {};
			\node [circle,fill,scale=0.25] (31) at (17.144255,38.024450) {};
			\node [circle,fill,scale=0.25] (30) at (17.232275,62.053789) {};
			\node [circle,fill,scale=0.25] (29) at (31.853302,69.486552) {};
			\node [circle,fill,scale=0.25] (28) at (45.330074,75.911979) {};
			\node [circle,fill,scale=0.25] (27) at (61.985331,66.464548) {};
			\node [circle,fill,scale=0.25] (26) at (66.797067,54.190709) {};
			\node [circle,fill,scale=0.25] (25) at (64.039121,41.232274) {};
			\node [circle,fill,scale=0.25] (24) at (65.535453,29.408313) {};
			\node [circle,fill,scale=0.25] (23) at (45.877751,16.244499) {};
			\node [circle,fill,scale=0.25] (22) at (25.506113,25.633252) {};
			\node [circle,fill,scale=0.25] (21) at (8.753056,50.073349) {};
			\node [circle,fill,scale=0.25] (20) at (25.682151,74.415648) {};
			\node [circle,fill,scale=0.25] (19) at (46.161370,83.559901) {};
			\node [circle,fill,scale=0.25] (18) at (66.278729,70.063570) {};
			\node [circle,fill,scale=0.25] (17) at (75.383863,58.112469) {};
			\node [circle,fill,scale=0.25] (16) at (74.444988,42.249388) {};
			\node [circle,fill,scale=0.25] (15) at (61.271395,18.366748) {};
			\node [circle,fill,scale=0.25] (14) at (30.239609,14.210270) {};
			%\node [circle,fill,scale=0.25] (13) at (0.000000,50.004890) {};
			\node [circle,fill,scale=0.25] (12) at (30.435208,85.819070) {};
			\node [circle,fill,scale=0.25] (11) at (61.056235,81.515893) {};
			\node [circle,fill,scale=0.25] (10) at (80.332518,64.674816) {};
			\node [circle,fill,scale=0.25] (9) at (80.000001,35.168704) {};
			\node [circle,fill,scale=0.25] (8) at (71.119805,15.843520) {};
			\node [circle,fill,scale=0.25] (7) at (24.997555,6.699266) {};
			\node [circle,fill,scale=0.25] (6) at (24.997555,93.300733) {};
			\node [circle,fill,scale=0.25] (5) at (71.158924,84.185820) {};
			\node [circle,fill,scale=0.25] (4) at (90.004891,49.936430) {};
			\node [circle,fill,scale=0.25] (3) at (75.002446,6.699266) {};
			\node [circle,fill,scale=0.25] (2) at (75.002446,93.300733) {};
			%\node [circle,fill,scale=0.25] (1) at (99.999999,50.004890) {};
			\draw [black] (52) to (48);
			\draw [black] (52) to (50);
			\draw [black] (52) to (49);
			\draw [black] (51) to (47);
			\draw [black] (51) to (46);
			\draw [black] (51) to (48);
			\draw [black] (50) to (44);
			\draw [black] (50) to (45);
			\draw [black] (49) to (43);
			\draw [black] (49) to (40);
			\draw [black] (48) to (39);
			\draw [black] (47) to (38);
			\draw [black] (47) to (37);
			\draw [black] (46) to (36);
			\draw [black] (46) to (45);
			\draw [black] (45) to (35);
			\draw [black] (44) to (32);
			\draw [black] (44) to (33);
			\draw [black] (43) to (42);
			\draw [black] (43) to (32);
			\draw [black] (42) to (31);
			\draw [black] (42) to (41);
			\draw [black] (41) to (30);
			\draw [black] (41) to (40);
			\draw [black] (40) to (29);
			\draw [black] (39) to (28);
			\draw [black] (39) to (29);
			\draw [black] (38) to (27);
			\draw [black] (38) to (28);
			\draw [black] (37) to (36);
			\draw [black] (37) to (26);
			\draw [black] (36) to (25);
			\draw [black] (35) to (34);
			\draw [black] (35) to (25);
			\draw [black] (34) to (24);
			\draw [black] (34) to (33);
			\draw [black] (33) to (23);
			\draw [black] (32) to (22);
			\draw [black] (31) to (21);
			\draw [black] (31) to (22);
			\draw [black] (30) to (20);
			\draw [black] (30) to (21);
			\draw [black] (29) to (20);
			\draw [black] (28) to (19);
			\draw [black] (27) to (18);
			\draw [black] (27) to (26);
			\draw [black] (26) to (16);
			\draw [black] (25) to (16);
			\draw [black] (24) to (15);
			\draw [black] (24) to (16);
			\draw [black] (23) to (14);
			\draw [black] (23) to (15);
			\draw [black] (22) to (14);
			%\draw [black] (21) to (13);
			\draw [black] (20) to (12);
			\draw [black] (19) to (11);
			\draw [black] (19) to (12);
			\draw [black] (18) to (17);
			\draw [black] (18) to (11);
			\draw [black] (17) to (10);
			\draw [black] (17) to (16);
			\draw [black] (16) to (9);
			\draw [black] (15) to (8);
			\draw [black] (14) to (7);
			%\draw [black] (13) to (6);
			%\draw [black] (13) to (7);
			\draw [black] (12) to (6);
			\draw [black] (11) to (5);
			\draw [black] (10) to (4);
			\draw [black] (10) to (5);
			\draw [black] (9) to (8);
			\draw [black] (9) to (4);
			\draw [black] (8) to (3);
			\draw [black] (7) to (3);
			\draw [black] (6) to (2);
			\draw [black] (5) to (2);
			%\draw [black] (4) to (1);
			%\draw [black] (3) to (1);
			%\draw [black] (2) to (1);
			\path[draw=black, very thick] (21) +(-1.8,-1.8) 
			rectangle +(1.8,1.8);
			\path[draw=black, very thick] (7) +(-1.8,-1.8) 
			rectangle +(1.8,1.8);
			\path[draw=black, very thick] (6) +(-1.8,-1.8) 
			rectangle +(1.8,1.8);
%			\path[draw=gray,very thick] (21) circle[radius=2];
%			\path[draw=gray,very thick] (7) circle[radius=2];
%			\path[draw=gray,very thick] (6) circle[radius=2];
			\path[draw=black,very thick] (4) circle[radius=2];
			\path[draw=black,very thick] (3) circle[radius=2];
			\path[draw=black,very thick] (2) circle[radius=2];
		\end{tikzpicture}
	\qquad
	\begin{tikzpicture}[scale=0.05]
		\definecolor{marked}{rgb}{0.25,0.5,0.25}
		\node [circle,fill,scale=0.25] (52) at (38.356969,49.427872) {};
		\node [circle,fill,scale=0.25] (51) at (47.158925,57.114914) {};
		\node [circle,fill,scale=0.25] (50) at (41.770171,39.227383) {};
		\node [circle,fill,scale=0.25] (49) at (32.762836,49.760391) {};
		\node [circle,fill,scale=0.25] (48) at (41.271394,59.432763) {};
		\node [circle,fill,scale=0.25] (47) at (52.498778,61.916870) {};
		\node [circle,fill,scale=0.25] (46) at (50.493888,48.264058) {};
		\node [circle,fill,scale=0.25] (45) at (48.371639,40.430317) {};
		\node [circle,fill,scale=0.25] (44) at (39.344744,29.760391) {};
		\node [circle,fill,scale=0.25] (43) at (29.193154,40.743276) {};
		\node [circle,fill,scale=0.25] (42) at (21.721271,44.586797) {};
		\node [circle,fill,scale=0.25] (41) at (21.711492,55.295843) {};
		\node [circle,fill,scale=0.25] (40) at (29.124695,58.982885) {};
		\node [circle,fill,scale=0.25] (39) at (39.061125,69.907091) {};
		\node [circle,fill,scale=0.25] (38) at (52.938876,70.083129) {};
		\node [circle,fill,scale=0.25] (37) at (59.872862,54.513447) {};
		\node [circle,fill,scale=0.25] (36) at (58.679707,47.784841) {};
		\node [circle,fill,scale=0.25] (35) at (55.266504,35.921760) {};
		\node [circle,fill,scale=0.25] (34) at (55.051345,29.046455) {};
		\node [circle,fill,scale=0.25] (33) at (45.486553,23.911980) {};
		\node [circle,fill,scale=0.25] (32) at (31.911981,30.308068) {};
		\node [circle,fill,scale=0.25] (31) at (17.144255,38.024450) {};
		\node [circle,fill,scale=0.25] (30) at (17.232275,62.053789) {};
		\node [circle,fill,scale=0.25] (29) at (31.853302,69.486552) {};
		\node [circle,fill,scale=0.25] (28) at (45.330074,75.911979) {};
		\node [circle,fill,scale=0.25] (27) at (61.985331,66.464548) {};
		\node [circle,fill,scale=0.25] (26) at (66.797067,54.190709) {};
		\node [circle,fill,scale=0.25] (25) at (64.039121,41.232274) {};
		\node [circle,fill,scale=0.25] (24) at (65.535453,29.408313) {};
		\node [circle,fill,scale=0.25] (23) at (45.877751,16.244499) {};
		\node [circle,fill,scale=0.25] (22) at (25.506113,25.633252) {};
		\node [circle,fill,scale=0.25] (21) at (8.753056,50.073349) {};
		\node [circle,fill,scale=0.25] (20) at (25.682151,74.415648) {};
		\node [circle,fill,scale=0.25] (19) at (46.161370,83.559901) {};
		\node [circle,fill,scale=0.25] (18) at (66.278729,70.063570) {};
		\node [circle,fill,scale=0.25] (17) at (75.383863,58.112469) {};
		\node [circle,fill,scale=0.25] (16) at (74.444988,42.249388) {};
		\node [circle,fill,scale=0.25] (15) at (61.271395,18.366748) {};
		\node [circle,fill,scale=0.25] (14) at (30.239609,14.210270) {};
		\node [circle,fill,scale=0.25] (13) at (0.000000,50.004890) {};
		\node [circle,fill,scale=0.25] (12) at (30.435208,85.819070) {};
		\node [circle,fill,scale=0.25] (11) at (61.056235,81.515893) {};
		\node [circle,fill,scale=0.25] (10) at (80.332518,64.674816) {};
		\node [circle,fill,scale=0.25] (9) at (80.000001,35.168704) {};
		\node [circle,fill,scale=0.25] (8) at (71.119805,15.843520) {};
		%\node [circle,fill,scale=0.25] (7) at (24.997555,6.699266) {};
		\node [circle,fill,scale=0.25] (6) at (24.997555,93.300733) {};
		\node [circle,fill,scale=0.25] (5) at (71.158924,84.185820) {};
		\node [circle,fill,scale=0.25] (4) at (90.004891,49.936430) {};
		\node [circle,fill,scale=0.25] (3) at (75.002446,6.699266) {};
		%\node [circle,fill,scale=0.25] (2) at (75.002446,93.300733) {};
		\node [circle,fill,scale=0.25] (1) at (99.999999,50.004890) {};
		\draw [black] (52) to (48);
		\draw [black] (52) to (50);
		\draw [black] (52) to (49);
		\draw [black] (51) to (47);
		\draw [black] (51) to (46);
		\draw [black] (51) to (48);
		\draw [black] (50) to (44);
		\draw [black] (50) to (45);
		\draw [black] (49) to (43);
		\draw [black] (49) to (40);
		\draw [black] (48) to (39);
		\draw [black] (47) to (38);
		\draw [black] (47) to (37);
		\draw [black] (46) to (36);
		\draw [black] (46) to (45);
		\draw [black] (45) to (35);
		\draw [black] (44) to (32);
		\draw [black] (44) to (33);
		\draw [black] (43) to (42);
		\draw [black] (43) to (32);
		\draw [black] (42) to (31);
		\draw [black] (42) to (41);
		\draw [black] (41) to (30);
		\draw [black] (41) to (40);
		\draw [black] (40) to (29);
		\draw [black] (39) to (28);
		\draw [black] (39) to (29);
		\draw [black] (38) to (27);
		\draw [black] (38) to (28);
		\draw [black] (37) to (36);
		\draw [black] (37) to (26);
		\draw [black] (36) to (25);
		\draw [black] (35) to (34);
		\draw [black] (35) to (25);
		\draw [black] (34) to (24);
		\draw [black] (34) to (33);
		\draw [black] (33) to (23);
		\draw [black] (32) to (22);
		\draw [black] (31) to (21);
		\draw [black] (31) to (22);
		\draw [black] (30) to (20);
		\draw [black] (30) to (21);
		\draw [black] (29) to (20);
		\draw [black] (28) to (19);
		\draw [black] (27) to (18);
		\draw [black] (27) to (26);
		\draw [black] (26) to (16);
		\draw [black] (25) to (16);
		\draw [black] (24) to (15);
		\draw [black] (24) to (16);
		\draw [black] (23) to (14);
		\draw [black] (23) to (15);
		\draw [black] (22) to (14);
		\draw [black] (21) to (13);
		\draw [black] (20) to (12);
		\draw [black] (19) to (11);
		\draw [black] (19) to (12);
		\draw [black] (18) to (17);
		\draw [black] (18) to (11);
		\draw [black] (17) to (10);
		\draw [black] (17) to (16);
		\draw [black] (16) to (9);
		\draw [black] (15) to (8);
		%\draw [black] (14) to (7);
		\draw [black] (13) to (6);
		%\draw [black] (13) to (7);
		\draw [black] (12) to (6);
		\draw [black] (11) to (5);
		\draw [black] (10) to (4);
		\draw [black] (10) to (5);
		\draw [black] (9) to (8);
		\draw [black] (9) to (4);
		\draw [black] (8) to (3);
		%\draw [black] (7) to (3);
		%\draw [black] (6) to (2);
		%\draw [black] (5) to (2);
		\draw [black] (4) to (1);
		\draw [black] (3) to (1);
		%\draw [black] (2) to (1);
		\path[draw=black, very thick] (14) +(-1.8,-1.8) 
		rectangle +(1.8,1.8);
		\path[draw=black, very thick] (13) +(-1.8,-1.8) 
		rectangle +(1.8,1.8);
		\path[draw=black, very thick] (3) +(-1.8,-1.8) 
		rectangle +(1.8,1.8);
%		\path[draw=gray,very thick] (14) circle[radius=2];
%		\path[draw=gray,very thick] (13) circle[radius=2];
%		\path[draw=gray,very thick] (3) circle[radius=2];
		\path[draw=black,very thick] (6) circle[radius=2];
		\path[draw=black,very thick] (5) circle[radius=2];
		\path[draw=black,very thick] (1) circle[radius=2];
	\end{tikzpicture}
	\end{center}
	\caption{Graph $A$ (left) and graph $B$ (right) are two non-isomorphic 
	graphs obtained from $G_{52}$ by removing two cubic vertices. The 
	neighbours 
	of these removed vertices are indicated by squares (for the first vertex) 
	and 
	circles (for the second vertex).
%	\Jarne{I am having some trouble 
%	rotating this second graph via tikz, but maybe this is fine as this way 
%	show clearly 
%	which vertices of $G_{52}$ were deleted? Otherwise I could redraw it in the 
%	correct orientation using CaGe.}\Gabor{I think no rotation is needed, it's 
%fine as it is.}
}
	\label{fig:expGrowthBuildingBlocks}
\end{figure}

We note that this also gives an alternative proof of Collier and Schmeichel's theorem stating that the number of hypohamiltonian graphs grows at least exponentially~\cite{CS77}. (In fact, here we show that this even holds restricted to planar graphs.)

One might wonder if there are further, not necessarily planar suitable cells 
that are both $K_1$- and $K_2$-cells (which are preferably smaller than the 
quintuple $(J_{18},a,b,c,d)$). A related question is whether there exist 
$K_1$-cells in which all inner vertices, i.e.\ all vertices except $a,b,c,d$, 
have degree at least 4. This would solve an old question of Thomassen, namely 
whether 
hypohamiltonian graphs of minimum degree at least 4 exist~\cite{Th78}. (We note 
that all $K_2$-hypohamiltonian graphs we shall encounter in the remainder of 
this text have minimum degree $3$.)
%\Jarne{Maybe we should put this as a 
%footnote?}
% \Gabor{I would mention the question. E.g. This would solve an old question of 
%Thomassen, whether hypohamiltonian graphs of minimum degree at least 4 
%exist~\cite{Th78}.}
Using such cells might even allow to find 4-connected hypohamiltonian graphs, 
whose existence is another open question of Thomassen~\cite{Th78} and one of 
the most important unresolved problems  related to hypohamiltonicity, and the 
behaviour of longest cycle intersections in general. 
%\Gabor{I changed the last 
%sentence a bit.}
%Using such cells might even allow to find 4-connected hypohamiltonian graphs, which is another open question of Thomassen~\cite{Th78} and an important unresolved problem among questions related to hypohamiltonicity, and the behaviour of longest cycle intersections more generally.  
We recall that infinitely many 4-connected non-hamiltonian graphs exist in 
which exactly one vertex has the property that its removal yields a 
non-hamiltonian graph~\cite{Za15}, and that there are infinite families of 
planar hypohamiltonian graphs containing a constant number of cubic vertices 
(note that hypohamiltonian graphs have minimum degree at least~3), based on an 
operation of Thomassen~\cite[p.~41]{Th81}.

%\Jan{@Carol, could you please rephrase the part about this degree at least 4 a bit and add references? Possibly we could wait to mention this part till after the observations below.} \Carol{I've rephrased it and added a ref. I think that the location here is quite good because it motivates the following observations, but if you feel that it's better elsewhere, please move it.}
%JG: It's ok for me now. Thanks!

Using the generator \verb|geng|~\cite{MP14}, which can generate 
all connected (simple) graphs of a given order (and lower bound on the girth), 
in combination with our program from the proof of Lemma~\ref{lem:j18}, we 
searched for other suitable cells, $K_1$-cells, as well as $K_2$-cells other 
than $(J_{18},a,b,c,d)$. The results of these computations are summarised in 
the following observations and Table~\ref{table:counts_suitable_cells_g5}. The 
suitable cells up to 18 vertices reported in this table can be obtained from 
the database of interesting graphs from the \textit{House of 
Graphs}~\cite{CDG23} by searching for the keywords ``suitable cell''.

\begin{observation}
There are no suitable cells of order at most $11$ and none of girth at least $4$ of order $12$ and $13$. The smallest suitable cells of girth at least $5$ have order $16$ and their exact counts are listed in Table~\ref{table:counts_suitable_cells_g5}.
\end{observation}

%\Jarne{There are no suitable cells of orders 6-11. No suitable cells of girth
%$\geq 4$
%of orders 12-13. None of girth $\geq 5$ of order 14-15. There are four suitable
%cells of order 16 for girth $\geq 5$. There are 28 suitable cells of order 17
%for girth $\geq 5$. So far none of them are $K_1$ or $K_2$-cells. There are 365
%suitable cells of order 18 for girth $\geq 5$, but $J_{18}$ is the only
%$K_1$-cell and the only $K_2$-cell. This is about as far as I can go, unless we
%look in some class with fewer graphs.}

\begin{table}[ht!]
\centering
\begin{tabular}{ l | ccccc}
%\hline
Order 		& $\leq 15$ & 16 & 17 & 18 & 19\\
\hline
Number of cells & 0 & 4 & 28 & 365 & 6\,861\\
\end{tabular}
\caption{Counts of all suitable cells of girth at least 5 of a given order up 
to $19$ vertices.}
\label{table:counts_suitable_cells_g5}
\end{table}

\begin{observation}
There is exactly one suitable cell of girth at least $5$ up to $19$ vertices 
%\Jan{Now up to 19 vertices?}
which is a $K_1$-cell and exactly one of those suitable cells is a $K_2$-cell. 
In both cases it is the same cell, namely $(J_{18},a,b,c,d)$ from 
Figure~\ref{fig:j18}.
\end{observation}

%\Carol{I'd merge Obs. 1 and 2, and remove the table (and thus also the exact
%counts; although one might mention the count for order 18), e.g.: \emph{There
%are no suitable cells of order at most $11$ and none of girth at least $4$ of
%order $12$ and $13$. The smallest suitable cells of girth at least $5$ have
%order $16$[, and there are exactly $365$ suitable cells of girth at least~$5$
%on $18$ vertices]. There is exactly one suitable cell of girth at least $5$ up
%to $18$ vertices which is a $K_1$-cell and exactly one of those suitable cells
%is a $K_2$-cell. In both cases it is the same cell, namely 
%$(J_{18},a,b,c,d)$.}}
%\Jarne{I agree that the counts themselves are not very interesting. I have
%computed that for $19$ vertices there exist $6861$ suitable cells of girth at
%least $5$, but none are a $K_1$- or $K_2$-cell. This is maybe something we
%should mention as well.}
%\Jan{Personally, I would keep the two observations and the Table as I think 
%the counts provide useful information (e.g.\ about the growth rate / order of 
%magnitude and also for independent verification of our results). I think one 
%big observation would be too dense, but I can live with that. But even if you 
%merge it in one Observation and drop the table, I would still mention the 
%counts for 19 vertices as it shows that $J_{18}$ is really special as there 
%are 
%no $K_1$- or $K_2$-cells on 19 vertices (with girth at least 5).}

\begin{corollary}
The quintuple $(J_{18},a,b,c,d)$ from Figure~\ref{fig:j18} is the smallest 
$K_1$- and $K_2$-cell of girth at least $5$. Moreover, any other $K_1$- or 
$K_2$-cell of girth at least $5$ must have at least $20$ vertices.
\end{corollary}

\section{Orders}
\label{sect:orders}

In this section we determine all $K_2$-hypohamiltonian graphs up to a given 
order for certain important classes of graphs and characterise the orders for 
which they exist. In particular, in 
Sections~\ref{subsect:general}--\ref{subsect:planar_cubic}, we deal with 
general graphs, cubic graphs (including snarks), planar graphs, and cubic 
planar graphs, respectively.
These results were obtained by using a combination of computational methods and theoretical arguments.

For the computational part, we amongst others implemented an efficient
algorithm to test if a given graph is $K_2$-hypohamiltonian. The core of our 
algorithm is a straightforward procedure which searches for hamiltonian cycles 
in a graph.
% \footnote{We wish to emphasise that the aim of this research was not 
% to design the fastest possible hamiltonicity test, but a filter that is fast 
% enough for our purposes and that, crucially, is well integrated with other 
% tools such as nauty \cite{nauty-website}. An implementation that is, say, 10\% 
% faster, 
% would not 
% help much because the number of graphs grows very quickly. Ad hoc experiments 
% show that our algorithm performs well compared to others, but to systematically 
% benchmark our implementation and compare it to other implementations is out of 
% this article's scope.}
Starting with a path of two edges centered at a vertex of
minimum degree with a chosen end, we recursively grow a path by adding a vertex
to this end in each possible way. Given a path $P$ with fewer than
$n$ vertices, if the first vertex of $P$ is not adjacent to a
vertex not in $P$, or there is a vertex not in $P$ that is
adjacent to fewer than two vertices that are not interior
vertices of $P$, then $P$ cannot be part of a hamiltonian cycle
so we discontinue extending it. If $P$ has $n$ vertices and the ends are adjacent, we have found a hamiltonian cycle.
%\Jan{Carol and Gabor: if you think this paragraph is still too technical or 
%long for SIJDMA, feel free to remove/reduce it. The same description (with 
%much 
%more details) is also in the appendix. But I think it should be more or less 
%OK 
%to keep it in here.} \Gabor{I think we can keep it here. A sentence about 
%what 
%the challenges are and why our approach is sound would be useful.}

This algorithm for finding hamiltonian cycles can then be used to determine
whether graphs are hamiltonian, $K_1$-hamiltonian, and $K_2$-hamiltonian.
% A more
% detailed explanation can be found in Section~\ref{app:algorithm} of the Appendix.

The source code of this algorithm can be downloaded from
GitHub \cite{Re22} and in the subsections below we describe how we
verified the correctness of our implementation.

\subsection{The general case}
\label{subsect:general}

A characterisation of the orders for which $K_2$-hypohamiltonian graphs exist is a $K_2$-analogue of a natural question---\emph{For which $n$ do there exist hypohamiltonian graphs on $n$ vertices?}---addressed in a series of articles, among which the paper by Aldred, McKay, and Wormald~\cite{AMW97} which completely settles the problem by proving that there is no hypohamiltonian graph on 17 vertices. Thus, the answer to the aforementioned question is that there exists a hypohamiltonian graph of order $n$ if and only if $n \in \{ 10, 13, 15, 16 \}$ or $n \ge 18$ is an integer.

The first author determined that all (seven) hypohamiltonian graphs on at most
16 vertices are $K_2$-hamiltonian, while the last author showed that the
generalised Petersen graph ${\rm GP}(2,11)$ is $K_2$-hamiltonian~\cite{Za21}
(that ${\rm GP}(2,11)$ is hypohamiltonian was shown by Bondy~\cite{Bo72}).
Combining these graphs via Lemmas~\ref{lem:thom} and~\ref{lem:carol}, we obtain
hypohamiltonian $K_2$-hamiltonian graphs of order 10, 13, 15, 16, 18, and every
integer which is at least $20$; a detailed account is given in the proof of
Theorem~\ref{thm:orders_general}. For all other (non-trivial) orders it was an
open question whether a $K_2$-hypohamiltonian graph of that order exists. It is
not difficult to verify that $K_2$-hypohamiltonian graphs are 3-connected.
Hence, their minimum degree is at least $3$.

%\Jarne{Below are my computations for the $K_2$-hypohamiltonicity of all
%connected simple graphs. This settles all orders except for $14$ and $17$. I
%would like to still do all graphs on $17$ vertices with girth $\geq 4$ on the
%cluster (97 CPU years), however I am not optimistic that one will exist. Jan
%will maybe do this on the UGent cluster. $14$-vertices is not feasible with my
%program (I estimate (probably a lot) more than 1777 CPU years).}

%\Jan{I agree with Carol that it will be too long to include Jarne's operation here. In my opinion it would indeed be better for a separate paper (possibly together with other side-results).

%@Carol: if you agree, could you please remove the above discussion and add a short motivation of why it is useful to determine the orders + which orders are already settled (i.e.\ basically what you already wrote above)? I think it would be good to also mention these basic properties a $K_2$-hypohamiltonian graph must have.

%\Jan{@Jarne: 97 CPU years is already too much for the UGent computer, unless
%if it would be super important (e.g.\ close the last open case), but that is
%not the case here. So I wouldn't do it, but perhaps it can be useful to also
%look into higher girths...}

We used the program \verb|geng|~\cite{MP14} to generate all
connected simple graphs of a given order (with a given lower bound on the
girth and minimum degree) and then filtered these graphs using our program to
check whether they
are $K_2$-hypohamiltonian or not. The results of these
computations are summarised in Table~\ref{table:counts_general} in 
Section~\ref{app:counts_k2hypo} of the Appendix. We verified all results for 
orders 5--12, 14--16 (with girth at least $4$), and 17--19 (with girth at least 
$5$) 
using 
an
independent program. Checking higher orders with the independent program would 
be unfeasible without allocating a lot of computational resources.
% \Jarne{Do 
%some extra double checks.}
%\Jan{@Jarne: can you also independently verify the K2-hypohamiltonicity for
%higher orders (e.g.\ 18+19 and girth $\geq 5$), since among the orders you
%mentioned there is only one K2-hypohamiltonian graph (i.e.\ the Petersen
%graph), so this is not very convincing evidence of the
% correctness.}
%\Jarne{Added 14,15,18. If I have some computation time to spare I will try
%some more.}
%The counts of non-hamiltonian graphs can also be found in sequence \href{https://oeis.org/A126149}{A126149} in the On-Line Encyclopedia of Integer Sequences~\cite{OEIS}.
Previously, all non-hamiltonian graphs up to 13 vertices were already determined in~\cite{GMZ20}.
%The $K_2$-hypohamiltonian graphs from Table~\ref{table:counts_general} can be
%downloaded from the \textit{House of Graphs}~\cite{CDG23} at
%\url{https://hog.grinvin.org/K2hypohamiltonian} \Jan{TODO: I still need to
%create such a webpage}.
The $K_2$-hypohamiltonian graphs from Table~\ref{table:counts_general} can be
downloaded from the database of interesting graphs from the \textit{House of 
Graphs}~\cite{CDG23} by searching for the keywords ``K2-hypohamiltonian''.
% \Jan{TODO: I still need to upload these graphs}

%\Jarne{Adding the hypohamiltonian $K_2$-hamiltonian graphs might be interesting
%as well. TODO: find a better way to represent it in the table and refer to the
%total hypohamiltonian counts.}

%The following corollary immediately follows from the results from Table~\ref{table:counts_general}. %\Jan{Note: at the beginning of the subsection we still have to add a short introduction to make clear which orders were still open.}

%\begin{corollary}\label{cor:orders}
%There exist $K_2$-hypohamiltonian graphs of order $10$, $13$, $15$, $16$ $18$, $19$. There exist no $K_2$-Hamiltonian graphs of order $5$, $6$, $7$, $8$, $9$, $11$, $12$. Moreover, any $K_2$-hypohamiltonian graph of order $14$ (if such a graph exists) must have girth $3$ and any $K_2$-hypohamiltonian graph of order $17$ (if such a graph exists) must have girth at most $4$. \Jan{Maybe make it more clear in the surrounding text that 14 and 17 are the only open cases.}
%\end{corollary}

%\Jarne{I have yet to encounter any graph of girth $3$ or $4$. Maybe they do not exist (up to some order)?}
%\Jan{That seems like an interesting question/open problem to mention. I would also explicitly note that all of the K2-hypoham graphs you found have girth 5, since that cannot be deduced from the table.}
%\Carol{I've expanded the old Note~5, now 4 (regarding girths).}
%JG: ok, thanks!

%Perhaps using a theorem environment is a bit easier?
%\noindent \textbf{Theorem} \emph{For any integer $n\geq 18$, there exists a
%$K_2$-hypohamiltonian graph of order $n$.}
\begin{theorem}\label{thm:orders_general}
 The Petersen graph is the only $K_2$-hypohamiltonian graph of order at 
 most~$12$, any $K_2$-hypohamiltonian graph of order~$14$ (if such a graph exists) must have girth~$3$, and any $K_2$-hypohamiltonian graph of order~$17$
 (if such a graph exists) must have girth at most~$4$. Moreover, for any integer $n \in \{10, 13, 15, 16\}$ or $n \geq 18$, there exists a $K_2$-hypohamiltonian graph of order~$n$.
\end{theorem}
\begin{proof}
The statements from the theorem's first sentence follow directly from our computations summarised in Table~\ref{table:counts_general} in the Appendix. It also gives an independent verification of the first author's computations that there exists an $n$-vertex hypohamiltonian $K_2$-hamiltonian graph for every $n \in \{10, 13, 15, 16\}$.

As shown by the last author in~\cite{Za21}, the generalised Petersen graph ${\rm GP}(2,11)$ on $22$ vertices is a cubic hypohamiltonian $K_2$-hamiltonian graph. Since each of the aforementioned $K_2$-hamiltonian graphs of orders 10, 13, 15, and 16 is also hypohamiltonian, and each has at least five cubic vertices, we can apply Lemmas~\ref{lem:thom} and~\ref{lem:carol} to obtain hypohamiltonian $K_2$-hamiltonian graphs of orders 18, 20, and 21, each containing a cubic vertex. Here we always apply the operation $\:$ to two 3-fragments of hypohamiltonian graphs obtained by removing a cubic vertex of the graph. Thus, when applying $\:$, in each of the underlying hypohamiltonian graphs at most four cubic vertices are removed or rendered non-cubic. Furthermore, note that if the underlying graphs have order $p$ and $q$, respectively, so that the fragments have order $p - 1$ and $q - 1$, respectively, the graph resulting from the application of $\:$ has order $p + q - 5$.

Our computations yield the existence of a hypohamiltonian $K_2$-hamiltonian
graph on 19 vertices, see Figure~\ref{fig:19-vertex}.
%\Carol{Since this is actually the only new order that we
%contribute (right?), perhaps we should add a figure of one of the 19-vertex
%graphs; a ``nice'' (symmetric) one if possible}
%\Jarne{I added a figure of the most symmetric $K_2$-hypohamiltonian graph on $19$ vertices.}
%an example of which can be
%found in Figure~\ref{fig:19-vertex}.
Hence, there exist hypohamiltonian
$K_2$-hamiltonian graphs of every order between $18$ and $22$, each containing
a cubic vertex. Combining these graphs, via Lemma~\ref{lem:carol}, with
the Petersen graph a suitable number of times,
% \Gabor{a suitable number of times would be more precise}
  we obtain a $K_2$-hypohamiltonian graph
of every order that is at least~18.
\end{proof}
%\Jarne{This number of $10$ cubic vertices is a bit arbitrarily chosen and not
%verifiable unless we give the graphs.}

\begin{figure}[!htb]

	\begin{subfigure}[t]{0.49\linewidth}
		\centering
	\begin{tikzpicture}[rotate=-54]
		
		\node[circle,fill,scale=0.5] (0) at (3.0/1.4,-0.0/1.4) {};
		\node[circle,fill,scale=0.5] (1) at
		(2.7406363729278027/1.4,-1.2202099292274005/1.4) {};
		\node[circle,fill,scale=0.5] (2) at
		(2.0073918190765747/1.4,-2.2294344764321825/1.4) {};
		\node[circle,fill,scale=0.5] (3) at
		(0.9270509831248424/1.4,-2.8531695488854605/1.4) {};
		\node[circle,fill,scale=0.5] (4) at
		(-0.31358538980296/1.4,-2.98356568610482/1.4) {};
		\node[circle,fill,scale=0.5] (5) at
		(-1.4999999999999993/1.4,-2.598076211353316/1.4) {};
		\node[circle,fill,scale=0.5] (6) at
		(-2.427050983124842/1.4,-1.7633557568774196/1.4) {};
		\node[circle,fill,scale=0.5] (7) at
		(-2.934442802201417/1.4,-0.623735072453278/1.4) {};
		\node[circle,fill,scale=0.5] (8) at
		(-2.934442802201417/1.4,0.6237350724532772/1.4) {};
		\node[circle,fill,scale=0.5] (9) at
		(-2.427050983124843/1.4,1.7633557568774192/1.4) {};
		\node[circle,fill,scale=0.5] (10) at
		(-1.5000000000000013/1.4,2.598076211353315/1.4) {};
		\node[circle,fill,scale=0.5] (11) at
		(-0.3135853898029627/1.4,2.9835656861048196/1.4) {};
		\node[circle,fill,scale=0.5] (12) at
		(0.9270509831248417/1.4,2.853169548885461/1.4) {};
		\node[circle,fill,scale=0.5] (13) at
		(2.0073918190765756/1.4,2.229434476432182/1.4) {};
		\node[circle,fill,scale=0.5] (14) at
		(2.7406363729278027/1.4,1.2202099292274005/1.4) {};

 		\node [circle,fill,scale=0.5] (15) at (0,0) {};
 		\node [circle, fill, scale=0.5] (16) at (-0.75, 0.5449068960040206) {};
 		\node [circle, fill, scale=0.5] (17) at (0.6203181864639007,
 		 		0.6889331410629573) {};
 		\node [circle, fill, scale=0.5] (18) at (-0.4635254915624212,
 		 		 		-0.8028497019894527) {};
 		\draw (0) to (1) to (2) to (3) to (4) to (5) to (6) to (7) to (8) to
 		(9) to (10) to (11) to (12) to (13) to (14) to (0);
 		\draw (15) to (16);
 		\draw (15) to (17);
 		\draw (15) to (18);
 		\draw (16) -- (0);
 		\draw (17) -- (1);
 		\draw (18) -- (2);
 		\draw (16) -- (3);
 		\draw (17) -- (4);
 		\draw (18) -- (5);
 		\draw (16) -- (6);
 		\draw (17) -- (7);
 		\draw (18) -- (8);
 		\draw (16) -- (9);
 		\draw (17) -- (10);
 		\draw (18) -- (11);
 		\draw (16) -- (12);
 		\draw (17) -- (13);
 		\draw (18) -- (14);
 		
 	\end{tikzpicture}
 \caption{A hypohamiltonian $K_2$-hamiltonian graph on $19$ vertices. }
 \label{fig:19-vertex}
 \end{subfigure}\hfill
 \begin{subfigure}[t]{0.49\linewidth}
 	\centering
 \begin{tikzpicture}
 	\foreach \i/\name in {{0/ar}, {1/br}, {5/er}}
 	\node[circle,fill,scale=0.5](\name) at ($ 2*({cos(\i*60) + 
 		1/2},{sin(\i*60)}) $) {}; 
 	\foreach \i/\name in {{2/al}, {3/bl}, {4/el}}
 	\node[circle,fill,scale=0.5](\name) at ($ 2*({cos(\i*60) - 
 		1/2},{sin(\i*60)}) $) {};
 	\node[circle,fill,scale=0.5](mt) at ($ 2*(0,{sin(60)}) $) {};
 	\node[circle,fill,scale=0.5](mb1) at ($ 2*(-1/3,{sin(4*60)}) $) {};
 	\node[circle,fill,scale=0.5](mb2) at ($ 2*(1/3,{sin(4*60)}) $) {};
 	\foreach \i/\name in {{0/ari}, {1/bri}, {2/cri}, {4/eri}}
 	\node[circle,fill,scale=0.5](\name) at ($ 
 	1*({cos(\i*72)-(cos(2*72))},{sin(\i*72)}) $) 
 	{};
 	\node[circle,fill,scale=0.5] (dri) at ($ 
 	1*({cos(3*72)-(cos(2*72))+1/5},{sin(3*72)}) 
 	$) {};
 	\foreach \i/\name in {{0/ali}, {1/bli}, {2/cli}, {4/eli}}
 	\node[circle,fill,scale=0.5](\name) at ($ 
 	1*(-{cos(\i*72)+(cos(2*72))},{sin(\i*72)}) 
 	$) {};
 	\node[circle,fill,scale=0.5] (dli) at ($ 
 	1*(-{cos(3*72)+(cos(2*72))-1/5},{sin(3*72)}) 
 	$) {};
 	\foreach \source/\dest in {{ar/ari}, {br/bri}, {er/eri}, {al/bli}, 
 		{bl/ali}, {el/eli}}
 	\draw (\source) to (\dest);
 	\draw (br) to (ar) to (er) to (mb2) to (mb1) to (el) to (bl) to (al) to 
 	(mt) to (br);
 	\draw (ali) to (cli) to (eli) to (bli) to (dli) to (ali);
 	\draw (ari) to (cri) to (eri) to (bri) to (dri) to (ari);
 	\draw (mt) to (cli);
 	\draw (dli) to (mb2);
 	\draw (dri) to (mb1);
 	\path[draw=black, very thick] (mb2) circle[radius=0.2];
 	\path[draw=black, very thick] (mb1) circle[radius=0.2];
 \end{tikzpicture}
\caption{The smallest $K_2$-hypohamiltonian graph of girth $\geq 5$ which 
	is not hypohamiltonian. It has order $18$. The exceptional vertices are 
	marked.}
\label{fig:non-hypoham_K2-hypoham}
\end{subfigure}
	\caption{Two $K_2$-hypohamiltonian graphs.}
	\label{fig:19-vertex_and_non-hypoham_K2-hypoham}
\end{figure}
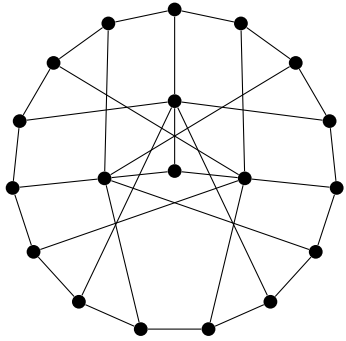
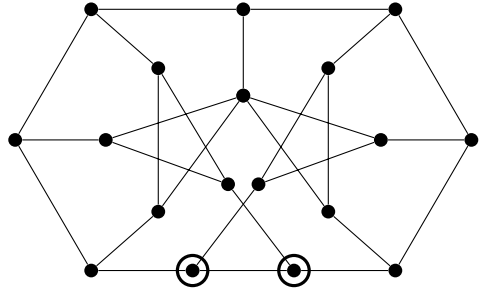

We emphasise that the only integers $n$ for which it is not yet settled whether $K_2$-hypohamiltonian graphs of order $n$ exist are 14 and 17; we leave this as an open problem.

We end this section with a brief remark, which can be directly inferred from 
our computations, on the order of the smallest non-hypohamiltonian 
$K_2$-hypohamiltonian graph. %Moreover, any $K_2$-hypohamiltonian graph of 
%order $14$ must have girth $3$ and any $K_2$-hypohamiltonian graph of order 
%$17$ must have girth at most $4$. \Jan{This basically repeats what was already 
%said in the Corollary. Perhaps it should be rephrased.}

\begin{observation}
The smallest $K_2$-hypohamiltonian graph that contains a vertex whose deletion yields a non-hamiltonian graph has order at least $14$ and at most $18$.
\end{observation}

 The upper bound of 18 is given by the graph from 
 Figure~\ref{fig:non-hypoham_K2-hypoham}. It is the 
 smallest $K_2$-hypohamiltonian graph of girth at least 5 which is not 
 hypohamiltonian.

\subsection{The cubic case}
\label{subsect:cubic}

\subsubsection{General cubic graphs}
\label{subsubsect:cubic}

%\Jan{TODO: add some motivation of why it is interesting to determine all orders for which cubic cubic $K_2$-hypohamiltonian graphs exist (i.e. basically what Carol wrote above)}
%In this subsection we give a result analogous to theorem 5.1 from
%\cite{Th81},
%but for $K_2$-hypohamiltonicity, classifying for what orders cubic $K_2$-hypohamiltonian graphs exist.

Our main motivation for characterising the orders for which cubic $K_2$-hypohamiltonian graphs exist is our goal to give a $K_2$-analogue of
Thomassen's Theorem~5.1 from \cite{Th81}; it characterises all orders for which
cubic hypohamiltonian graphs exist. It is elementary to see that a cubic
$K_2$-hypohamiltonian graph must have girth at least~$5$. Using the program
\verb|Snarkhunter|~\cite{BGM11}, we generated
%\Carol{Can we rephrase this to avoid the double usage of generator/generated?}
all cubic graphs of girth at least 5 up to 32 vertices and verified which of
them are $K_2$-hypohamiltonian using the program mentioned at the beginning of
Section~\ref{sect:orders}. Our findings are summarised in
Table~\ref{table:counts_cubic}. All results up to and including $n = 26$ were
verified by an independently written program. The $K_2$-hypohamiltonian graphs
from Table~\ref{table:counts_cubic} can be
downloaded from the database of interesting graphs from the \textit{House of 
Graphs}~\cite{CDG23} by searching for the keywords ``cubic 
K2-hypohamiltonian''.
%can be downloaded from the \textit{House of
%Graphs}~\cite{CDG23} at
%\url{https://hog.grinvin.org/K2hypohamiltonian}. \Jan{TODO: I still need to
%create such a webpage}. 
The counts for cubic hypohamiltonian graphs were obtained
by Aldred, McKay, and Wormald~\cite{AMW97} and improved by the first and last
author~\cite{GZ17}.

%\Jarne{TODO: refer to the hypohamiltonian counts.}

\begin{table}[ht!]
\centering
	\begin{tabular}{c | r | r | c | c}
		Order & Total & Non-ham. & $K_2$-hypoham. & hypo- and $K_2$-ham.\\
		\hline
		$10$ & $1$ & $1$ & $1$ & $1$\\
		$12$--$16$ & $60$ & $0$ & $0$ & $0$\\
		$18$ & $455$ & $3$ & $0$ & $0$\\
		$20$ & $5\,783$ & $15$ & $1$ & $1$\\
		$22$ & $90\,938
		$ & $110$ & $3$ & $3$\\
		$24$ & $1\,620\,479
		$ & $1\,130$ & $0$ & $0$\\
		$26$ & $31\,478\,584
		$ & $15\,444$ & $6$ & $5$\\
		$28$ & $656\,783\,890
		$ & $239\,126
		$ & $14$ & $12$\\
		$30$ & $14\,621\,871\,204
		$ & $4\,073\,824
		$ & $15$ & $14$\\
		$32$ & $345\,975\,648\,562
		$ & $75\,458\,941
		$ & $12$ & $12$
	\end{tabular}
\caption{Counts of all graphs of girth at least~5 that are cubic; cubic and
non-hamiltonian; cubic and $K_2$-hypohamiltonian; cubic, hypohamiltonian, and
$K_2$-hypohamiltonian, respectively. Recall that cubic $K_2$-hypohamiltonian graphs
have girth at least 5.
%	\Carol{Please check whether I rewrote the caption / interpreted this
%correctly. I'd also say we remove the girth column from the table (since we
%explain this in the caption).}
}
\label{table:counts_cubic}
\end{table}

%\Jarne{TODO: compute 34, maybe 36.}
%\Jan{For sure 36 vertices will be too expensive. To be honest I also wouldn't
%bother with order 34 as it will cost quite a lot of CPU time while we already
%found quite some cubic $K_2$-hypohamiltonian graphs. So I would only try to do
%this if it is really worth it (e.g. if 34 would be the last gap to cover all
%orders).}

%\Jan{TODO: also include results for cubic graphs of higher girths?!}
%\Jarne{There are none on $34$. I am now computing $36$, as it might be
%interesting to determine whether or not the flowers snarks are the only 
%graphs
%of girth $6$ for small orders. Precisely one of the graphs
%	in the above table has girth $6$, namely flower snark $J_7$.}
%\Jan{Yes, that sounds interesting. Please add this as an Observation or so 
%once 
%the computation has finished.}
By restricting the generation in Snarkhunter \cite{BGM11} to cubic graphs of 
girth at least 6, we were also able to determine the following.
\begin{observation}
	Up to and including $36$ vertices there are only two 
	$K_2$-hypohamiltonian 
	graphs of girth at least $6$, namely the flower snark $J_7$ and the flower 
	snark 
	$J_9$.
\end{observation}
%\Jarne{Is the previous observation interesting enough?}

%\Carol{One---rather ugly---way would be to try to use the dot product. I have
%something on this in part I, Section 3.2. If we get two cubic input graphs on
%$n$ and $n'$ vertices, the result has order $n + n' - 2$ and will be cubic.
%But
%if someone decides to implement this, I must honestly say that my proof must
%be
%checked once more very carefully.}
%
%
%\Jan{TODO: I think it would be good to explicitly mention which orders are now
%settled and which ones are still open (possibly in a separate Theorem or
%Corollary environment).}

For the next theorem's proof, we will make use of the following definition and result from~\cite{Za21}. Consider a graph $G$ containing a 5-cycle $C = v_0 \ldots v_4v_0$ such that every $v_i$ is cubic. Denote the neighbour of $v_i$ not on $C$ by $v'_i$. Then the cycle $C$ is called \emph{extendable} if for any $i$, taking indices modulo $5$, (i)~there exists a hamiltonian cycle ${\frak h}$ in $G - v_i$ with $v_{i-2}v_{i+2} \notin E({\frak h})$ and (ii)~there exists a hamiltonian cycle ${\frak h}'$ in $G - v'_i$ with ${\frak h}' \cap C = v_{i-2}v_{i-1}v_iv_{i+1}v_{i+2}$.

\begin{lemma}[\cite{Za21}]
	\label{lem:ext5cycle}
	Let $G$ be a $K_2$-hamiltonian graph containing an extendable $5$-cycle.
	Then there exists a $K_2$-hamiltonian graph $G'$ of order $|V(G)| + 20$
	containing an extendable $5$-cycle. If $G$ is non-hamiltonian or cubic,
	then so is $G'$, respectively. If $G$ has girth~$5$, then $G'$ has
	girth~$5$. If $G$ is plane and $C$ a facial cycle in $G$, then $G'$ is
	planar.
\end{lemma}

\begin{theorem}\label{thm:cubK2hypoham}
	There exists a cubic $K_2$-hypohamiltonian graph of order $n$ if and only if $n \in \{10, 20, 22\}$ or $n \ge 26$ is even.
\end{theorem}
\begin{proof}
	It follows from the counts in Table~\ref{table:counts_cubic} that there 
	exist cubic $K_2$-hypohamiltonian graphs of order 10, 20, 22, 26, 28, 30, 
	and 32, and that no cubic $K_2$-hypohamiltonian graphs of order 
	12, 14, 16, 18, and 24 exist. Hence, it remains to show that there exist cubic 
	$K_2$-hypohamiltonian graphs of even orders $34$ and upwards. In 
	Table~\ref{table:counts_snarks} in Subsection~\ref{subsubsect:snarks} we 
	give cubic $K_2$-hypohamiltonian graphs on $34$ and $36$ vertices and 
	Table~\ref{table:counts_hypohamiltonian_snarks} presents cubic 
	$K_2$-hypohamiltonian graphs on $38$, $40$, and $44$ vertices. We wrote a 
	computer program which tests if a given graph has an extendable 5-cycle. 
	Its implementation can be found on GitHub \cite{Re22}
%	 \Jarne{TODO: add URL. Alternatively, 
%	we could add the url to the references and refer to it here and elsewhere.} 
	and using this 
	program we 
	verified that there exist cubic $K_2$-hypohamiltonian 
	graphs of order $n \in \{22, 26, 28, 30, 32, 34, 36, 38, 40, 44\}$ that 
	contain an extendable $5$-cycle. 
%In the Appendix \ref{app:ext_5-cycle} we give for each order such a graph together with an extendable $5$-cycle $C$ and paths which prove that $C$ is indeed an extendable $5$-cycle so that this can be verified by hand. 
In the same GitHub repository we included a file with a cubic 
$K_2$-hypohamiltonian graph for each of these orders together with an 
extendable $5$-cycle $C$ and paths which prove that $C$ is indeed an extendable 
$5$-cycle so that this can be verified by hand. 
By Lemma~\ref{lem:ext5cycle} we get an $n$-vertex graph with an extendable $5$-cycle for any even $n\geq 34$, which completes the proof.
\end{proof}

\subsubsection{Snarks}
\label{subsubsect:snarks}

It is common to call a graph a \emph{snark} if it is cubic and cyclically 
4-edge-connected, its chromatic index is 4, and its girth is at least 5. Snarks 
are particularly interesting since for several famous conjectures (such as 
Tutte's 5-flow Conjecture) it has been proven that if the conjecture is false, 
the smallest possible counterexample must be a snark.
The interplay between snarks and hypohamiltonicity has been studied by various authors; we refer the reader to~\cite{GZ18} for further references. Our aim in this section is to study the analogous problem for $K_2$-hypohamiltonicity.

In~\cite{BGHM13} Brinkmann, the first author, H\"agglund, and Markstr\"om  
presented an algorithm to generate snarks, used it to generate all 
such graphs up to 36 vertices, and determined the number of hypohamiltonian 
snarks 
among them (and later this was extended up to order 38 for snarks of girth at 
least 6, see~\cite{BG17}).  Using our program for testing 
$K_2$-hypohamiltonicity, 
we determined which of these snarks are $K_2$-hypohamiltonian. The results can 
be found in Table~\ref{table:counts_snarks} and these graphs can also be 
downloaded from the \textit{House of
Graphs}~\cite{CDG23} at
\url{https://houseofgraphs.org/Snarks}. The counts up to and including $n = 30$ 
vertices were verified by an independently written program. 
%\Jan{TODO: I still 
%need to
%add a table with the K2-hypoham snarks there.}
%\Jan{Jarne, please also add a sentence about the independent verification.}

\begin{table}[ht!]
	\centering
	\begin{tabular}{c | c | r | r | r | r}
		$n$ & Girth & Total & hypoham. & $K_2$-hypoham. & hypo- and 
		$K_2$-ham.\\
		\hline
		$10$ & $\geq 5$ & $1$ & $1
		$ & $1$ & $1$\\
		$12$--$16$ & $\geq 5$ & $0
		$ & $0$ & $0$ & $0$\\
		$18$ & $\geq 5$ & $2
		$ & $2$ & $0$ & $0$\\ 
		$20$ & $\geq 5$ & $6
		$ & $1$ & $1$ & $1$\\
		$22$ & $\geq 5$ & $20
		$ & $2$ & $2$ & $2$\\
		$24$ & $\geq 5$ & $38
		$ & $0$ & $0$ & $0$\\
		$26$ & $\geq 5$ & $280
		$ & $95$ & $6$ & $5$\\
		$28$ & $\geq 5$ & $2\,900
		$ & $31$& $14$ & $12$\\
		$30$ & $\geq 5$ & $28\,399
		$ & $104$ & $9$ & $9$\\
		$32$ & $\geq 5$ & $293\,059
		$ & $13$ &$11$ & $11$\\
		$34$ & $\geq 5$ & $3\,833\,587
		$ & $31\,198$&$1\,036$ & $936$\\
		$36$ & $\geq 5$ & $60\,167\,732
		$ & $10\,838$ & $3\,849$ & $3008$\\
		\hline
		$38$ & $\geq 6$ & $39$ & $29$ & $20$ & $10$
	\end{tabular}
		\caption{Counts of all snarks, hypohamiltonian snarks, 
		$K_2$-hypohamiltonian snarks, and snarks which are both hypohamiltonian 
		and $K_2$-hypohamiltonian.}
		\label{table:counts_snarks}
\end{table}

In~\cite{GZ18} the first and last author determined all hypohamiltonian snarks 
with at least 38 and at most 44 vertices which can be obtained by taking the 
dot product on two smaller hypohamiltonian snarks. We verified which of these 
hypohamiltonian snarks are $K_2$-hypohamiltonian and the resulting counts 
can be found in Table~\ref{table:counts_hypohamiltonian_snarks}. (We 
needed these graphs for the proof of Theorem~\ref{thm:cubK2hypoham}.)
%\Jan{And most likely also for the upcoming proof of the orders for which 
%K2-hypoham snarks exist?}

\begin{table}[ht!]
	\centering
	\begin{tabular}{c | r | r}
		Order & d.p.\ hypoham. & d.p.\ hypo- and $K_2$-ham.\\
		\hline
%		$10$ & $1$ & $1$\\
%		$12-16$ & $0$ & $0$\\
%		$18$ & $2$ & $0$\\
%		$20$ & $1$ & $1$\\
%		$22$ & $2$ & $2$\\
%		$24$ & $0$ & $0$\\
%		$26$ & $95$ & $5$\\
%		$28$ & $31$ & $12$\\
%		$30$ & $104$ & $9$\\
%		$32$ & $13$ & $11$\\
%		$34$ & $31189$ & $936$\\
%		$36$ & $10838$ & $3008$\\
		$38$ & $51\,431$ & $2\,482$\\
		$40$ & $8\,820$ & $2\,522$\\
		$42$ & $20\,575\,458$ & $450\,886$\\
		$44$ & $8\,242\,146$ & $1\,606\,786$
	\end{tabular}
	\caption{Counts of hypohamiltonian snarks which can be obtained by taking 
	the dot product on two smaller hypohamiltonian snarks (d.p.\ hypoham.) and the number of 
	$K_2$-hamiltonian snarks among them (d.p.\ hypo- and $K_2$-ham.).}
	\label{table:counts_hypohamiltonian_snarks}
\end{table}

%\begin{corollary}
%	There exist $K_2$-hypohamiltonian snarks of orders $10$, $20$, $22$ and 
%	every order $2i$, for $i \in \{13,...,22\}$, but no $K_2$-hypohamiltonian 
%	snarks of orders $12, 14, 16, 18$ or $24$.
%\end{corollary}

%\Jarne{Maybe mention here the dodecahedron operation does not preserve 
%snarks.}
%
%\Jan{TODO: Write down Jarne's operation for snarks and use it to determine 
%all 
%orders for which $K_2$-hamiltonian snarks exist (Jarne will do that when he 
%has 
%the token). If that can be done, we will probably drop or rephrase the above 
%corollary.}

The technique from Lemma~\ref{lem:ext5cycle} does not preserve the chromatic 
index of a graph. We will therefore use the dot product to create larger graphs 
whilst preserving $3$-regularity, chromatic index, and $K_2$-hypohamiltonicity 
by imposing constraints on the smaller graphs. We recall the definition of 
the dot product as was given by the last author in \cite{Za21}.

Let $G$ and $H$ be disjoint graphs on at least six vertices. For independent 
edges $ab, cd$ in $G$ and adjacent cubic vertices $x$ and $y$ in $H$, consider 
$G' := G - ab - cd$ and $H' := H - x - y$, and let $a', b'$ be the neighbours 
of 
$x$ in $H - y$ and $c', d'$ be the neighbours of $y$ in $H - x$. Then the 
\textit{dot product} $G\cdot H$ is defined as the graph 
$$(V(G)\cup V(H'), E(G')\cup E(H')\cup \{aa',bb',cc',dd'\}).$$

A similar result concerning $K_2$-hypohamiltonian graphs and the dot product 
can already be found in \cite{Za21}. As mentioned by the author in \cite{Za21}, 
this result is 
quite impractical due to the many requirements imposed on $G$ and $H$. It also 
contained certain inaccuracies, which we rectify in Section~\ref{sect:last}. 
The 
following theorem gives different constraints on $G$ and $H$, which make 
it more practical to use.

In a graph $G$, consider pairwise distinct vertices $s_1, s_2, t_1, t_2$. We 
will refer to a 
pair of paths $(\mathfrak{p},\mathfrak{q})$ as \textit{disjoint spanning
$s_1t_1$- and 
$s_2t_2$-paths} if $\mathfrak{p}$ is an $s_1t_1$-path, $\mathfrak{q}$ is an 
$s_2t_2$-path, $\mathfrak{p}$ and $\mathfrak{q}$ are disjoint, and 
each vertex of $G$ lies in either $\mathfrak{p}$ or $\mathfrak{q}$.

\begin{theorem}\label{thm:dotProduct}
	Let $G$ and $H$ be disjoint non-hamiltonian graphs with $a,b,c,d\in V(G)$ and 
	$a'$, $b'$, $c'$, $d'$,
	$x$, $y\in V(H)$ as introduced in the definition of the dot product---in particular, $x$ and $y$ 
	are cubic--- $a,b\not\in N_G(c)\cup N_G(d)$, $a'\not \in N_H(b')$, and 
	$c'\not\in N_H(d')$. The graph $G\cdot H$ is non-hamiltonian and 
	$K_2$-hamiltonian if
	\begin{enumerate}[label=\normalfont(\roman*)]
		\item for any $vw \in E(G')$ there exists in $G - v 
		- w$ a hamiltonian $ab$-path not containing 
		$cd$ or a hamiltonian $cd$-path not containing $ab$ or disjoint 
		spanning $s_1t_1$- and $s_2t_2$-paths containing neither 
		$ab$, nor 
		$cd$, where $\{s_1, 
		s_2\} = \{a,b\}$ 
		and  $\{t_1, t_2\} = \{c,d\}$;
%		 \Jarne{Is it clear that this means $s_1$ 
%		and $s_2$ are different elements of $\{a,b\}$?}
		\item for any $s\in \{a,b\}$ and $t\in\{c,d\}$, the graph $G$ admits a 
		hamiltonian $st$-path, containing neither $ab$, nor $cd$ and $G$ admits 
		disjoint spanning $ab$- and $cd$-paths, containing neither $ab$, nor 
		$cd$;
		\item $G - a$ and $G - b$ contain a hamiltonian cycle through $cd$, and 
		$G - c$ and $G - d$ contain a hamiltonian cycle through $ab$;
		\item $H - x$ and $H - y$ are hamiltonian and there exist in $H'$ 
		disjoint $a'b'$- 
		and 
		$c'd'$-paths spanning $V(H')$;
		\item for any $vw\in E(H)$ with $v,w\not\in \{x,y\}$ there exists in $H 
		- x - y - v - w$ a hamiltonian $s't'$-path with $s'\in \{a',b'\}$ and 
		$t'\in \{c', d'\}$ or disjoint spanning $s'_1t'_1$- and 
		$s'_2t'_2-$paths with 
		$\{s'_1,s'_2\} = \{a',b'\}$ and $\{t'_1,t'_2\} = \{c',d'\}$; and 
		\item $H - x - a'$, $H - x - b'$, $H - y - c'$, and $H - y - d'$ are 
		hamiltonian.
	\end{enumerate}
\end{theorem}
\begin{proof}
	Since $G$ and $H$ are non-hamiltonian their dot product will also be 
	non-hamiltonian. This was shown for example in \cite{Za21}. 
%	\Jarne{As 
%	there is some history behind the dot product, perhaps we could mention an 
%	earlier 
%	source (if it exists) or is this fine?}
	 We show $G\cdot 
	H$ is $K_2$-hamiltonian. Consider $vw\in E(G')$. Suppose that by 
	(i) there 
	exists a hamiltonian $ab$-path in $G - v - w$ not containing $cd$, call it 
	$\mathfrak{p}$. By (iv), $H - x - y$ contains a hamiltonian $a'b'$-path 
	$\mathfrak{q}$. Now $\mathfrak{p}\cup \mathfrak{q}$  is a 
	hamiltonian cycle 
	in $G\cdot H - v - w$. The proof if $G - v - w$ has a hamiltonian 
	$cd$-path not containing $ab$ is analogous. For the final case of 
	(i), assume that $G - v - w$ has disjoint spanning $s_1t_1$- and 
	$s_2t_2$-paths 
	$\mathfrak{p}$ and $\mathfrak{q}$ not containing $ab$ and $cd$, where 
	$\{s_1, 
	s_2\} = \{a,b\}$ 
	and  $\{t_1, t_2\} = \{c,d\}$. By (iv) 
	there 
	exist disjoint spanning 
	$ab$- and $cd$-paths $\mathfrak{p}'$ and $\mathfrak{q}'$ in $H - x - y$. 
	Then $\mathfrak{p}\cup \mathfrak{p}'\cup\mathfrak{q}\cup\mathfrak{q}'$ is a 
	hamiltonian cycle in $G\cdot H - v - w$. 
	
%	\Jarne{From here on out the proof for $K_2$-hamiltonicity is the same as in 
%	the previous paper.}
	By (iii) there is a hamiltonian $cd$-path $\mathfrak{p}$ in $G - a$, and 
	(vi) implies that there is a hamiltonian $c'd'$-path $\mathfrak{q}$ in $H - 
	x - y - a'$. Hence, $\mathfrak{p}\cup\mathfrak{q}$ is a 
	hamiltonian cycle in $G\cdot H - a - a'$. The cases when removing $b, b'$ 
	or $c, c'$ or $d, d'$ can be dealt with in a similar way. 
	
	Consider $vw\in E(H)$ with $v,w\not\in \{x,y\}$. Suppose we are in the 
	first case of (v) and there exists in $H 
	- x - y - v - w$ a hamiltonian $s't'$-path $\mathfrak{p}$ with $s'\in 
	\{a',b'\}$ and $t'\in \{c',d'\}$. Let $N_{G\cdot H}(s')\cap V(G) = \{s\}$ 
	and $N_{G\cdot H}(t')\cap V(G) = \{t\}$. By (ii) there exists a hamiltonian 
	$st$-path containing neither $ab$, nor $cd$, which together with 
	$\mathfrak{p}$ forms a hamiltonian cycle in $G\cdot H - v - w$. Finally, 
	suppose we are in the second case of (v) and that $H - x - y - v - w$ 
	contains disjoint spanning 
	$s'_1t'_1$- and $s'_2t'_2$-paths $\mathfrak{p}'$ and $\mathfrak{q}'$, where 
	$\{s'_1, 
	s'_2\} = \{a',b'\}$ 
	and  $\{t'_1, t'_2\} = \{c',d'\}$. By (ii) there exist disjoint spanning 
	$ab$- and $cd$-paths $\mathfrak{p}$ and $\mathfrak{q}$ containing neither 
	$ab$, nor $cd$ in $G$. Then $\mathfrak{p}\cup \mathfrak{p}'\cup 
	\mathfrak{q}\cup\mathfrak{q}'$ is a hamiltonian cycle in $G\cdot H - v - 
	w$.
\end{proof}

It was proven by Isaacs in \cite{Is75} that the dot product of two snarks is 
again a snark. Hence, the dot product of two $K_2$-hypohamiltonian snarks $G$ 
and $H$ satisfying the properties of the previous theorem is a 
$K_2$-hypohamiltonian snark. The following lemma will help us iteratively 
apply the dot product to $K_2$-hypohamiltonian snarks.

\begin{lemma}\label{lem: iterative_dot_product}
	Let $G$ and $H$ be defined as in Theorem \ref{thm:dotProduct} such that 
	$a,b,c,d\in V(G)$ satisfy (i)--(iii) and $x,y\in V(H)$ satisfy 
	(iv)--(vi). Let 
	$x_Gy_G\in E(G)$ such that $x_G$ and $y_G$ are cubic vertices and such that 
	$N_G[x_G]\cup N_G[y_G]$ does not intersect 
	$\{a,b,c,d\}$. Let $a'_G, b'_G$ be the neighbours of $x_G$ not equal to $y_G$ 
	and let $c'_G, d'_G$ be the neighbours of $y_G$ not equal to $x_G$.	If 
	\begin{itemize}
		\item $G - x_G$ contains a hamiltonian $ab$-path not containing $cd$ or 
		a hamiltonian $cd$-path not containing $ab$;
		\item $G - y_G$ contains a hamiltonian $ab$-path not containing $cd$ or 
		a hamiltonian $cd$-path not containing $ab$; and
		\item $G$ contains disjoint spanning $ab$- and $cd$-paths,
		one containing $x_G$ and the other $y_G$,
	\end{itemize}
	then $x_G,y_G\in V(G\cdot H)$ satisfy properties (iv)--(vi).
\end{lemma}
\begin{proof}
	We first show (iv). Suppose $G - x_G$ contains a hamiltonian $ab$-path 
	$\mathfrak{p}$ not containing $cd$. Since $H$ satisfies (iv), there is a 
	hamiltonian $a'b'$-path $\mathfrak{q}$ in $H - x - y$. Then 
	$\mathfrak{p}\cup \mathfrak{q}$ is a hamiltonian cycle in $G\cdot H - x_G$. 
	If $G - x_G$ contains a hamiltonian $cd$-path not containing $ab$, a 
	similar argument shows $G\cdot H -x_G$ is hamiltonian. The same reasoning  
	proves that $G\cdot H - y_G$ is hamiltonian.
	
	Assume without loss of generality that $G$ contains disjoint spanning $ab$- 
	and 
	$cd$- 
	paths $\mathfrak{p}$ and $\mathfrak{q}$ such that $x_G\in 
	V(\mathfrak{p})$ and $y_G\in V(\mathfrak{q})$. Since $H$ satisfies (iv) we 
	get 
	disjoint 
	spanning $a'b'$- and $c'd'$-paths $\mathfrak{p}'$ and $\mathfrak{q}'$. Then 
	$G\cdot H$ gets spanned by two disjoint cycles $\mathfrak{p}\cup 
	\mathfrak{p}'$ and $\mathfrak{q}\cup \mathfrak{q}'$, the former containing 
	$x_G$ and the latter containing $y_G$. Hence, we get disjoint spanning 
	$a'_Gb'_G$- and $c'_Gd'_G$-paths in $G\cdot H - x_G - y_G$. This shows 
	$x_G,y_G\in V(G\cdot H)$ satisfy (iv).
	
	We show (v). Let $vw\in E(G\cdot H)$ with $v,w\not \in \{x,y\}$. Suppose 
	$G\cdot H - v - w$ contains a 
	hamiltonian cycle. If this cycle contains $x_Gy_G$, then $G\cdot H - x_G - 
	y_G - v - w$ admits a hamiltonian $s't'$-path with $s'\in \{a'_G,b'_G\}$ 
	and $t'\in \{c'_G, d'_G\}$. If the cycle does not contain $x_Gy_G$, it 
	contains $a'_Gx_Gb'_G$ and $c'_Gy_Gd'_G$, which in $G\cdot H - x_G - y_G - 
	v - w$ leads to disjoint spanning $s'_1t'_1$- and $s'_2t'_2$-paths with 
	$\{s'_1,s'_2\} = \{a'_G,b'_G\}$ and $\{t'_1,t'_2\} = \{c'_G,d'_G\}$. We 
	need to show that $G\cdot H - v - w$ contains a hamiltonian cycle.
	
	Suppose 
	that $vw\in E(G)$ and assume that there exists in $G - v - w$ a hamiltonian 
	$ab$-path $\mathfrak{p}$ not containing $cd$. Since $H$ satisfies (iv) 
	there is a hamiltonian $a'b'$-path $\mathfrak{q}$ in $H -x - y$. Then 
	$\mathfrak{p}\cup\mathfrak{q}$ is a hamiltonian cycle in $G\cdot H - v - 
	w$. Similarly, if there exists in $G - v - w$ a hamiltonian $cd$-path not 
	containing $ab$, there is a hamiltonian cycle in $G\cdot H - v - w$. 
	Otherwise, suppose $G - v - w$ contains disjoint spanning $s_1t_1$- and 
	$s_2t_2$-paths $\mathfrak{p}$ and $\mathfrak{q}$ containing neither $ab$, 
	nor $cd$, where $\{s_1,s_2\} = \{a,b\}$ and $\{t_1,t_2\} = \{c,d\}$. Since 
	$H$ satisfies (iv) there are disjoint spanning $a'b'$- and $c'd'$-paths 
	$\mathfrak{p}'$ and $\mathfrak{q}'$. Then, 
	$\mathfrak{p}\cup\mathfrak{p}'\cup\mathfrak{q}\cup\mathfrak{q}'$ is a 
	hamiltonian cycle in $G\cdot H - v - w$. 
	
	Suppose that $vw\in E(H)$ and assume that there exists in $H - x - y - v - 
	w$ a hamiltonian $s't'$-path $\mathfrak{q}$ with $s'\in \{a',b'\}$ and 
	$t'\in\{c',d'\}$. Since $G$ satisfies (ii), there is a hamiltonian 
	$st$-path $\mathfrak{p}$ in $G$ with $s \in N_{G\cdot H}(s')\cap V(G)$ and 
	$t\in N_{G\cdot H}(t')\cap V(G)$. We get a hamiltonian cycle 
	$\mathfrak{p}\cup \mathfrak{q}$ in $G\cdot H - v - w$. If $H - x - y - v - 
	w$ contains disjoint spanning $s'_1t'_1$- and $s'_2t'_2$-paths 
	$\mathfrak{p}'$ and $\mathfrak{q}'$, with $\{s'_1,s'_2\} = \{a',b'\}$ and 
	$\{t'_1,t'_2\} = \{c',d'\}$. Since $G$ satisfies (ii) it admits disjoint 
	spanning $ab$- and $cd$-paths $\mathfrak{p}$ and $\mathfrak{q}$, containing 
	neither $ab$, nor $cd$. Hence, $\mathfrak{p}\cup \mathfrak{p}'\cup 
	\mathfrak{q}\cup\mathfrak{q}'$ is a hamiltonian cycle in $G\cdot H  - v - 
	w$.
	
	Finally, assume without loss of generality that $vw = aa'$. Since $G$ 
	satisfies (ii), there is a hamiltonian $cd$-path $\mathfrak{p}$ in $G - a$. 
	Since $H$ satisfies (vi) there is a hamiltonian $c'd'$-path $\mathfrak{p}'$ 
	in $H - x - y - a'$. Then $\mathfrak{p}\cup \mathfrak{p}'$ is a hamiltonian 
	cycle in $G\cdot H - a - a'$. The same argument works for $bb'$, $cc'$, and 
	$dd'$. Hence, $x_G,y_G\in V(G\cdot H)$ satisfy (v).
	
	Since $G\cdot H$ is $K_2$-hamiltonian by Theorem~\ref{thm:dotProduct} 
	condition (vi) 
	is also satisfied.
\end{proof}

We can now classify for which orders there exist $K_2$-hypohamiltonian snarks.
\begin{theorem}
	There exists a $K_2$-hypohamiltonian snark of order $n$ if and only if 
	$n\in \{10,20,22\}$ or $n\geq 26$ is even.
\end{theorem}
\begin{proof}
	It follows from the counts in Table~\ref{table:counts_snarks}, that there 
	exists no $K_2$-hypohamiltonian snark of order $12$, $14$, $16$, $18$, and 
	$24$ and that the Petersen graph is a $K_2$-hypohamiltonian snark of order 
	$10$. Using our computer program, we verified that there 
	exist 
	$K_2$-hypohamiltonian snarks of order $n\in 
	\{20,22,26,28,30,32,34,36,44\}$ 
	satisfying (iv)--(vi) for some pair of adjacent 
	vertices. Recall that we found $K_2$-hypohamiltonian snarks of order $44$,
	see Table~\ref{table:counts_hypohamiltonian_snarks}. In particular, the 
	flower 
	snark $J_5$	satisfies these conditions, as 
	well as (i)--(iii), and the conditions of 
	Lemma~\ref{lem: iterative_dot_product} for some pair of independent edges 
	and pair of adjacent vertices. In the GitHub repository \cite{Re22}, we
	attached a file containing certificates which prove the claims above for 
	a 
	$K_2$-hypohamiltonian snark of each of these orders so 
	that one can verify them by hand.
	
	Let $G$ be the flower snark $J_5$ and $H$ be a graph satisfying (iv)--(vi) 
	for 
	some pair of adjacent vertices. By Lemma~\ref{lem: iterative_dot_product}, 
	$G\cdot H$ satisfies conditions (iv)--(vi) and has order $\lvert V(G)\rvert 
	+ 
	\lvert V(H)\rvert-2 = \lvert V(H)\rvert+18$. We can then take the dot 
	product 
	with $G$ in this way ad infinitum. Taking the dot product of the flower 
	snark 
	$J_5$ with the graphs mentioned above, we can find a 
	$K_2$-hypohamiltonian 
	snark for all even orders $n\geq 26$.
\end{proof}

Note that the above also gives us an alternative proof for 
Theorem~\ref{thm:cubK2hypoham}.
%One can easily verify that the properties of $H$ 
%are verified by Petersen's graph for any choice of $x$ and $y$. 
%\Jarne{This is very promising, but there is still something missing here. In 
%order to iteratively apply this procedure we need to prove that there is a 
%pair 
%of independent edges in $G\cdot H$ satisfying the first three properties. 
%Empirically I have found that all $K_2$-hypohamiltonian snarks except 
%Petersen's satisfy these first three properties up to and including $36$ 
%vertices:
%Checked 5027 graphs in 26.665272 seconds: 5027 satisfy the constraints for G.\\
%1 graphs : n=20\\
%2 graphs : n=22\\
%6 graphs : n=26\\
%14 graphs : n=28\\
%9 graphs : n=30\\
%11 graphs : n=32\\
%1036 graphs : n=34\\
%3849 graphs : n=36\\
%99 graphs : n=38\\
%5027 graphs altogether; cpu=0.01 sec
%
%I have also generated all possible dot products of these graphs up to and 
%including $30$ vertices with Petersen's graph and every single result 
%satisfies 
%the first three properties for a lot of pairs of independent edges. Hence, I 
%am 
%convinced this is true, but finding a proof is not straightforward.
%
%Some properties: an independent edge pair in $G$ which satisfies properties 
%(i)-(iii) does not need to satisfy those properties in $G\cdot H$. No 
%combination of edges $aa'$, $bb'$, $cc'$, $dd'$ is a good edge pair. Good 
%candidates seem an edge in $G$ and an edge in $H$, but not all such edge pairs 
%satisfy the properties in general. Edges which are adjacent to $ab$ or $cd$ 
%still often appear in good edge pairs.}

\subsection{The planar case}
\label{subsect:planar}

%\Jan{I reshuffled this section so we can already refer to the computational results in the proof of Theorem~\ref{thm:planar_orders}. Feel free to undo this or change it in some better order.}

In~\cite{HS93} Holton and Sheehan asked for the order of the smallest planar hypohamiltonian graph. This problem remains open, but we know that the answer is at least~23 and at most~40, see~\cite{GZ17} and~\cite{JMOPZ17}, respectively. The latter paper also proves that there exist planar hypohamiltonian graphs of order $n$ for every $n \ge 42$
%---perhaps planar hypohamiltonian graphs have a Holton gap at 41--- %Due to the reshuffling we have not defined Holton gap yet at this point
and that the smallest planar hypohamiltonian graph of girth~5 has order~45. It is natural to raise the same questions for $K_2$-hypohamiltonian graphs.

%\Jan{TODO: Maybe first mention that a $K_2$-hypohamiltonian must not only be 
%3-connected but also be cyclically 4-edge-connected?}
Using the program \verb|plantri|~\cite{BM07} we generated all  cyclically 
4-edge-connected planar graphs of girth 5 up to 48 vertices, as well as an 
incomplete 
sample of such graphs for larger orders, and verified which of them are 
$K_2$-hypohamiltonian using the program mentioned at the beginning of 
Section~\ref{sect:orders}. The results are summarised in 
Table~\ref{table:counts_planar}. Note that a $K_2$-hypohamiltonian graph is 
always cyclically $4$-edge-connected \cite{Za21}.
%\Jarne{Is the reference 
%alright 
%here? Or should I put that it is easy to see?}.
As another correctness test, we stored all non-hamiltonian graphs obtained from 
our computation with \verb|plantri|, confirmed their 
non-hamiltonicity, and tested which of them are $K_2$-hamiltonian using an 
independent program.
%All non-hamiltonian graphs found in this way
%were saved %\Carol{What does the information that these were ``saved'' supposed to convey to the reader?} \Jarne{In my opinion it is necessary since we only double checked the $K_2$-hypohamiltonicity of the graphs which were found to be non-hamiltonian by the first program. We did not double check the hamiltonicity of graphs which were not found to be non-hamiltonian. I am not worried about the correctness of the result, but simply saying this was independently verified would in my opinion be a bit misleading. But maybe this can be stated in a different way.}
%and their non-hamiltonicity and $K_2$-hamiltonicity was independently verified by another program.	
%%\Jan{TODO: say something about the independent verification. I stored the
%%non-ham graphs and will give them to Jarne so he can independently verify
%%$K_2$-hypohamilonicity using his independent program.}
The $K_2$-hypohamiltonian graphs from Table~\ref{table:counts_planar} can be 
obtained from the database of interesting graphs from the \textit{House of 
Graphs}~\cite{CDG23} by searching for the keywords ``planar 
K2-hypohamiltonian''.
%\Jan{TODO: I still need to upload these graphs}
The unique planar $K_2$-hypohamiltonian graph of girth 5 on 48 vertices is the graph $G_{48}$ from Figure~\ref{fig:48v}. We also determined that only two of the $K_2$-hypohamiltonian graphs from Table~\ref{table:counts_planar} are hypohamiltonian: next to $G_{48}$, the
graph from Figure~\ref{fig:52v} is also hypohamiltonian.

%\Jan{Just a note for ourselves: I did the computation only for girth 5 since 
%for girth $<5$ this approach is not really suitable with plantri and one 
%cannot 
%very far.}

\begin{table}[ht!]
\centering
\begin{tabular}{ l | ccccccc}
%\hline
Order 		& $\leq 47$ & 48 & 49 & 50 & 51 & 52 & 53 \\
\hline
Number of graphs & 0 & 1 & $\geq 9$ & $\geq 11$ & $\geq 0$ & $\geq 6$ &
$\geq 9$\\
\end{tabular}
\caption{Counts of $K_2$-hypohamiltonian graphs among planar graphs of
girth 5.}
\label{table:counts_planar}
\end{table}

\begin{figure}[!htb]
\begin{center}
\begin{tikzpicture}[scale=0.05]
    \definecolor{marked}{rgb}{0.25,0.5,0.25}
    \node [circle,fill,scale=0.25] (52) at (38.356969,49.427872) {};
    \node [circle,fill,scale=0.25] (51) at (47.158925,57.114914) {};
    \node [circle,fill,scale=0.25] (50) at (41.770171,39.227383) {};
    \node [circle,fill,scale=0.25] (49) at (32.762836,49.760391) {};
    \node [circle,fill,scale=0.25] (48) at (41.271394,59.432763) {};
    \node [circle,fill,scale=0.25] (47) at (52.498778,61.916870) {};
    \node [circle,fill,scale=0.25] (46) at (50.493888,48.264058) {};
    \node [circle,fill,scale=0.25] (45) at (48.371639,40.430317) {};
    \node [circle,fill,scale=0.25] (44) at (39.344744,29.760391) {};
    \node [circle,fill,scale=0.25] (43) at (29.193154,40.743276) {};
    \node [circle,fill,scale=0.25] (42) at (21.721271,44.586797) {};
    \node [circle,fill,scale=0.25] (41) at (21.711492,55.295843) {};
    \node [circle,fill,scale=0.25] (40) at (29.124695,58.982885) {};
    \node [circle,fill,scale=0.25] (39) at (39.061125,69.907091) {};
    \node [circle,fill,scale=0.25] (38) at (52.938876,70.083129) {};
    \node [circle,fill,scale=0.25] (37) at (59.872862,54.513447) {};
    \node [circle,fill,scale=0.25] (36) at (58.679707,47.784841) {};
    \node [circle,fill,scale=0.25] (35) at (55.266504,35.921760) {};
    \node [circle,fill,scale=0.25] (34) at (55.051345,29.046455) {};
    \node [circle,fill,scale=0.25] (33) at (45.486553,23.911980) {};
    \node [circle,fill,scale=0.25] (32) at (31.911981,30.308068) {};
    \node [circle,fill,scale=0.25] (31) at (17.144255,38.024450) {};
    \node [circle,fill,scale=0.25] (30) at (17.232275,62.053789) {};
    \node [circle,fill,scale=0.25] (29) at (31.853302,69.486552) {};
    \node [circle,fill,scale=0.25] (28) at (45.330074,75.911979) {};
    \node [circle,fill,scale=0.25] (27) at (61.985331,66.464548) {};
    \node [circle,fill,scale=0.25] (26) at (66.797067,54.190709) {};
    \node [circle,fill,scale=0.25] (25) at (64.039121,41.232274) {};
    \node [circle,fill,scale=0.25] (24) at (65.535453,29.408313) {};
    \node [circle,fill,scale=0.25] (23) at (45.877751,16.244499) {};
    \node [circle,fill,scale=0.25] (22) at (25.506113,25.633252) {};
    \node [circle,fill,scale=0.25] (21) at (8.753056,50.073349) {};
    \node [circle,fill,scale=0.25] (20) at (25.682151,74.415648) {};
    \node [circle,fill,scale=0.25] (19) at (46.161370,83.559901) {};
    \node [circle,fill,scale=0.25] (18) at (66.278729,70.063570) {};
    \node [circle,fill,scale=0.25] (17) at (75.383863,58.112469) {};
    \node [circle,fill,scale=0.25] (16) at (74.444988,42.249388) {};
    \node [circle,fill,scale=0.25] (15) at (61.271395,18.366748) {};
    \node [circle,fill,scale=0.25] (14) at (30.239609,14.210270) {};
    \node [circle,fill,scale=0.25] (13) at (0.000000,50.004890) {};
    \node [circle,fill,scale=0.25] (12) at (30.435208,85.819070) {};
    \node [circle,fill,scale=0.25] (11) at (61.056235,81.515893) {};
    \node [circle,fill,scale=0.25] (10) at (80.332518,64.674816) {};
    \node [circle,fill,scale=0.25] (9) at (80.000001,35.168704) {};
    \node [circle,fill,scale=0.25] (8) at (71.119805,15.843520) {};
    \node [circle,fill,scale=0.25] (7) at (24.997555,6.699266) {};
    \node [circle,fill,scale=0.25] (6) at (24.997555,93.300733) {};
    \node [circle,fill,scale=0.25] (5) at (71.158924,84.185820) {};
    \node [circle,fill,scale=0.25] (4) at (90.004891,49.936430) {};
    \node [circle,fill,scale=0.25] (3) at (75.002446,6.699266) {};
    \node [circle,fill,scale=0.25] (2) at (75.002446,93.300733) {};
    \node [circle,fill,scale=0.25] (1) at (99.999999,50.004890) {};
    \draw [black] (52) to (48);
    \draw [black] (52) to (50);
    \draw [black] (52) to (49);
    \draw [black] (51) to (47);
    \draw [black] (51) to (46);
    \draw [black] (51) to (48);
    \draw [black] (50) to (44);
    \draw [black] (50) to (45);
    \draw [black] (49) to (43);
    \draw [black] (49) to (40);
    \draw [black] (48) to (39);
    \draw [black] (47) to (38);
    \draw [black] (47) to (37);
    \draw [black] (46) to (36);
    \draw [black] (46) to (45);
    \draw [black] (45) to (35);
    \draw [black] (44) to (32);
    \draw [black] (44) to (33);
    \draw [black] (43) to (42);
    \draw [black] (43) to (32);
    \draw [black] (42) to (31);
    \draw [black] (42) to (41);
    \draw [black] (41) to (30);
    \draw [black] (41) to (40);
    \draw [black] (40) to (29);
    \draw [black] (39) to (28);
    \draw [black] (39) to (29);
    \draw [black] (38) to (27);
    \draw [black] (38) to (28);
    \draw [black] (37) to (36);
    \draw [black] (37) to (26);
    \draw [black] (36) to (25);
    \draw [black] (35) to (34);
    \draw [black] (35) to (25);
    \draw [black] (34) to (24);
    \draw [black] (34) to (33);
    \draw [black] (33) to (23);
    \draw [black] (32) to (22);
    \draw [black] (31) to (21);
    \draw [black] (31) to (22);
    \draw [black] (30) to (20);
    \draw [black] (30) to (21);
    \draw [black] (29) to (20);
    \draw [black] (28) to (19);
    \draw [black] (27) to (18);
    \draw [black] (27) to (26);
    \draw [black] (26) to (16);
    \draw [black] (25) to (16);
    \draw [black] (24) to (15);
    \draw [black] (24) to (16);
    \draw [black] (23) to (14);
    \draw [black] (23) to (15);
    \draw [black] (22) to (14);
    \draw [black] (21) to (13);
    \draw [black] (20) to (12);
    \draw [black] (19) to (11);
    \draw [black] (19) to (12);
    \draw [black] (18) to (17);
    \draw [black] (18) to (11);
    \draw [black] (17) to (10);
    \draw [black] (17) to (16);
    \draw [black] (16) to (9);
    \draw [black] (15) to (8);
    \draw [black] (14) to (7);
    \draw [black] (13) to (6);
    \draw [black] (13) to (7);
    \draw [black] (12) to (6);
    \draw [black] (11) to (5);
    \draw [black] (10) to (4);
    \draw [black] (10) to (5);
    \draw [black] (9) to (8);
    \draw [black] (9) to (4);
    \draw [black] (8) to (3);
    \draw [black] (7) to (3);
    \draw [black] (6) to (2);
    \draw [black] (5) to (2);
    \draw [black] (4) to (1);
    \draw [black] (3) to (1);
    \draw [black] (2) to (1);
\end{tikzpicture}
\caption{A planar hypohamiltonian $K_2$-hamiltonian graph of girth 5 on 52 vertices.}
\label{fig:52v}
\end{center}
\end{figure}
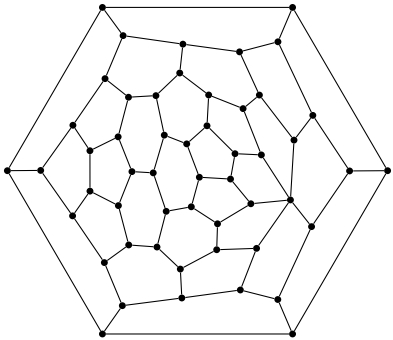

%\Jan{Or possibly use an Observation instead of Theorem?}

%\Jarne{I have verified that there does not exist a $3$-connected planar graph 
%which is $K_2$-hypohamiltonian on $16$-vertices or less in order to give a 
%lower bound. We can probably go up to $18$.}

Using \verb|plantri| we also generated all $3$-connected planar graphs up to 
and including $17$ vertices. We verified using our program that none of these 
are $K_2$-hypohamiltonian. As $K_2$-hamiltonian graphs are $3$-connected, this 
gives us a lower bound for the order of the smallest planar 
$K_2$-hypohamiltonian 
graph.

%\noindent \textbf{Theorem 4.}
\begin{observation}
The following statements hold for planar $K_2$-hypohamiltonian graphs.
\begin{enumerate}
\item[(i)] There exist no planar $K_2$-hypohamiltonian graphs on at most $17$ 
vertices. 
%\Jan{Jarne, did you obtain this lower bound using plantri? Please 
%mention in the text how you obtained the lower bound.}
\item[(ii)] The smallest planar $K_2$-hypohamiltonian graph of girth~$5$ has order $48$. There is exactly one such graph of that order: the graph $G_{48}$ from Figure~\ref{fig:48v}.
%\item[(i)] \emph{There exists exactly one planar $K_2$-hypohamiltonian graph of girth~$5$ and order $48$, and no such graph of smaller order.}
\item[(iii)] There exist planar $K_2$-hypohamiltonian graphs of order $48, 49, 50, 52, 53$. 
%\item[(iii)] There exist planar $K_2$-hypohamiltonian graphs of order $48, 49, 50, 52, 53,$ $68, 70, 74, 76$, and $78$. 
%\Jan{I am not sure if we should mention the orders 68 and above here as these are cubic graphs... If we really want to mention it here, we should refer to the next subsection about cubic planar graphs. If it would be really important to fill gaps, I think it is likely that we can construct planar $K_2$-hypohamiltonian graphs of order 54 and slightly higher.}
\item[(iv)] There exist planar hypohamiltonian $K_2$-hamiltonian graphs of order $48$ and $52$.
\end{enumerate}
\end{observation}

Holton and Sheehan~\cite{HS93} asked whether there is a number $n_0$ such that 
for every integer $n \ge n_0$ there exists a planar hypohamiltonian graph on 
$n$ vertices. This was solved, affirmatively, by Araya and the third 
author~\cite{WA11}. In a similar vein, Holton~\cite{Ho93} writes 
``\emph{Suppose the smallest planar hypohamiltonian graph had $n_0$ vertices. 
Does there exist a planar hypohamiltonian graph on $n$ vertices for all $n > 
n_0$ or are there ``gaps''?}'' For an infinite graph family ${\cal F}$ we call 
such a gap a \emph{Holton gap}; more specifically, if there exist integers 
$a,b,c$ with $a < b < c$ such that there exist members of ${\cal F}$ of order 
$a$ and $c$, but not of $b$, we say that ${\cal F}$ has a \emph{Holton gap at 
$b$}.

Holton's problem has not yet been settled---in particular, we do not know whether planar hypohamiltonian graphs have a Holton gap---but substantial progress has been made; for the state-of-the-art, see~\cite{JMOPZ17}. We recall that therein it is shown that there exist planar hypohamiltonian graphs of order 40 as well as of order $n$ for every $n \ge 42$, but it remains an open question whether planar hypohamiltonian graphs have a Holton gap at 41. We now discuss $K_2$-analogues of these questions.

%\Jarne{I moved the lemma from \cite{Za21} to \ref{subsubsect:cubic}.}

\begin{theorem}\label{thm:planar_orders}
For every integer $n \ge 177$ there exists a planar $K_2$-hypohamiltonian graph of order~$n$.
\end{theorem}
\begin{proof}

In Figure~\ref{fig:planar_graphs}, we present planar $K_2$-hypohamiltonian 
graphs of order $n$ for every $n \in \{48, 49, 50, 52, 53 \}$. By 
Lemma~\ref{lem:carol}, these five graphs yield the existence of planar 
$K_2$-hypohamiltonian graphs of order $n$ for every $n \in \{ 177, \ldots, 196 
\}$.
%\Carol{It remains to show the ``${\rm exc}(G_i) \cap X_i = \emptyset$''
%condition; for the 48- and 52-vertex case this is trivial as these are
%hypohamiltonian, but for the other three this must be checked.}

%\begin{claim}
	Applying Lemma~\ref{lem:carol} a suitable number of times to the graphs 
	in 
	Figure~\ref{fig:planar_graphs}, we get a planar 
	$K_2$-hypohamiltonian graph of order $n$ 
	for every $n \in \{177, \ldots, 196\}$. We show that by choosing 
	appropriate $3$-cuts these graphs contain an extendable $5$-cycle.
%\end{claim}
%\Jan{I would recommend to use the claim environment to stay consistent with 
%the style of the other claims.}
%\Jan{``We can do this'' sounds a bit vague to me. Maybe better something 
%like: 
%``We can apply Lemma~\ref{lem:carol} to a planar $K_2$-hypohamiltonian 
%planar 
%graph of order $n$ for every $n \in \{48, 49, 50, 52, 53 \}$ in such a way 
%that 
%each of the resulting graphs contains an extendable cycle''? By the way: 
%where 
%does this twenty come from? Is it an artefact from an old version?}
%\Jarne{I changed the claim a little, but maybe we should just remove the 
%claim 
%alltogether and insert it in the proof as this claim takes up almost all of 
%the 
%proof. In theory we only need $20$ $K_2$-hypohamiltonian graphs with an 
%extendable cycle of consecutive orders modulo $20$, but having $20$ actually 
%consecutive orders of such graphs
%allows us to give an easy argument which shows that all following orders 
%contain a $K_2$-hypohamiltonian graph (by 
%inserting a dodecahedron which gives $20$ extra vertices.)}
%
%\Gabor{I agree with Jarne.}

%\smallskip

Let $G_1$ and $G_2$ be the graphs to which we apply the operation of 
Lemma~\ref{lem:carol}. For $i\in\{1,2\}$, let $G_i$ contain a $3$-cut $X_i$
and $X_i$-fragments $F_i$ and $F_i'$ such that the conditions of the lemma are 
satisfied and such that $F_i$ is non-trivial. We show that an extendable 
$5$-cycle $C = v_0\ldots v_4$ in $G_1$ disjoint from $X_1$ leads to an 
extendable $5$-cycle in $(F_1, X_1)\:(F_2, X_2) =: G$. Indeed, for every $i =  
0,\ldots,4$, we find a hamiltonian cycle in $G_1 - v_i$, which leads to a 
hamiltonian path in $F_1$ with endpoints in $X_1$. Since there exists, for each 
$x\in X_2$, a hamiltonian path in $F_2 - x$ between the two vertices of $X_2 - 
x$, it follows that there exists a hamiltonian cycle in $G - v_i$. Similarly, 
we have for every $i = 0,\ldots, 4$ a hamiltonian cycle in $G_1 - v'_i$, 
where $v'_i$
is the neighbour of $v_i$ that does not lie on $C$. Noting that $v'_i$ lies in 
$F_1$, we see in the same way that there exists a hamiltonian cycle in $G - 
v'_i$. Hence, $G$ contains an extendable $5$-cycle. It is easy to see that any 
trivial $3$-cut in $G_1$ which lies in $F_1$ gives rise to a trivial $3$-cut in 
$G$ and by Theorem~6 in~\cite{Za15}, we can infer that a vertex in $G$ is 
exceptional if and only if it was exceptional in $G_1$ or $G_2$. 

Figure~\ref{fig:planar_graphs} depicts $K_2$-hypohamiltonian 
planar graphs of order $48$, $49$, $50$, $52$, and $53$, their exceptional 
vertices, and an extendable $5$-cycle. In the GitHub repository \cite{Re22}, 
we attached a file with these planar $K_2$-hypohamiltonian graphs 
together with an extendable $5$-cycle $C$ and paths which prove that $C$ is 
indeed an extendable $5$-cycle so that this can be verified by hand. We can 
then apply the operation of 
Lemma~\ref{lem:carol} to the graphs of Figure~\ref{fig:planar_graphs} a 
suitable number of times with the appropriate $3$-cuts 
to obtain planar 
$K_2$-hypohamiltonian graphs of order $n$ with an extendable $5$-cycle for 
every $n\in \{177,\ldots,196\}$.

The proof is now completed by applying
Lemma~\ref{lem:ext5cycle} to these twenty graphs.
\end{proof}

\begin{figure}[!htb]
	\centering
	\begin{adjustbox}{minipage=\linewidth,scale=0.96}
		\centering
	\begin{subfigure}{0.3\linewidth}
		\begin{tikzpicture}[scale=0.05]
			\definecolor{marked}{rgb}{0.25,0.5,0.25}
			\node [circle,fill,scale=0.25] (48) at (17.891833,40.590508) {};
			\node [circle,fill,scale=0.25] (47) at (22.958057,23.835542) {};
			\node [circle,fill,scale=0.25] (46) at (9.746136,51.727373) {};
			\node [circle,fill,scale=0.25] (45) at (20.618102,67.014349) {};
			\node [circle,fill,scale=0.25] (44) at (26.953641,63.239515) {};
			\node [circle,fill,scale=0.25] (43) at (24.988962,50.524283) {};
			\node [circle,fill,scale=0.25] (42) at (41.423841,42.820089) {};
			\node [circle,fill,scale=0.25] (41) at (30.143487,33.073952) {};
			\node [circle,fill,scale=0.25] (40) at (24.999998,6.694259) {};
			\node [circle,fill,scale=0.25] (39) at (71.611478,16.032010) {};
			\node [circle,fill,scale=0.25] (38) at (74.999998,6.694259) {};
			\node [circle,fill,scale=0.25] (37) at (0.000000,50.005519) {};
			\node [circle,fill,scale=0.25] (36) at (28.587195,84.089404) {};
			\node [circle,fill,scale=0.25] (35) at (42.097130,82.488962) {};
			\node [circle,fill,scale=0.25] (34) at (36.964680,71.263797) {};
			\node [circle,fill,scale=0.25] (33) at (32.207505,48.570640) {};
			\node [circle,fill,scale=0.25] (32) at (38.421633,46.285872) {};
			\node [circle,fill,scale=0.25] (31) at (48.090507,51.009934) {};
			\node [circle,fill,scale=0.25] (30) at (40.894039,34.122517) {};
			\node [circle,fill,scale=0.25] (29) at (36.004415,25.645694) {};
			\node [circle,fill,scale=0.25] (28) at (40.927152,13.504415) {};
			\node [circle,fill,scale=0.25] (27) at (57.980132,17.544150) {};
			\node [circle,fill,scale=0.25] (26) at (79.426047,33.073952) {};
			\node [circle,fill,scale=0.25] (25) at (90.242824,48.383002) {};
			\node [circle,fill,scale=0.25] (24) at (99.999999,50.005519) {};
			\node [circle,fill,scale=0.25] (23) at (24.999998,93.305740) {};
			\node [circle,fill,scale=0.25] (22) at (59.128034,86.473511) {};
			\node [circle,fill,scale=0.25] (21) at (42.174393,65.921633) {};
			\node [circle,fill,scale=0.25] (20) at (46.997792,61.473509) {};
			\node [circle,fill,scale=0.25] (19) at (51.688741,60.877483) {};
			\node [circle,fill,scale=0.25] (18) at (51.754966,48.945916) {};
			\node [circle,fill,scale=0.25] (17) at (48.134657,39.056291) {};
			\node [circle,fill,scale=0.25] (16) at (45.364238,24.376379) {};
			\node [circle,fill,scale=0.25] (15) at (63.013245,28.747241) {};
			\node [circle,fill,scale=0.25] (14) at (73.013245,36.804635) {};
			\node [circle,fill,scale=0.25] (13) at (82.064017,59.431567) {};
			\node [circle,fill,scale=0.25] (12) at (74.999998,93.305740) {};
			\node [circle,fill,scale=0.25] (11) at (54.448123,75.579471) {};
			\node [circle,fill,scale=0.25] (10) at (58.951434,65.822296) {};
			\node [circle,fill,scale=0.25] (9) at (58.432671,57.124724) {};
			\node [circle,fill,scale=0.25] (8) at (52.759381,38.416115) {};
			\node [circle,fill,scale=0.25] (7) at (57.693156,34.023178) {};
			\node [circle,fill,scale=0.25] (6) at (74.922736,49.486755) {};
			\node [circle,fill,scale=0.25] (5) at (76.964680,76.186533) {};
			\node [circle,fill,scale=0.25] (4) at (63.874171,74.299116) {};
			\node [circle,fill,scale=0.25] (3) at (61.412802,53.581678) {};
			\node [circle,fill,scale=0.25] (2) at (67.671081,51.418322) {};
			\node [circle,fill,scale=0.25] (1) at (69.757173,66.870861) {};
			\draw [black] (48) to (46);
			\draw [black] (48) to (43);
			\draw [black] (48) to (47);
			\draw [black] (47) to (40);
			\draw [black] (47) to (41);
			\draw [black] (46) to (37);
			\draw [black] (46) to (45);
			\draw [black] (45) to (36);
			\draw [black] (45) to (44);
			\draw [black] (44) to (34);
			\draw [black] (44) to (43);
			\draw [black] (43) to (33);
			\draw [black] (42) to (32);
			\draw [black] (42) to (31);
			\draw [black] (42) to (30);
			\draw [black] (41) to (32);
			\draw [black] (41) to (29);
			\draw [black] (41) to (33);
			\draw [black] (40) to (37);
			\draw [black] (40) to (28);
			\draw [black] (40) to (38);
			\draw [black] (39) to (26);
			\draw [black] (39) to (38);
			\draw [black] (39) to (27);
			\draw [black] (38) to (24);
			\draw [black] (37) to (23);
			\draw [black] (36) to (35);
			\draw [black] (36) to (23);
			\draw [black] (35) to (22);
			\draw [black] (35) to (34);
			\draw [black] (34) to (21);
			\draw [black] (33) to (21);
			\draw [black] (32) to (20);
			\draw [black] (31) to (18);
			\draw [black] (31) to (19);
			\draw [black] (30) to (29);
			\draw [black] (30) to (17);
			\draw [black] (29) to (16);
			\draw [black] (28) to (16);
			\draw [black] (28) to (27);
			\draw [black] (27) to (15);
			\draw [black] (26) to (25);
			\draw [black] (26) to (14);
			\draw [black] (25) to (13);
			\draw [black] (25) to (24);
			\draw [black] (24) to (12);
			\draw [black] (23) to (12);
			\draw [black] (22) to (11);
			\draw [black] (22) to (12);
			\draw [black] (21) to (11);
			\draw [black] (20) to (11);
			\draw [black] (20) to (19);
			\draw [black] (19) to (10);
			\draw [black] (18) to (9);
			\draw [black] (18) to (17);
			\draw [black] (17) to (8);
			\draw [black] (16) to (7);
			\draw [black] (16) to (8);
			\draw [black] (15) to (14);
			\draw [black] (15) to (7);
			\draw [black] (14) to (6);
			\draw [black] (13) to (5);
			\draw [black] (13) to (6);
			\draw [black] (12) to (5);
			\draw [black] (11) to (4);
			\draw [black] (10) to (9);
			\draw [black] (10) to (4);
			\draw [black] (9) to (3);
			\draw [black] (8) to (3);
			\draw [black] (7) to (2);
			\draw [black] (6) to (2);
			\draw [black] (5) to (1);
			\draw [black] (4) to (1);
			\draw [black] (3) to (1);
			\draw [black] (2) to (1);
			\begin{pgfonlayer}{bg}
				\path[fill=darkgray] (2.center) -- (6.center) -- (14.center) -- 
				(15.center) -- (7.center);
			\end{pgfonlayer}
		\end{tikzpicture}
	\end{subfigure}\hfill
	\begin{subfigure}{0.3\linewidth}
		\begin{tikzpicture}[scale=0.05]
			\definecolor{marked}{rgb}{0.25,0.5,0.25}
			\node [circle,fill,scale=0.25] (49) at (27.016554,26.812277) {};
			\node [circle,fill,scale=0.25] (48) at (18.164678,67.928974) {};
			\node [circle,fill,scale=0.25] (47) at (24.345277,65.582442) {};
			\node [circle,fill,scale=0.25] (46) at (9.333755,55.106850) {};
			\node [circle,fill,scale=0.25] (45) at (15.870523,37.046930) {};
			\node [circle,fill,scale=0.25] (44) at (39.503458,28.456946) {};
			\node [circle,fill,scale=0.25] (43) at (50.377122,23.271526) {};
			\node [circle,fill,scale=0.25] (42) at (49.706685,15.907187) {};
			\node [circle,fill,scale=0.25] (41) at (31.175363,14.472030) {};
			\node [circle,fill,scale=0.25] (40) at (0.000000,49.994762) {};
			\node [circle,fill,scale=0.25] (39) at (24.020534,81.882463) {};
			\node [circle,fill,scale=0.25] (38) at (32.013410,70.296459) {};
			\node [circle,fill,scale=0.25] (37) at (25.109994,57.683846) {};
			\node [circle,fill,scale=0.25] (36) at (18.709410,51.429918) {};
			\node [circle,fill,scale=0.25] (35) at (27.309869,40.891472) {};
			\node [circle,fill,scale=0.25] (34) at (46.406873,35.224177) {};
			\node [circle,fill,scale=0.25] (33) at (62.853552,27.870312) {};
			\node [circle,fill,scale=0.25] (32) at (73.957680,26.183742) {};
			\node [circle,fill,scale=0.25] (31) at (68.300859,13.990153) {};
			\node [circle,fill,scale=0.25] (30) at (25.005238,6.699142) {};
			\node [circle,fill,scale=0.25] (29) at (25.005238,93.300857) {};
			\node [circle,fill,scale=0.25] (28) at (42.750892,87.748793) {};
			\node [circle,fill,scale=0.25] (27) at (33.270481,78.498847) {};
			\node [circle,fill,scale=0.25] (26) at (38.309241,64.367274) {};
			\node [circle,fill,scale=0.25] (25) at (35.030380,57.411481) {};
			\node [circle,fill,scale=0.25] (24) at (27.728893,46.579718) {};
			\node [circle,fill,scale=0.25] (23) at (42.489002,40.482924) {};
			\node [circle,fill,scale=0.25] (22) at (62.005029,35.685103) {};
			\node [circle,fill,scale=0.25] (21) at (85.899853,38.503037) {};
			\node [circle,fill,scale=0.25] (20) at (75.005238,6.699142) {};
			\node [circle,fill,scale=0.25] (19) at (75.005238,93.300857) {};
			\node [circle,fill,scale=0.25] (18) at (71.590195,85.276556) {};
			\node [circle,fill,scale=0.25] (17) at (58.495706,84.124239) {};
			\node [circle,fill,scale=0.25] (16) at (43.494659,78.236956) {};
			\node [circle,fill,scale=0.25] (15) at (45.547875,65.582442) {};
			\node [circle,fill,scale=0.25] (14) at (39.451080,51.440393) {};
			\node [circle,fill,scale=0.25] (13) at (55.164468,46.988267) {};
			\node [circle,fill,scale=0.25] (12) at (74.858580,41.509532) {};
			\node [circle,fill,scale=0.25] (11) at (99.999999,49.994762) {};
			\node [circle,fill,scale=0.25] (10) at (78.064111,70.924994) {};
			\node [circle,fill,scale=0.25] (9) at (61.114604,75.356170) {};
			\node [circle,fill,scale=0.25] (8) at (53.781690,71.783992) {};
			\node [circle,fill,scale=0.25] (7) at (51.822754,54.080242) {};
			\node [circle,fill,scale=0.25] (6) at (70.595014,47.176827) {};
			\node [circle,fill,scale=0.25] (5) at (88.424473,57.746700) {};
			\node [circle,fill,scale=0.25] (4) at (70.196943,68.683217) {};
			\node [circle,fill,scale=0.25] (3) at (58.453804,60.292268) {};
			\node [circle,fill,scale=0.25] (2) at (77.707942,53.514560) {};
			\node [circle,fill,scale=0.25] (1) at (69.233188,59.600879) {};
			\draw [black] (49) to (44);
			\draw [black] (49) to (41);
			\draw [black] (49) to (45);
			\draw [black] (48) to (39);
			\draw [black] (48) to (47);
			\draw [black] (48) to (46);
			\draw [black] (47) to (37);
			\draw [black] (47) to (38);
			\draw [black] (46) to (40);
			\draw [black] (46) to (36);
			\draw [black] (45) to (35);
			\draw [black] (45) to (40);
			\draw [black] (44) to (34);
			\draw [black] (44) to (43);
			\draw [black] (43) to (33);
			\draw [black] (43) to (42);
			\draw [black] (42) to (31);
			\draw [black] (42) to (41);
			\draw [black] (41) to (30);
			\draw [black] (40) to (29);
			\draw [black] (40) to (30);
			\draw [black] (39) to (27);
			\draw [black] (39) to (29);
			\draw [black] (38) to (26);
			\draw [black] (38) to (27);
			\draw [black] (37) to (25);
			\draw [black] (37) to (36);
			\draw [black] (36) to (24);
			\draw [black] (35) to (23);
			\draw [black] (35) to (24);
			\draw [black] (34) to (22);
			\draw [black] (34) to (23);
			\draw [black] (33) to (32);
			\draw [black] (33) to (22);
			\draw [black] (32) to (21);
			\draw [black] (32) to (31);
			\draw [black] (31) to (20);
			\draw [black] (30) to (20);
			\draw [black] (29) to (19);
			\draw [black] (29) to (28);
			\draw [black] (28) to (16);
			\draw [black] (28) to (17);
			\draw [black] (27) to (16);
			\draw [black] (26) to (25);
			\draw [black] (26) to (15);
			\draw [black] (25) to (14);
			\draw [black] (24) to (14);
			\draw [black] (23) to (13);
			\draw [black] (22) to (12);
			\draw [black] (21) to (11);
			\draw [black] (21) to (12);
			\draw [black] (20) to (11);
			\draw [black] (19) to (18);
			\draw [black] (19) to (11);
			\draw [black] (18) to (10);
			\draw [black] (18) to (17);
			\draw [black] (17) to (9);
			\draw [black] (16) to (8);
			\draw [black] (16) to (15);
			\draw [black] (15) to (7);
			\draw [black] (14) to (7);
			\draw [black] (13) to (6);
			\draw [black] (13) to (7);
			\draw [black] (12) to (6);
			\draw [black] (11) to (5);
			\draw [black] (10) to (4);
			\draw [black] (10) to (5);
			\draw [black] (9) to (8);
			\draw [black] (9) to (4);
			\draw [black] (8) to (3);
			\draw [black] (7) to (3);
			\draw [black] (6) to (2);
			\draw [black] (5) to (2);
			\draw [black] (4) to (1);
			\draw [black] (3) to (1);
			\draw [black] (2) to (1);
			\begin{pgfonlayer}{bg}
				\path[fill=darkgray] (1.center) -- (3.center) -- (8.center) -- 
				(9.center) -- (4.center);
			\end{pgfonlayer}
			\path[draw=black, very thick] (6) circle[radius=2];
			\path[draw=black, very thick] (13) circle[radius=2];
			\path[draw=black, very thick] (35) circle[radius=2];
		\end{tikzpicture}
	\end{subfigure}\hfill
\begin{subfigure}{0.3\linewidth}
	\begin{tikzpicture}[scale=0.05]
		\definecolor{marked}{rgb}{0.25,0.5,0.25}
		\node [circle,fill,scale=0.25] (50) at (38.980281,41.636865) {};
		\node [circle,fill,scale=0.25] (49) at (45.744355,47.512005) {};
		\node [circle,fill,scale=0.25] (48) at (48.339635,41.125983) {};
		\node [circle,fill,scale=0.25] (47) at (40.860326,34.668438) {};
		\node [circle,fill,scale=0.25] (46) at (32.052724,34.014509) {};
		\node [circle,fill,scale=0.25] (45) at (23.725352,40.594666) {};
		\node [circle,fill,scale=0.25] (44) at (25.503221,48.022887) {};
		\node [circle,fill,scale=0.25] (43) at (32.165119,54.429344) {};
		\node [circle,fill,scale=0.25] (42) at (38.162870,59.047716) {};
		\node [circle,fill,scale=0.25] (41) at (51.496884,56.544395) {};
		\node [circle,fill,scale=0.25] (40) at (55.318279,48.891386) {};
		\node [circle,fill,scale=0.25] (39) at (59.323593,33.391234) {};
		\node [circle,fill,scale=0.25] (38) at (48.114847,27.618269) {};
		\node [circle,fill,scale=0.25] (37) at (42.106877,21.967917) {};
		\node [circle,fill,scale=0.25] (36) at (32.604477,26.187799) {};
		\node [circle,fill,scale=0.25] (35) at (17.094105,37.376111) {};
		\node [circle,fill,scale=0.25] (34) at (17.809341,56.432001) {};
		\node [circle,fill,scale=0.25] (33) at (32.880352,65.454172) {};
		\node [circle,fill,scale=0.25] (32) at (41.606213,68.744251) {};
		\node [circle,fill,scale=0.25] (31) at (55.308062,71.308878) {};
		\node [circle,fill,scale=0.25] (30) at (57.014407,61.918871) {};
		\node [circle,fill,scale=0.25] (29) at (65.770920,45.008685) {};
		\node [circle,fill,scale=0.25] (28) at (67.569224,34.852356) {};
		\node [circle,fill,scale=0.25] (27) at (57.821600,25.983447) {};
		\node [circle,fill,scale=0.25] (26) at (45.631961,15.081231) {};
		\node [circle,fill,scale=0.25] (25) at (24.604068,25.084296) {};
		\node [circle,fill,scale=0.25] (24) at (8.950650,48.268110) {};
		\node [circle,fill,scale=0.25] (23) at (21.477472,67.497700) {};
		\node [circle,fill,scale=0.25] (22) at (35.097579,77.327065) {};
		\node [circle,fill,scale=0.25] (21) at (43.516911,79.922344) {};
		\node [circle,fill,scale=0.25] (20) at (63.727394,72.514559) {};
		\node [circle,fill,scale=0.25] (19) at (72.933484,63.073464) {};
		\node [circle,fill,scale=0.25] (18) at (71.666496,52.559517) {};
		\node [circle,fill,scale=0.25] (17) at (75.947685,29.549402) {};
		\node [circle,fill,scale=0.25] (16) at (60.028610,17.288241) {};
		\node [circle,fill,scale=0.25] (15) at (30.172679,13.528150) {};
		\node [circle,fill,scale=0.25] (14) at (0.000000,49.994891) {};
		\node [circle,fill,scale=0.25] (13) at (23.306428,79.043628) {};
		\node [circle,fill,scale=0.25] (12) at (43.813222,87.677530) {};
		\node [circle,fill,scale=0.25] (11) at (60.447532,82.783282) {};
		\node [circle,fill,scale=0.25] (10) at (80.412792,65.321344) {};
		\node [circle,fill,scale=0.25] (9) at (81.485644,42.413405) {};
		\node [circle,fill,scale=0.25] (8) at (71.247574,15.377542) {};
		\node [circle,fill,scale=0.25] (7) at (25.002556,6.702770) {};
		\node [circle,fill,scale=0.25] (6) at (25.002556,93.297229) {};
		\node [circle,fill,scale=0.25] (5) at (73.188923,83.059160) {};
		\node [circle,fill,scale=0.25] (4) at (90.180852,49.586185) {};
		\node [circle,fill,scale=0.25] (3) at (74.997445,6.702770) {};
		\node [circle,fill,scale=0.25] (2) at (74.997445,93.297229) {};
		\node [circle,fill,scale=0.25] (1) at (99.999999,49.994891) {};
		\draw [black] (50) to (47);
		\draw [black] (50) to (44);
		\draw [black] (50) to (48);
		\draw [black] (49) to (48);
		\draw [black] (49) to (43);
		\draw [black] (49) to (40);
		\draw [black] (48) to (39);
		\draw [black] (47) to (38);
		\draw [black] (47) to (46);
		\draw [black] (46) to (45);
		\draw [black] (46) to (36);
		\draw [black] (45) to (35);
		\draw [black] (45) to (44);
		\draw [black] (44) to (34);
		\draw [black] (43) to (34);
		\draw [black] (43) to (42);
		\draw [black] (42) to (33);
		\draw [black] (42) to (41);
		\draw [black] (41) to (40);
		\draw [black] (41) to (30);
		\draw [black] (40) to (29);
		\draw [black] (39) to (27);
		\draw [black] (39) to (28);
		\draw [black] (38) to (27);
		\draw [black] (38) to (37);
		\draw [black] (37) to (26);
		\draw [black] (37) to (36);
		\draw [black] (36) to (25);
		\draw [black] (35) to (24);
		\draw [black] (35) to (25);
		\draw [black] (34) to (23);
		\draw [black] (34) to (24);
		\draw [black] (33) to (23);
		\draw [black] (33) to (32);
		\draw [black] (32) to (22);
		\draw [black] (32) to (30);
		\draw [black] (31) to (20);
		\draw [black] (31) to (30);
		\draw [black] (31) to (21);
		\draw [black] (30) to (18);
		\draw [black] (29) to (28);
		\draw [black] (29) to (18);
		\draw [black] (28) to (17);
		\draw [black] (27) to (16);
		\draw [black] (26) to (15);
		\draw [black] (26) to (16);
		\draw [black] (25) to (15);
		\draw [black] (24) to (14);
		\draw [black] (23) to (13);
		\draw [black] (22) to (21);
		\draw [black] (22) to (13);
		\draw [black] (21) to (12);
		\draw [black] (20) to (19);
		\draw [black] (20) to (11);
		\draw [black] (19) to (10);
		\draw [black] (19) to (18);
		\draw [black] (18) to (9);
		\draw [black] (17) to (8);
		\draw [black] (17) to (9);
		\draw [black] (16) to (8);
		\draw [black] (15) to (7);
		\draw [black] (14) to (6);
		\draw [black] (14) to (7);
		\draw [black] (13) to (6);
		\draw [black] (12) to (11);
		\draw [black] (12) to (6);
		\draw [black] (11) to (5);
		\draw [black] (10) to (4);
		\draw [black] (10) to (5);
		\draw [black] (9) to (4);
		\draw [black] (8) to (3);
		\draw [black] (7) to (3);
		\draw [black] (6) to (2);
		\draw [black] (5) to (2);
		\draw [black] (4) to (1);
		\draw [black] (3) to (1);
		\draw [black] (2) to (1);
		\begin{pgfonlayer}{bg}
			\path[fill=darkgray] (5.center) -- (10.center) -- (19.center) -- 
			(20.center) -- (11.center);
		\end{pgfonlayer}
		\path[draw=black, very thick] (3) circle[radius=2];
		\path[draw=black, very thick] (8) circle[radius=2];
		\path[draw=black, very thick] (14) circle[radius=2];
		\path[draw=black, very thick] (39) circle[radius=2];
		\path[draw=black, very thick] (48) circle[radius=2];
	\end{tikzpicture}
\end{subfigure}\hfill
\begin{subfigure}{0.3\linewidth}
	\begin{tikzpicture}[scale=0.05]
    \definecolor{marked}{rgb}{0.25,0.5,0.25}
    \node [circle,fill,scale=0.25] (52) at (38.308632,49.071934) {};
    \node [circle,fill,scale=0.25] (51) at (46.838770,57.048704) {};
    \node [circle,fill,scale=0.25] (50) at (42.075296,39.010368) {};
    \node [circle,fill,scale=0.25] (49) at (32.717041,49.326587) {};
    \node [circle,fill,scale=0.25] (48) at (40.880959,59.224579) {};
    \node [circle,fill,scale=0.25] (47) at (52.006942,61.992198) {};
    \node [circle,fill,scale=0.25] (46) at (50.485076,48.276636) {};
    \node [circle,fill,scale=0.25] (45) at (48.616901,40.386796) {};
    \node [circle,fill,scale=0.25] (44) at (39.784644,29.537793) {};
    \node [circle,fill,scale=0.25] (43) at (29.371875,40.285709) {};
    \node [circle,fill,scale=0.25] (42) at (21.811524,44.018757) {};
    \node [circle,fill,scale=0.25] (41) at (21.516960,54.728020) {};
    \node [circle,fill,scale=0.25] (40) at (28.828115,58.544523) {};
    \node [circle,fill,scale=0.25] (39) at (38.476756,69.711106) {};
    \node [circle,fill,scale=0.25] (38) at (52.393446,70.120364) {};
    \node [circle,fill,scale=0.25] (37) at (59.657253,54.787101) {};
    \node [circle,fill,scale=0.25] (36) at (58.679839,47.998928) {};
    \node [circle,fill,scale=0.25] (35) at (55.620490,36.037370) {};
    \node [circle,fill,scale=0.25] (34) at (55.525663,29.137399) {};
    \node [circle,fill,scale=0.25] (33) at (46.029418,23.815029) {};
    \node [circle,fill,scale=0.25] (32) at (32.301909,29.945348) {};
    \node [circle,fill,scale=0.25] (31) at (17.367916,37.365777) {};
    \node [circle,fill,scale=0.25] (30) at (16.937763,61.448556) {};
    \node [circle,fill,scale=0.25] (29) at (31.333726,69.155890) {};
    \node [circle,fill,scale=0.25] (28) at (44.730302,75.818371) {};
    \node [circle,fill,scale=0.25] (27) at (61.572465,66.709118) {};
    \node [circle,fill,scale=0.25] (26) at (66.617622,54.595479) {};
    \node [circle,fill,scale=0.25] (25) at (64.281692,41.550647) {};
    \node [circle,fill,scale=0.25] (24) at (65.937031,29.681171) {};
    \node [circle,fill,scale=0.25] (23) at (46.525209,16.152365) {};
    \node [circle,fill,scale=0.25] (22) at (26.006089,25.138729) {};
    \node [circle,fill,scale=0.25] (21) at (8.725079,49.299199) {};
    \node [circle,fill,scale=0.25] (20) at (25.124735,73.949719) {};
    \node [circle,fill,scale=0.25] (19) at (45.516276,83.483839) {};
    \node [circle,fill,scale=0.25] (18) at (65.863244,70.345696) {};
    \node [circle,fill,scale=0.25] (17) at (75.207470,58.590860) {};
    \node [circle,fill,scale=0.25] (16) at (74.583801,42.748095) {};
    \node [circle,fill,scale=0.25] (15) at (61.864907,18.563000) {};
    \node [circle,fill,scale=0.25] (14) at (30.919694,13.834937) {};
    \node [circle,fill,scale=0.25] (13) at (0.000000,49.127333) {};
    \node [circle,fill,scale=0.25] (12) at (29.747768,85.459113) {};
    \node [circle,fill,scale=0.25] (11) at (60.518727,81.671322) {};
    \node [circle,fill,scale=0.25] (10) at (80.102070,65.193001) {};
    \node [circle,fill,scale=0.25] (9) at (80.264372,35.714097) {};
    \node [circle,fill,scale=0.25] (8) at (71.768394,16.250516) {};
    \node [circle,fill,scale=0.25] (7) at (25.748541,6.259267) {};
    \node [circle,fill,scale=0.25] (6) at (24.236782,92.867894) {};
    \node [circle,fill,scale=0.25] (5) at (70.596343,84.518490) {};
    \node [circle,fill,scale=0.25] (4) at (90.014078,50.620256) {};
    \node [circle,fill,scale=0.25] (3) at (75.753431,7.132105) {};
    \node [circle,fill,scale=0.25] (2) at (74.241672,93.740732) {};
    \node [circle,fill,scale=0.25] (1) at (99.999999,50.872839) {};
    \draw [black] (52) to (48);
    \draw [black] (52) to (50);
    \draw [black] (52) to (49);
    \draw [black] (51) to (47);
    \draw [black] (51) to (46);
    \draw [black] (51) to (48);
    \draw [black] (50) to (44);
    \draw [black] (50) to (45);
    \draw [black] (49) to (43);
    \draw [black] (49) to (40);
    \draw [black] (48) to (39);
    \draw [black] (47) to (38);
    \draw [black] (47) to (37);
    \draw [black] (46) to (36);
    \draw [black] (46) to (45);
    \draw [black] (45) to (35);
    \draw [black] (44) to (32);
    \draw [black] (44) to (33);
    \draw [black] (43) to (42);
    \draw [black] (43) to (32);
    \draw [black] (42) to (31);
    \draw [black] (42) to (41);
    \draw [black] (41) to (30);
    \draw [black] (41) to (40);
    \draw [black] (40) to (29);
    \draw [black] (39) to (28);
    \draw [black] (39) to (29);
    \draw [black] (38) to (27);
    \draw [black] (38) to (28);
    \draw [black] (37) to (36);
    \draw [black] (37) to (26);
    \draw [black] (36) to (25);
    \draw [black] (35) to (34);
    \draw [black] (35) to (25);
    \draw [black] (34) to (24);
    \draw [black] (34) to (33);
    \draw [black] (33) to (23);
    \draw [black] (32) to (22);
    \draw [black] (31) to (21);
    \draw [black] (31) to (22);
    \draw [black] (30) to (20);
    \draw [black] (30) to (21);
    \draw [black] (29) to (20);
    \draw [black] (28) to (19);
    \draw [black] (27) to (18);
    \draw [black] (27) to (26);
    \draw [black] (26) to (16);
    \draw [black] (25) to (16);
    \draw [black] (24) to (15);
    \draw [black] (24) to (16);
    \draw [black] (23) to (14);
    \draw [black] (23) to (15);
    \draw [black] (22) to (14);
    \draw [black] (21) to (13);
    \draw [black] (20) to (12);
    \draw [black] (19) to (11);
    \draw [black] (19) to (12);
    \draw [black] (18) to (17);
    \draw [black] (18) to (11);
    \draw [black] (17) to (10);
    \draw [black] (17) to (16);
    \draw [black] (16) to (9);
    \draw [black] (15) to (8);
    \draw [black] (14) to (7);
    \draw [black] (13) to (6);
    \draw [black] (13) to (7);
    \draw [black] (12) to (6);
    \draw [black] (11) to (5);
    \draw [black] (10) to (4);
    \draw [black] (10) to (5);
    \draw [black] (9) to (8);
    \draw [black] (9) to (4);
    \draw [black] (8) to (3);
    \draw [black] (7) to (3);
    \draw [black] (6) to (2);
    \draw [black] (5) to (2);
    \draw [black] (4) to (1);
    \draw [black] (3) to (1);
    \draw [black] (2) to (1);
\begin{pgfonlayer}{bg}
\path[fill=darkgray] (5.center) -- (10.center) -- (17.center) -- (18.center) -- 
(11.center);
\end{pgfonlayer}
\end{tikzpicture}
\end{subfigure}\qquad\qquad
\begin{subfigure}{0.3\linewidth}
	\begin{tikzpicture}[scale=0.05]
		\definecolor{marked}{rgb}{0.25,0.5,0.25}
		\node [circle,fill,scale=0.25] (53) at (44.493319,93.184638) {};
		\node [circle,fill,scale=0.25] (52) at (39.384647,86.558118) {};
		\node [circle,fill,scale=0.25] (51) at (31.094948,83.097855) {};
		\node [circle,fill,scale=0.25] (50) at (22.755693,84.109674) {};
		\node [circle,fill,scale=0.25] (49) at (28.930459,99.636726) {};
		\node [circle,fill,scale=0.25] (48) at (70.130444,88.083616) {};
		\node [circle,fill,scale=0.25] (47) at (55.115556,85.746962) {};
		\node [circle,fill,scale=0.25] (46) at (44.855992,77.858383) {};
		\node [circle,fill,scale=0.25] (45) at (41.422030,70.070603) {};
		\node [circle,fill,scale=0.25] (44) at (33.458860,73.090793) {};
		\node [circle,fill,scale=0.25] (43) at (20.636641,71.322144) {};
		\node [circle,fill,scale=0.25] (42) at (8.495548,62.049863) {};
		\node [circle,fill,scale=0.25] (41) at (9.644218,31.530242) {};
		\node [circle,fill,scale=0.25] (40) at (0.002854,70.204265) {};
		\node [circle,fill,scale=0.25] (39) at (70.197203,99.999999) {};
		\node [circle,fill,scale=0.25] (38) at (91.164728,66.426166) {};
		\node [circle,fill,scale=0.25] (37) at (78.965010,74.106907) {};
		\node [circle,fill,scale=0.25] (36) at (63.675964,81.622988) {};
		\node [circle,fill,scale=0.25] (35) at (52.706387,75.351964) {};
		\node [circle,fill,scale=0.25] (34) at (45.539094,60.316275) {};
		\node [circle,fill,scale=0.25] (33) at (29.753491,65.749447) {};
		\node [circle,fill,scale=0.25] (32) at (17.613968,62.234641) {};
		\node [circle,fill,scale=0.25] (31) at (12.441720,45.834343) {};
		\node [circle,fill,scale=0.25] (30) at (23.706966,21.527889) {};
		\node [circle,fill,scale=0.25] (29) at (0.360419,28.923397) {};
		\node [circle,fill,scale=0.25] (28) at (99.643788,71.066687) {};
		\node [circle,fill,scale=0.25] (27) at (88.645726,49.233167) {};
		\node [circle,fill,scale=0.25] (26) at (80.377390,48.672076) {};
		\node [circle,fill,scale=0.25] (25) at (75.728544,61.277995) {};
		\node [circle,fill,scale=0.25] (24) at (64.013198,64.261690) {};
		\node [circle,fill,scale=0.25] (23) at (53.131293,62.883726) {};
		\node [circle,fill,scale=0.25] (22) at (38.392990,56.697892) {};
		\node [circle,fill,scale=0.25] (21) at (29.713027,49.994653) {};
		\node [circle,fill,scale=0.25] (20) at (23.804395,39.436415) {};
		\node [circle,fill,scale=0.25] (19) at (29.730423,27.267321) {};
		\node [circle,fill,scale=0.25] (18) at (34.438473,8.485093) {};
		\node [circle,fill,scale=0.25] (17) at (29.802796,0.000000) {};
		\node [circle,fill,scale=0.25] (16) at (99.997146,29.795734) {};
		\node [circle,fill,scale=0.25] (15) at (90.754367,32.495078) {};
		\node [circle,fill,scale=0.25] (14) at (74.762514,35.547339) {};
		\node [circle,fill,scale=0.25] (13) at (68.261065,55.522353) {};
		\node [circle,fill,scale=0.25] (12) at (57.403538,55.465230) {};
		\node [circle,fill,scale=0.25] (11) at (40.713583,48.474372) {};
		\node [circle,fill,scale=0.25] (10) at (33.129008,40.761791) {};
		\node [circle,fill,scale=0.25] (9) at (43.156725,22.154880) {};
		\node [circle,fill,scale=0.25] (8) at (49.632730,12.383889) {};
		\node [circle,fill,scale=0.25] (7) at (71.069539,0.363275) {};
		\node [circle,fill,scale=0.25] (6) at (77.335217,24.166300) {};
		\node [circle,fill,scale=0.25] (5) at (66.502555,38.219234) {};
		\node [circle,fill,scale=0.25] (4) at (53.346958,41.866971) {};
		\node [circle,fill,scale=0.25] (3) at (46.866510,29.954433) {};
		\node [circle,fill,scale=0.25] (2) at (66.408372,11.454558) {};
		\node [circle,fill,scale=0.25] (1) at (58.762003,29.165055) {};
		\draw [black] (53) to (47);
		\draw [black] (53) to (52);
		\draw [black] (53) to (49);
		\draw [black] (52) to (46);
		\draw [black] (52) to (51);
		\draw [black] (51) to (50);
		\draw [black] (51) to (44);
		\draw [black] (50) to (43);
		\draw [black] (50) to (49);
		\draw [black] (49) to (39);
		\draw [black] (49) to (40);
		\draw [black] (48) to (36);
		\draw [black] (48) to (39);
		\draw [black] (48) to (37);
		\draw [black] (47) to (35);
		\draw [black] (47) to (36);
		\draw [black] (46) to (35);
		\draw [black] (46) to (45);
		\draw [black] (45) to (34);
		\draw [black] (45) to (44);
		\draw [black] (44) to (33);
		\draw [black] (43) to (32);
		\draw [black] (43) to (33);
		\draw [black] (42) to (31);
		\draw [black] (42) to (40);
		\draw [black] (42) to (32);
		\draw [black] (41) to (29);
		\draw [black] (41) to (31);
		\draw [black] (41) to (30);
		\draw [black] (40) to (29);
		\draw [black] (39) to (28);
		\draw [black] (38) to (27);
		\draw [black] (38) to (37);
		\draw [black] (38) to (28);
		\draw [black] (37) to (25);
		\draw [black] (36) to (24);
		\draw [black] (35) to (23);
		\draw [black] (34) to (22);
		\draw [black] (34) to (23);
		\draw [black] (33) to (22);
		\draw [black] (32) to (21);
		\draw [black] (31) to (20);
		\draw [black] (30) to (18);
		\draw [black] (30) to (19);
		\draw [black] (29) to (17);
		\draw [black] (28) to (16);
		\draw [black] (27) to (26);
		\draw [black] (27) to (15);
		\draw [black] (26) to (14);
		\draw [black] (26) to (25);
		\draw [black] (25) to (13);
		\draw [black] (24) to (12);
		\draw [black] (24) to (13);
		\draw [black] (23) to (12);
		\draw [black] (22) to (11);
		\draw [black] (21) to (10);
		\draw [black] (21) to (11);
		\draw [black] (20) to (10);
		\draw [black] (20) to (19);
		\draw [black] (19) to (9);
		\draw [black] (18) to (17);
		\draw [black] (18) to (8);
		\draw [black] (17) to (7);
		\draw [black] (16) to (15);
		\draw [black] (16) to (7);
		\draw [black] (15) to (6);
		\draw [black] (14) to (5);
		\draw [black] (14) to (6);
		\draw [black] (13) to (5);
		\draw [black] (12) to (4);
		\draw [black] (11) to (4);
		\draw [black] (10) to (3);
		\draw [black] (9) to (8);
		\draw [black] (9) to (3);
		\draw [black] (8) to (2);
		\draw [black] (7) to (2);
		\draw [black] (6) to (2);
		\draw [black] (5) to (1);
		\draw [black] (4) to (1);
		\draw [black] (3) to (1);
		\draw [black] (2) to (1);
		\begin{pgfonlayer}{bg}
			\path[fill=darkgray] (6.center) -- (14.center) -- (26.center) -- 
			(27.center) -- (15.center);
		\end{pgfonlayer}
		\path[draw=black, very thick] (4) circle[radius=2];
		\path[draw=black, very thick] (21) circle[radius=2];
		\path[draw=black, very thick] (24) circle[radius=2];
		\path[draw=black, very thick] (32) circle[radius=2];
		\path[draw=black, very thick] (36) circle[radius=2];
	\end{tikzpicture}
\end{subfigure}
	\end{adjustbox}
\caption{planar $K_2$-hypohamiltonian graphs on $48$, $49$, $50$, $52$, and $53$
vertices together with an extendable $5$-cycle and all of their exceptional
vertices.} 
%\Carol{I propose using instead of blue dark grey (here and elsewhere); the 
%circles could also just be black.}
\label{fig:planar_graphs}
\end{figure}
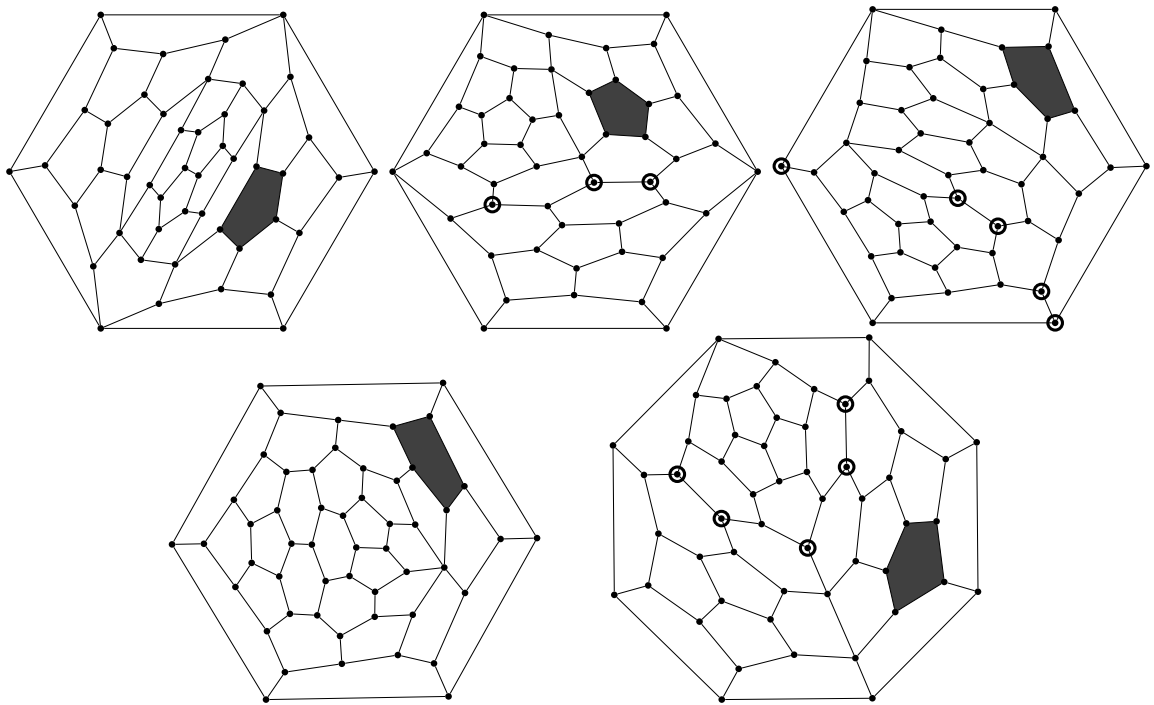

\subsection{The cubic planar case}
\label{subsect:planar_cubic}

%\Jan{Maybe we should reminder the reader that cubic $K_2$-hypohamiltonian graphs
%must be cyclically 4-edge-connected and have girth (at least) 5?} \Jarne{The
%latter is done in section \ref{subsubsect:cubic}.}
%JG: but not the cyclically 4-edge-connected part. This should be mentioned in the section about the planar case

Using the program \verb|plantri|~\cite{BM07} we generated all cyclically 
4-edge-connected cubic planar
graphs of
%\Jarne{at least?} \Carol{such graphs have girth 4 or 5. So girth 5 equals
%girth at least 5.}
girth 5 up to 78
vertices and verified
which of them are $K_2$-hypohamiltonian using the program mentioned at the
beginning of Section~\ref{sect:orders}. The results are summarised in
Table~\ref{table:counts_cubic_planar}. 

The counts of all cyclically 4-edge-connected non-hamiltonian cubic planar
graphs of girth 5 up to 76 vertices were
already tabulated before by McKay in~\cite{M17} and in~\cite{GZ17} the first 
and last author
determined
% \Gabor{"we" doesn't work here...} 
that there are exactly 860~350 such graphs on 78 vertices. 
%\Jan{I
%didn't include the counts of non-ham graphs in the Table to keep it concise and
%since these counts were already published in Brendan's paper anyway.}
All cyclically 4-edge-connected non-hamiltonian cubic planar
graphs of girth $5$ were stored and we verified that these are indeed 
non-hamiltonian and contain
the same number of $K_2$-hypohamiltonian graphs as in Table~\ref{table:counts_cubic_planar} using an independent program.

The three cubic planar $K_2$-hypohamiltonian graphs up to 74 vertices are shown 
in Figure~\ref{fig:cubic_planar} and the six remaining cubic planar 
$K_2$-hypohamiltonian graphs up to 78 vertices are shown in 
Figure~\ref{fig:cubic_planar_76_78} in the Appendix.
The $K_2$-hypohamiltonian graphs from Table~\ref{table:counts_cubic_planar} can 
be obtained from the database of interesting graphs from the \textit{House of 
Graphs}~\cite{CDG23} by searching for the keywords ``cubic planar 
K2-hypohamiltonian''.

\begin{table}[ht!]
\centering
\begin{tabular}{ l | ccccccc}
%\hline
Order 		& $\leq 66$ & 68 & 70 & 72 & 74 & 76 & 78 \\
\hline
Number of graphs & 0 & 1 & 1 & 0 & 1 & 3 & 3\\
\end{tabular}
\caption{Counts of all $K_2$-hypohamiltonian graphs among cubic planar graphs 
up to 78 vertices.}
\label{table:counts_cubic_planar}
\end{table}

The following observation immediately follows from the results in 
Table~\ref{table:counts_cubic_planar}. (Recall that $K_2$-hypohamiltonian 
graphs are $3$-connected and cyclically $4$-edge-connected. Cubic 
$K_2$-hypohamiltonian graphs must have girth at least $5$. Since a planar 
$3$-connected graph can have girth at most $5$, cubic planar 
$K_2$-hypohamiltonian graphs must have girth $5$.) %\Jan{Use Theorem or rather 
%Observation 
%or Corollary?}
%\Jan{We say Theorem but what follows is an Observation...}
%\Gabor{A sentence about properties of cubic planar K2-hypo graphs would be 
%nice 
%here (to explain why the result actually follows from the 
%computations).} \Jarne{Perhaps just a sentence in the beginning of this 
%subsection mentioning that 
%cubic planar graphs have girth at most 5?}

%\noindent \textbf{Theorem 5.} \emph{}

\begin{observation}
There exists exactly one cubic planar $K_2$-hypohamiltonian graph of order $68, 
70$, and $74$, and no such graph on fewer than $68$ vertices or of order~$72$. 
Moreover, there exist such graphs of order $76$ and $78$.
\end{observation}

%\begin{center}
%\includegraphics[height=50mm]{68} \quad \includegraphics[height=50mm]{70}\\
%Figure: Planar cubic $K_2$-hypohamiltonian graphs on 68 and 70 vertices, resp.
%\end{center}

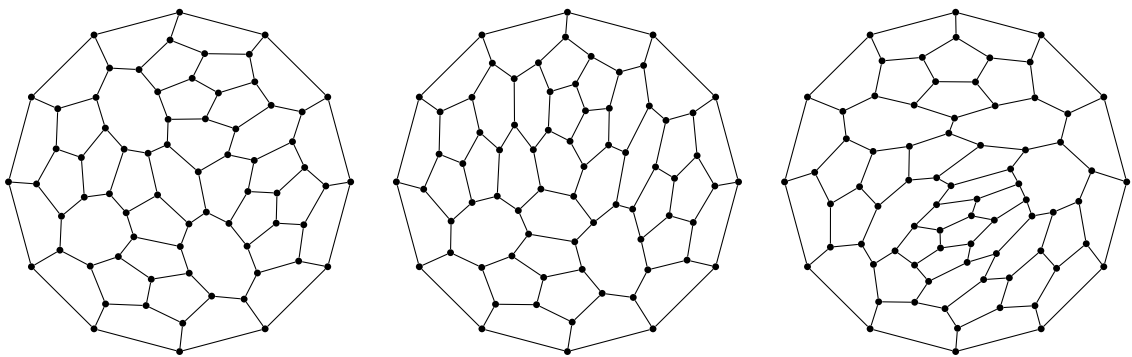
\begin{figure}[!htb]
\begin{center}
\begin{tikzpicture}[scale=0.045]
    \definecolor{marked}{rgb}{0.25,0.5,0.25}
    \node [circle,fill,scale=0.25] (68) at (15.472042,39.887174) {};
    \node [circle,fill,scale=0.25] (67) at (8.188155,49.394392) {};
    \node [circle,fill,scale=0.25] (66) at (13.920691,59.747801) {};
    \node [circle,fill,scale=0.25] (65) at (21.287538,57.176040) {};
    \node [circle,fill,scale=0.25] (64) at (22.150323,45.503567) {};
    \node [circle,fill,scale=0.25] (63) at (32.097227,27.982412) {};
    \node [circle,fill,scale=0.25] (62) at (41.728886,21.312429) {};
    \node [circle,fill,scale=0.25] (61) at (23.884186,25.186661) {};
    \node [circle,fill,scale=0.25] (60) at (15.024059,29.882197) {};
    \node [circle,fill,scale=0.25] (59) at (0.000002,50.000000) {};
    \node [circle,fill,scale=0.25] (58) at (14.393561,71.561306) {};
    \node [circle,fill,scale=0.25] (57) at (25.576571,74.904595) {};
    \node [circle,fill,scale=0.25] (56) at (28.305955,65.861954) {};
    \node [circle,fill,scale=0.25] (55) at (29.492284,46.499087) {};
    \node [circle,fill,scale=0.25] (54) at (34.386925,40.899286) {};
    \node [circle,fill,scale=0.25] (53) at (36.635140,33.806206) {};
    \node [circle,fill,scale=0.25] (52) at (52.804047,23.104365) {};
    \node [circle,fill,scale=0.25] (51) at (40.202421,13.555668) {};
    \node [circle,fill,scale=0.25] (50) at (28.364029,14.028538) {};
    \node [circle,fill,scale=0.25] (49) at (6.694872,24.995853) {};
    \node [circle,fill,scale=0.25] (48) at (6.694872,74.995851) {};
    \node [circle,fill,scale=0.25] (47) at (29.533763,83.549027) {};
    \node [circle,fill,scale=0.25] (46) at (38.136717,83.101044) {};
    \node [circle,fill,scale=0.25] (45) at (43.603781,76.547203) {};
    \node [circle,fill,scale=0.25] (44) at (33.797909,59.639953) {};
    \node [circle,fill,scale=0.25] (43) at (43.562302,46.158951) {};
    \node [circle,fill,scale=0.25] (42) at (50.199102,30.952381) {};
    \node [circle,fill,scale=0.25] (41) at (59.374480,16.343124) {};
    \node [circle,fill,scale=0.25] (40) at (50.920855,8.337482) {};
    \node [circle,fill,scale=0.25] (39) at (24.995851,6.694874) {};
    \node [circle,fill,scale=0.25] (38) at (24.995851,93.296830) {};
    \node [circle,fill,scale=0.25] (37) at (47.212542,91.811846) {};
    \node [circle,fill,scale=0.25] (36) at (57.325367,87.796580) {};
    \node [circle,fill,scale=0.25] (35) at (53.683423,80.529284) {};
    \node [circle,fill,scale=0.25] (34) at (46.648414,68.425417) {};
    \node [circle,fill,scale=0.25] (33) at (40.783141,58.437033) {};
    \node [circle,fill,scale=0.25] (32) at (52.729383,37.597478) {};
    \node [circle,fill,scale=0.25] (31) at (68.856808,15.455453) {};
    \node [circle,fill,scale=0.25] (30) at (49.999998,0.000000) {};
    \node [circle,fill,scale=0.25] (29) at (49.999998,99.999999) {};
    \node [circle,fill,scale=0.25] (28) at (70.374977,87.605774) {};
    \node [circle,fill,scale=0.25] (27) at (61.348928,76.198771) {};
    \node [circle,fill,scale=0.25] (26) at (57.541063,68.566449) {};
    \node [circle,fill,scale=0.25] (25) at (46.399534,60.992201) {};
    \node [circle,fill,scale=0.25] (24) at (57.814831,41.106686) {};
    \node [circle,fill,scale=0.25] (23) at (72.755928,23.145844) {};
    \node [circle,fill,scale=0.25] (22) at (74.995849,6.694874) {};
    \node [circle,fill,scale=0.25] (21) at (74.995849,93.296830) {};
    \node [circle,fill,scale=0.25] (20) at (71.976106,79.500579) {};
    \node [circle,fill,scale=0.25] (19) at (66.426081,65.613073) {};
    \node [circle,fill,scale=0.25] (18) at (55.458767,52.969968) {};
    \node [circle,fill,scale=0.25] (17) at (64.418449,37.771694) {};
    \node [circle,fill,scale=0.25] (16) at (69.719593,31.060229) {};
    \node [circle,fill,scale=0.25] (15) at (84.801721,26.663349) {};
    \node [circle,fill,scale=0.25] (14) at (93.296827,24.995853) {};
    \node [circle,fill,scale=0.25] (13) at (93.296827,74.995851) {};
    \node [circle,fill,scale=0.25] (12) at (76.771193,72.565122) {};
    \node [circle,fill,scale=0.25] (11) at (64.078312,58.013937) {};
    \node [circle,fill,scale=0.25] (10) at (69.960177,47.171064) {};
    \node [circle,fill,scale=0.25] (9) at (78.239585,37.838062) {};
    \node [circle,fill,scale=0.25] (8) at (86.295004,37.365190) {};
    \node [circle,fill,scale=0.25] (7) at (99.999997,50.000000) {};
    \node [circle,fill,scale=0.25] (6) at (85.822130,70.623858) {};
    \node [circle,fill,scale=0.25] (5) at (71.818481,56.329848) {};
    \node [circle,fill,scale=0.25] (4) at (78.463577,46.714783) {};
    \node [circle,fill,scale=0.25] (3) at (92.749292,48.788784) {};
    \node [circle,fill,scale=0.25] (2) at (82.993193,61.846689) {};
    \node [circle,fill,scale=0.25] (1) at (86.560476,54.239256) {};
    \draw [black] (68) to (67);
    \draw [black] (68) to (64);
    \draw [black] (68) to (60);
    \draw [black] (67) to (59);
    \draw [black] (67) to (66);
    \draw [black] (66) to (58);
    \draw [black] (66) to (65);
    \draw [black] (65) to (56);
    \draw [black] (65) to (64);
    \draw [black] (64) to (55);
    \draw [black] (63) to (62);
    \draw [black] (63) to (61);
    \draw [black] (63) to (53);
    \draw [black] (62) to (52);
    \draw [black] (62) to (51);
    \draw [black] (61) to (60);
    \draw [black] (61) to (50);
    \draw [black] (60) to (49);
    \draw [black] (59) to (48);
    \draw [black] (59) to (49);
    \draw [black] (58) to (48);
    \draw [black] (58) to (57);
    \draw [black] (57) to (47);
    \draw [black] (57) to (56);
    \draw [black] (56) to (44);
    \draw [black] (55) to (54);
    \draw [black] (55) to (44);
    \draw [black] (54) to (43);
    \draw [black] (54) to (53);
    \draw [black] (53) to (42);
    \draw [black] (52) to (41);
    \draw [black] (52) to (42);
    \draw [black] (51) to (50);
    \draw [black] (51) to (40);
    \draw [black] (50) to (39);
    \draw [black] (49) to (39);
    \draw [black] (48) to (38);
    \draw [black] (47) to (38);
    \draw [black] (47) to (46);
    \draw [black] (46) to (37);
    \draw [black] (46) to (45);
    \draw [black] (45) to (34);
    \draw [black] (45) to (35);
    \draw [black] (44) to (33);
    \draw [black] (43) to (32);
    \draw [black] (43) to (33);
    \draw [black] (42) to (32);
    \draw [black] (41) to (31);
    \draw [black] (41) to (40);
    \draw [black] (40) to (30);
    \draw [black] (39) to (30);
    \draw [black] (38) to (29);
    \draw [black] (37) to (29);
    \draw [black] (37) to (36);
    \draw [black] (36) to (28);
    \draw [black] (36) to (35);
    \draw [black] (35) to (27);
    \draw [black] (34) to (25);
    \draw [black] (34) to (26);
    \draw [black] (33) to (25);
    \draw [black] (32) to (24);
    \draw [black] (31) to (23);
    \draw [black] (31) to (22);
    \draw [black] (30) to (22);
    \draw [black] (29) to (21);
    \draw [black] (28) to (20);
    \draw [black] (28) to (21);
    \draw [black] (27) to (26);
    \draw [black] (27) to (20);
    \draw [black] (26) to (19);
    \draw [black] (25) to (18);
    \draw [black] (24) to (17);
    \draw [black] (24) to (18);
    \draw [black] (23) to (15);
    \draw [black] (23) to (16);
    \draw [black] (22) to (14);
    \draw [black] (21) to (13);
    \draw [black] (20) to (12);
    \draw [black] (19) to (11);
    \draw [black] (19) to (12);
    \draw [black] (18) to (11);
    \draw [black] (17) to (10);
    \draw [black] (17) to (16);
    \draw [black] (16) to (9);
    \draw [black] (15) to (8);
    \draw [black] (15) to (14);
    \draw [black] (14) to (7);
    \draw [black] (13) to (6);
    \draw [black] (13) to (7);
    \draw [black] (12) to (6);
    \draw [black] (11) to (5);
    \draw [black] (10) to (4);
    \draw [black] (10) to (5);
    \draw [black] (9) to (8);
    \draw [black] (9) to (4);
    \draw [black] (8) to (3);
    \draw [black] (7) to (3);
    \draw [black] (6) to (2);
    \draw [black] (5) to (2);
    \draw [black] (4) to (1);
    \draw [black] (3) to (1);
    \draw [black] (2) to (1);
\end{tikzpicture}
\quad
\begin{tikzpicture}[scale=0.045]
    \definecolor{marked}{rgb}{0.25,0.5,0.25}
    \node [circle,fill,scale=0.25] (70) at (29.478908,45.798180) {};
    \node [circle,fill,scale=0.25] (69) at (42.241522,47.890818) {};
    \node [circle,fill,scale=0.25] (68) at (35.632754,41.546732) {};
    \node [circle,fill,scale=0.25] (67) at (22.133994,43.928866) {};
    \node [circle,fill,scale=0.25] (66) at (19.412737,55.392886) {};
    \node [circle,fill,scale=0.25] (65) at (30.148882,59.404466) {};
    \node [circle,fill,scale=0.25] (64) at (40.024813,59.487179) {};
    \node [circle,fill,scale=0.25] (63) at (51.811414,45.748552) {};
    \node [circle,fill,scale=0.25] (62) at (38.643506,34.623655) {};
    \node [circle,fill,scale=0.25] (61) at (33.763441,28.172043) {};
    \node [circle,fill,scale=0.25] (60) at (15.996692,38.420182) {};
    \node [circle,fill,scale=0.25] (59) at (7.973530,47.849462) {};
    \node [circle,fill,scale=0.25] (58) at (12.481388,58.196856) {};
    \node [circle,fill,scale=0.25] (57) at (24.400330,64.557485) {};
    \node [circle,fill,scale=0.25] (56) at (34.532671,66.765922) {};
    \node [circle,fill,scale=0.25] (55) at (44.400330,65.252274) {};
    \node [circle,fill,scale=0.25] (54) at (50.752688,61.985111) {};
    \node [circle,fill,scale=0.25] (53) at (54.822167,54.516128) {};
    \node [circle,fill,scale=0.25] (52) at (57.626137,37.998345) {};
    \node [circle,fill,scale=0.25] (51) at (52.125723,32.349047) {};
    \node [circle,fill,scale=0.25] (50) at (43.052109,21.852769) {};
    \node [circle,fill,scale=0.25] (49) at (24.946235,24.789080) {};
    \node [circle,fill,scale=0.25] (48) at (15.822991,29.139784) {};
    \node [circle,fill,scale=0.25] (47) at (0.000000,49.999999) {};
    \node [circle,fill,scale=0.25] (46) at (12.779157,70.636889) {};
    \node [circle,fill,scale=0.25] (45) at (22.092637,74.921423) {};
    \node [circle,fill,scale=0.25] (44) at (34.433416,80.421836) {};
    \node [circle,fill,scale=0.25] (43) at (44.929694,76.583954) {};
    \node [circle,fill,scale=0.25] (42) at (52.365591,78.858561) {};
    \node [circle,fill,scale=0.25] (41) at (55.277088,70.959470) {};
    \node [circle,fill,scale=0.25] (40) at (62.018197,60.860214) {};
    \node [circle,fill,scale=0.25] (39) at (64.201819,43.325061) {};
    \node [circle,fill,scale=0.25] (38) at (54.309346,24.226633) {};
    \node [circle,fill,scale=0.25] (37) at (40.959471,13.813066) {};
    \node [circle,fill,scale=0.25] (36) at (28.982630,13.945409) {};
    \node [circle,fill,scale=0.25] (35) at (6.699752,24.995863) {};
    \node [circle,fill,scale=0.25] (34) at (6.699752,74.995865) {};
    \node [circle,fill,scale=0.25] (33) at (28.047972,84.995866) {};
    \node [circle,fill,scale=0.25] (32) at (41.521918,85.103393) {};
    \node [circle,fill,scale=0.25] (31) at (56.923077,86.997517) {};
    \node [circle,fill,scale=0.25] (30) at (62.233250,71.687345) {};
    \node [circle,fill,scale=0.25] (29) at (66.989246,58.643506) {};
    \node [circle,fill,scale=0.25] (28) at (69.048800,41.976840) {};
    \node [circle,fill,scale=0.25] (27) at (71.464020,33.085193) {};
    \node [circle,fill,scale=0.25] (26) at (60.148883,17.154673) {};
    \node [circle,fill,scale=0.25] (25) at (51.381306,8.676591) {};
    \node [circle,fill,scale=0.25] (24) at (25.004134,6.691481) {};
    \node [circle,fill,scale=0.25] (23) at (25.004134,93.300248) {};
    \node [circle,fill,scale=0.25] (22) at (49.429280,92.688172) {};
    \node [circle,fill,scale=0.25] (21) at (65.161291,82.357320) {};
    \node [circle,fill,scale=0.25] (20) at (73.928868,72.382135) {};
    \node [circle,fill,scale=0.25] (19) at (75.740281,54.408601) {};
    \node [circle,fill,scale=0.25] (18) at (81.290323,51.042183) {};
    \node [circle,fill,scale=0.25] (17) at (79.751862,40.074441) {};
    \node [circle,fill,scale=0.25] (16) at (73.490487,24.086019) {};
    \node [circle,fill,scale=0.25] (15) at (69.230769,15.947064) {};
    \node [circle,fill,scale=0.25] (14) at (50.000000,0.000000) {};
    \node [circle,fill,scale=0.25] (13) at (50.000000,99.999999) {};
    \node [circle,fill,scale=0.25] (12) at (72.249794,84.301076) {};
    \node [circle,fill,scale=0.25] (11) at (78.800663,68.114144) {};
    \node [circle,fill,scale=0.25] (10) at (87.890820,57.733663) {};
    \node [circle,fill,scale=0.25] (9) at (86.765922,38.039702) {};
    \node [circle,fill,scale=0.25] (8) at (84.954509,26.947890) {};
    \node [circle,fill,scale=0.25] (7) at (75.004136,6.691481) {};
    \node [circle,fill,scale=0.25] (6) at (75.004136,93.300248) {};
    \node [circle,fill,scale=0.25] (5) at (86.691482,70.339122) {};
    \node [circle,fill,scale=0.25] (4) at (93.085194,48.784118) {};
    \node [circle,fill,scale=0.25] (3) at (93.308518,24.995863) {};
    \node [circle,fill,scale=0.25] (2) at (93.308518,74.995865) {};
    \node [circle,fill,scale=0.25] (1) at (99.999999,49.999999) {};
    \draw [black] (70) to (67);
    \draw [black] (70) to (65);
    \draw [black] (70) to (68);
    \draw [black] (69) to (68);
    \draw [black] (69) to (64);
    \draw [black] (69) to (63);
    \draw [black] (68) to (62);
    \draw [black] (67) to (60);
    \draw [black] (67) to (66);
    \draw [black] (66) to (58);
    \draw [black] (66) to (57);
    \draw [black] (65) to (56);
    \draw [black] (65) to (57);
    \draw [black] (64) to (55);
    \draw [black] (64) to (56);
    \draw [black] (63) to (52);
    \draw [black] (63) to (53);
    \draw [black] (62) to (51);
    \draw [black] (62) to (61);
    \draw [black] (61) to (49);
    \draw [black] (61) to (50);
    \draw [black] (60) to (59);
    \draw [black] (60) to (48);
    \draw [black] (59) to (47);
    \draw [black] (59) to (58);
    \draw [black] (58) to (46);
    \draw [black] (57) to (45);
    \draw [black] (56) to (44);
    \draw [black] (55) to (43);
    \draw [black] (55) to (54);
    \draw [black] (54) to (53);
    \draw [black] (54) to (41);
    \draw [black] (53) to (40);
    \draw [black] (52) to (51);
    \draw [black] (52) to (39);
    \draw [black] (51) to (38);
    \draw [black] (50) to (37);
    \draw [black] (50) to (38);
    \draw [black] (49) to (48);
    \draw [black] (49) to (36);
    \draw [black] (48) to (35);
    \draw [black] (47) to (34);
    \draw [black] (47) to (35);
    \draw [black] (46) to (34);
    \draw [black] (46) to (45);
    \draw [black] (45) to (33);
    \draw [black] (44) to (32);
    \draw [black] (44) to (33);
    \draw [black] (43) to (32);
    \draw [black] (43) to (42);
    \draw [black] (42) to (31);
    \draw [black] (42) to (41);
    \draw [black] (41) to (30);
    \draw [black] (40) to (29);
    \draw [black] (40) to (30);
    \draw [black] (39) to (28);
    \draw [black] (39) to (29);
    \draw [black] (38) to (26);
    \draw [black] (37) to (36);
    \draw [black] (37) to (25);
    \draw [black] (36) to (24);
    \draw [black] (35) to (24);
    \draw [black] (34) to (23);
    \draw [black] (33) to (23);
    \draw [black] (32) to (22);
    \draw [black] (31) to (21);
    \draw [black] (31) to (22);
    \draw [black] (30) to (21);
    \draw [black] (29) to (20);
    \draw [black] (28) to (19);
    \draw [black] (28) to (27);
    \draw [black] (27) to (16);
    \draw [black] (27) to (17);
    \draw [black] (26) to (25);
    \draw [black] (26) to (15);
    \draw [black] (25) to (14);
    \draw [black] (24) to (14);
    \draw [black] (23) to (13);
    \draw [black] (22) to (13);
    \draw [black] (21) to (12);
    \draw [black] (20) to (11);
    \draw [black] (20) to (12);
    \draw [black] (19) to (18);
    \draw [black] (19) to (11);
    \draw [black] (18) to (10);
    \draw [black] (18) to (17);
    \draw [black] (17) to (9);
    \draw [black] (16) to (8);
    \draw [black] (16) to (15);
    \draw [black] (15) to (7);
    \draw [black] (14) to (7);
    \draw [black] (13) to (6);
    \draw [black] (12) to (6);
    \draw [black] (11) to (5);
    \draw [black] (10) to (4);
    \draw [black] (10) to (5);
    \draw [black] (9) to (8);
    \draw [black] (9) to (4);
    \draw [black] (8) to (3);
    \draw [black] (7) to (3);
    \draw [black] (6) to (2);
    \draw [black] (5) to (2);
    \draw [black] (4) to (1);
    \draw [black] (3) to (1);
    \draw [black] (2) to (1);
\end{tikzpicture}
\quad
\begin{tikzpicture}[scale=0.045]
    \definecolor{marked}{rgb}{0.25,0.5,0.25}
    \node [circle,fill,scale=0.25] (74) at (17.031211,70.787251) {};
    \node [circle,fill,scale=0.25] (73) at (8.798233,52.813958) {};
    \node [circle,fill,scale=0.25] (72) at (6.701075,74.997949) {};
    \node [circle,fill,scale=0.25] (71) at (28.491849,85.663961) {};
    \node [circle,fill,scale=0.25] (70) at (40.255591,86.696155) {};
    \node [circle,fill,scale=0.25] (69) at (44.138609,79.544521) {};
    \node [circle,fill,scale=0.25] (68) at (26.370115,75.342015) {};
    \node [circle,fill,scale=0.25] (67) at (18.055216,62.718111) {};
    \node [circle,fill,scale=0.25] (66) at (13.418530,43.516015) {};
    \node [circle,fill,scale=0.25] (65) at (13.393956,30.777422) {};
    \node [circle,fill,scale=0.25] (64) at (0.000000,49.995903) {};
    \node [circle,fill,scale=0.25] (63) at (25.002050,93.298925) {};
    \node [circle,fill,scale=0.25] (62) at (50.135168,92.610795) {};
    \node [circle,fill,scale=0.25] (61) at (59.990169,86.556893) {};
    \node [circle,fill,scale=0.25] (60) at (55.795855,79.446217) {};
    \node [circle,fill,scale=0.25] (59) at (37.871713,72.622264) {};
    \node [circle,fill,scale=0.25] (58) at (25.878595,58.982550) {};
    \node [circle,fill,scale=0.25] (57) at (21.913657,48.578683) {};
    \node [circle,fill,scale=0.25] (56) at (22.462522,31.670353) {};
    \node [circle,fill,scale=0.25] (55) at (6.701075,25.002048) {};
    \node [circle,fill,scale=0.25] (54) at (50.004096,99.999997) {};
    \node [circle,fill,scale=0.25] (53) at (71.770295,85.303511) {};
    \node [circle,fill,scale=0.25] (52) at (61.546654,72.409271) {};
    \node [circle,fill,scale=0.25] (51) at (49.668223,68.796590) {};
    \node [circle,fill,scale=0.25] (50) at (36.479070,60.481689) {};
    \node [circle,fill,scale=0.25] (49) at (27.287622,42.811501) {};
    \node [circle,fill,scale=0.25] (48) at (25.976900,25.690178) {};
    \node [circle,fill,scale=0.25] (47) at (25.002050,6.701074) {};
    \node [circle,fill,scale=0.25] (46) at (74.997951,93.298925) {};
    \node [circle,fill,scale=0.25] (45) at (73.146554,74.784957) {};
    \node [circle,fill,scale=0.25] (44) at (48.005243,64.356516) {};
    \node [circle,fill,scale=0.25] (43) at (36.290653,50.520192) {};
    \node [circle,fill,scale=0.25] (42) at (32.284756,29.679692) {};
    \node [circle,fill,scale=0.25] (41) at (27.500615,14.671910) {};
    \node [circle,fill,scale=0.25] (40) at (50.004096,0.000002) {};
    \node [circle,fill,scale=0.25] (39) at (74.997951,6.701074) {};
    \node [circle,fill,scale=0.25] (38) at (93.298927,74.997949) {};
    \node [circle,fill,scale=0.25] (37) at (82.772178,70.082738) {};
    \node [circle,fill,scale=0.25] (36) at (57.278611,60.776602) {};
    \node [circle,fill,scale=0.25] (35) at (44.441714,51.224706) {};
    \node [circle,fill,scale=0.25] (34) at (37.929058,36.954206) {};
    \node [circle,fill,scale=0.25] (33) at (37.994594,25.485377) {};
    \node [circle,fill,scale=0.25] (32) at (38.027362,14.524453) {};
    \node [circle,fill,scale=0.25] (31) at (50.544769,6.856722) {};
    \node [circle,fill,scale=0.25] (30) at (62.972065,12.918816) {};
    \node [circle,fill,scale=0.25] (29) at (73.171130,13.508642) {};
    \node [circle,fill,scale=0.25] (28) at (93.298927,25.002048) {};
    \node [circle,fill,scale=0.25] (27) at (99.999999,49.995903) {};
    \node [circle,fill,scale=0.25] (26) at (80.650446,61.694108) {};
    \node [circle,fill,scale=0.25] (25) at (70.410420,59.416726) {};
    \node [circle,fill,scale=0.25] (24) at (48.636029,48.873596) {};
    \node [circle,fill,scale=0.25] (23) at (44.425330,43.286638) {};
    \node [circle,fill,scale=0.25] (22) at (45.424756,35.741787) {};
    \node [circle,fill,scale=0.25] (21) at (45.580405,30.515278) {};
    \node [circle,fill,scale=0.25] (20) at (42.106988,20.684853) {};
    \node [circle,fill,scale=0.25] (19) at (46.882936,12.648479) {};
    \node [circle,fill,scale=0.25] (18) at (65.683624,21.725239) {};
    \node [circle,fill,scale=0.25] (17) at (79.495371,24.535103) {};
    \node [circle,fill,scale=0.25] (16) at (88.170719,31.785041) {};
    \node [circle,fill,scale=0.25] (15) at (89.776356,53.280903) {};
    \node [circle,fill,scale=0.25] (14) at (66.060457,53.797001) {};
    \node [circle,fill,scale=0.25] (13) at (56.238224,44.933234) {};
    \node [circle,fill,scale=0.25] (12) at (54.624396,40.018022) {};
    \node [circle,fill,scale=0.25] (11) at (54.468748,31.752273) {};
    \node [circle,fill,scale=0.25] (10) at (53.371017,26.280004) {};
    \node [circle,fill,scale=0.25] (9) at (58.179733,21.209141) {};
    \node [circle,fill,scale=0.25] (8) at (74.899646,29.704269) {};
    \node [circle,fill,scale=0.25] (7) at (85.991643,44.269681) {};
    \node [circle,fill,scale=0.25] (6) at (68.460718,49.430654) {};
    \node [circle,fill,scale=0.25] (5) at (61.300892,38.731874) {};
    \node [circle,fill,scale=0.25] (4) at (61.907102,28.876874) {};
    \node [circle,fill,scale=0.25] (3) at (78.528712,41.050216) {};
    \node [circle,fill,scale=0.25] (2) at (70.647989,44.720242) {};
    \node [circle,fill,scale=0.25] (1) at (72.327353,39.936102) {};
    \draw [black] (74) to (67);
    \draw [black] (74) to (72);
    \draw [black] (74) to (68);
    \draw [black] (73) to (64);
    \draw [black] (73) to (67);
    \draw [black] (73) to (66);
    \draw [black] (72) to (63);
    \draw [black] (72) to (64);
    \draw [black] (71) to (63);
    \draw [black] (71) to (70);
    \draw [black] (71) to (68);
    \draw [black] (70) to (62);
    \draw [black] (70) to (69);
    \draw [black] (69) to (59);
    \draw [black] (69) to (60);
    \draw [black] (68) to (59);
    \draw [black] (67) to (58);
    \draw [black] (66) to (57);
    \draw [black] (66) to (65);
    \draw [black] (65) to (55);
    \draw [black] (65) to (56);
    \draw [black] (64) to (55);
    \draw [black] (63) to (54);
    \draw [black] (62) to (54);
    \draw [black] (62) to (61);
    \draw [black] (61) to (60);
    \draw [black] (61) to (53);
    \draw [black] (60) to (52);
    \draw [black] (59) to (51);
    \draw [black] (58) to (50);
    \draw [black] (58) to (57);
    \draw [black] (57) to (49);
    \draw [black] (56) to (48);
    \draw [black] (56) to (49);
    \draw [black] (55) to (47);
    \draw [black] (54) to (46);
    \draw [black] (53) to (45);
    \draw [black] (53) to (46);
    \draw [black] (52) to (51);
    \draw [black] (52) to (45);
    \draw [black] (51) to (44);
    \draw [black] (50) to (44);
    \draw [black] (50) to (43);
    \draw [black] (49) to (43);
    \draw [black] (48) to (41);
    \draw [black] (48) to (42);
    \draw [black] (47) to (40);
    \draw [black] (47) to (41);
    \draw [black] (46) to (38);
    \draw [black] (45) to (37);
    \draw [black] (44) to (36);
    \draw [black] (43) to (35);
    \draw [black] (42) to (33);
    \draw [black] (42) to (34);
    \draw [black] (41) to (32);
    \draw [black] (40) to (31);
    \draw [black] (40) to (39);
    \draw [black] (39) to (29);
    \draw [black] (39) to (28);
    \draw [black] (38) to (37);
    \draw [black] (38) to (27);
    \draw [black] (37) to (26);
    \draw [black] (36) to (25);
    \draw [black] (36) to (35);
    \draw [black] (35) to (24);
    \draw [black] (34) to (22);
    \draw [black] (34) to (23);
    \draw [black] (33) to (20);
    \draw [black] (33) to (21);
    \draw [black] (32) to (19);
    \draw [black] (32) to (20);
    \draw [black] (31) to (30);
    \draw [black] (31) to (19);
    \draw [black] (30) to (18);
    \draw [black] (30) to (29);
    \draw [black] (29) to (17);
    \draw [black] (28) to (27);
    \draw [black] (28) to (16);
    \draw [black] (27) to (15);
    \draw [black] (26) to (25);
    \draw [black] (26) to (15);
    \draw [black] (25) to (14);
    \draw [black] (24) to (14);
    \draw [black] (24) to (23);
    \draw [black] (23) to (13);
    \draw [black] (22) to (21);
    \draw [black] (22) to (12);
    \draw [black] (21) to (11);
    \draw [black] (20) to (10);
    \draw [black] (19) to (9);
    \draw [black] (18) to (8);
    \draw [black] (18) to (9);
    \draw [black] (17) to (8);
    \draw [black] (17) to (16);
    \draw [black] (16) to (7);
    \draw [black] (15) to (7);
    \draw [black] (14) to (6);
    \draw [black] (13) to (12);
    \draw [black] (13) to (6);
    \draw [black] (12) to (5);
    \draw [black] (11) to (10);
    \draw [black] (11) to (5);
    \draw [black] (10) to (4);
    \draw [black] (9) to (4);
    \draw [black] (8) to (3);
    \draw [black] (7) to (3);
    \draw [black] (6) to (2);
    \draw [black] (5) to (2);
    \draw [black] (4) to (1);
    \draw [black] (3) to (1);
    \draw [black] (2) to (1);
\end{tikzpicture}

\caption{All cubic planar $K_2$-hypohamiltonian graphs on 68, 70, and 74 
vertices, respectively.}
\label{fig:cubic_planar}
\end{center}
\end{figure}

\begin{corollary}
The family of cubic planar $K_2$-hypohamiltonian graphs has a Holton gap at 
$72$.
\end{corollary}

%\Jan{Wiener or the third author?}

%\Gabor{Well, we use "third author" throughout the paper, so based on this we 
%should change it. On the other hand, I'm not a huge fan of this approach, so 
%we 
%can keep this and change all the others. The reason is simply selling the 
%paper: a reference might seem stronger if it is not emphasized that it is by 
%one (or more) of the current authors (because it suggests that a broader 
%community is involved in this type of research), so if it is not needed I 
%would 
%just ignore the names as well -- of course the reviewer can check the 
%reference, but usually they just don't.}

Araya and the third author~\cite{AW11} discovered a cubic planar 
hypohamiltonian graph of order 70 (a smaller such graph, if it exists, must 
have at least 54~vertices~\cite{GZ17}), and it was shown subsequently that such 
graphs of order $n$ exist for every even $n \ge 74$, see~\cite{Za15}. Thus, one 
is tempted to believe that the family of cubic planar hypohamiltonian graphs 
\emph{also} has a Holton gap at $72$, but this remains unknown.

%\noindent \textbf{Theorem 6.} \emph{There exist infinitely many cubic planar 
%$K_2$-hypohamiltonian graphs.}

\begin{theorem}
There exist infinitely many cubic planar $K_2$-hypohamiltonian graphs.
\end{theorem}
\begin{proof}
In Figure~\ref{fig:68-ext-cyc} in the Appendix we give a 68-vertex cubic 
planar $K_2$-hypohamiltonian graph $G$ containing an extendable cycle; the 
proof is contained in the aforementioned figure. By Lemma~\ref{lem:ext5cycle}, 
the proof is complete.
\end{proof}

%\Jan{Should we repeat the definition of exceptional vertices to reminder the 
%reader?}

%\Gabor{I would not.}
Finally, we point out that among all cubic planar $K_2$-hypohamiltonian graphs 
that we were able to investigate, exceptional vertices always occurred (i.e.\ 
we do not know whether cubic planar hypohamiltonian $K_2$-hamiltonian graphs 
exist):

%\Gabor{Hmm, I would say that it is not known whether cubic planar K2-ham 
%hypoham graphs exist and it is an interesting open question. }

%\Gabor{It really is an interesting question -- Jarne, Jan: can you send me 
%figures of cubic planar K2-ham graphs, such that the exceptional vertices 
%are 
%marked (for a possible proof of the nonexistence of these)?}

\begin{observation}
All cubic planar $K_2$-hypohamiltonian graphs up to and including $78$ vertices
%\Carol{@Jarne: up to what order?}
have at least two exceptional vertices; there exists a cubic planar
$K_2$-hypohamiltonian graph of order~$76$ containing exactly five exceptional
vertices.
\end{observation}

\section{Concluding remarks}
\label{sect:last}

%\Jan{Maybe change the title to something like ``Concluding remarks''?}
%
%\Gabor{Agreed.}

\noindent \textbf{1. Errata.} In the last author's paper~\cite{Za21}, the first
part of this series of articles, there is one minor issue and one more
significant issue. We now explain these omissions and present corrections.

The minor issue concerns Lemma~3 in Section~3.1 of~\cite{Za21}. Let ${\cal
H}_k$ be the set of all 2-connected $n$-vertex graphs of circumference $n - 1$
that contain exactly $k$ exceptional vertices. The aforementioned lemma states:
``Let $i, j$ be non-negative integers and consider disjoint graphs $G \in {\cal
H}_i$ and $H \in {\cal H}_j$. We require $G$ and $H$ to contain cubic vertices
$x$ and $y$, respectively, such that $N(x) \cap {\rm exc}(G)  = \emptyset$ and
$N(y) \cap {\rm exc}(H) = \emptyset$. Let $F_G$ ($F_H$) be the non-trivial
$N(x)$-fragment ($N(y)$-fragment) of $G$ ($H$). Then $(F_G,N(x)) \: (F_H,N(y))
\in {\cal H}_{i+j}$.''

The issue here is that if $x$ or $y$ are exceptional, the statement is false.
So in the lemma's second sentence, $N(x)$ must be replaced by $N[x]$, i.e.\ the
closed neighbourhood of $x$, and $N(y)$ by $N[y]$. After performing this
correction, we obtain precisely the statement of Theorem 6 from the last author's paper~\cite{Za15}.

In~\cite{Za21}, Lemma~3 is used in the proof of Theorem~1, whose statement
remains valid. The only change that must be made in the proof of Theorem~1 is
that in its last paragraph, the sentence ``By construction, $\Gamma_k$ contains
a cubic vertex $y_k$ such that no vertex of $N(y_k)$ is exceptional.'' must be
replaced with: By construction, $\Gamma_k$ contains a cubic vertex $y_k$ such
that no vertex of $N[y_k]$ is exceptional.

The major issue concerns Proposition~6 in Section~3.2 of~\cite{Za21}, which
discusses the dot product in the context of $K_2$-hamiltonicity.
Unfortunately, the third and fourth paragraph of its ``proof'' contain errors.
In order to correct these, we need to change, in the proposition's statement,
condition (v) to (v)' and add a sixth condition as described below:\\[1mm]
(v)' For any $vw \in E(H)$ with $v,w \not\in \{ x, y\}$ there exists in $H - x - y - v - w$ a Hamiltonian $st$-path with $s \in \{ a', b' \}$ and $t \in \{ c', d' \}$.\\[1mm]
(vi) $H - x - a'$, $H - x - b'$, $H - y - c'$, and $H - y - d'$ are hamiltonian.\\[1mm]
The third (and penultimate) paragraph of the proposition's ``proof'' can now be saved as we get a hamiltonian cycle in $G - a - a'$ since there now exists a hamiltonian $c'd'$-path in $H - x - y - a'$ by condition (vi); one proceeds analogously for the hamiltonicity of $G - b - b'$, $G - c - c'$, and $G - d - d'$. In the fourth paragraph of the proposition's ``proof'' it is stated that by (v) there exists in $H - v - w$ a hamiltonian $s't'$-path ${\frak p}$ with $s' \in \{ a', b' \}$ and $t' \in \{ c',d' \}$. The issue is that such a path may start in $a'$, visit $x, y, d'$ (in this order), then continue through $H$ to $b'$, and, again through $H$, end in $c'$. Such a path cannot be extended to the desired cycle by the arguments given in the ``proof''. By replacing condition (v) with (v'), such a path might still exist, but we are now guaranteed to have a path which \emph{can} be extended, as described in~\cite{Za21}, to the desired cycle.

\smallskip

\noindent \textbf{2. Bipartite graphs.} Hypohamiltonian graphs cannot be 
bipartite. However, it follows from an operation of 
Thomassen~\cite[p.~41]{Th81} that for any $\varepsilon > 0$ there exists a 
bipartite graph $H$ and a hypohamiltonian graph $G$ such that $H$ is an induced 
subgraph of $G$ and $\lvert V(H) \rvert / \lvert V(G) \rvert > 1-\varepsilon$. 
By an 
operation of the last author \cite{Za21}, this statement is also true for 
$K_2$-hypohamiltonian graphs. 
%but for every sufficiently large constant $c$ there exists a 
%hypohamiltonian graph $G$ containing an induced bipartite subgraph $B$ such 
%that $|V(B)| > |V(H)| - c$. \Jarne{What is $H$? I also do not understand the 
%statement if $H$ is supposed to be $G$.}
% This follows from an operation 
%described by Thomassen~\cite[p.~41]{Th81}.
 Gr\"otschel asked whether there is a 
bipartite graph admitting no hamiltonian path, but in which every 
vertex-deleted subgraph does contain a hamiltonian path, 
see~\cite[Problem~4.56]{GGL95}. Related to this, Ozeki~\cite{Oz19} asked 
whether there is a non-hamiltonian bipartite graph with bipartition $(X,Y)$ 
such that $G - x - y$ is hamiltonian for all $x \in X$ and all $y \in Y$.

Here we address, computationally, a relaxation of this problem in tune with the
article's main theme, $K_2$-hamiltonicity. (Note that every bipartite
$K_2$-hamiltonian graph must be balanced.) We found no 
$K_2$-hypohamiltonian bipartite graphs up to and including order 16, and no
$K_2$-hypohamiltonian bipartite graphs of girth at least~6 up to and including
order 25. (None of the other graphs we had already generated in 
Table~\ref{table:counts_general} were
bipartite.) Recently, Brinkmann, McKay, and the first author~\cite{BGM22} 
showed 
that the smallest 3-connected non-hamiltonian cubic bipartite graph has 50 
vertices and they also generated a sample of 238~531 such graphs up to 64 
vertices. We verified that none of these graphs are $K_2$-hypohamiltonian. 
Among planar $3$-connected bipartite graphs there are no $K_2$-hypohamiltonian 
graphs 
up to and including $32$ vertices.
Barnette conjectured that every 3-connected cubic planar bipartite graph is 
hamiltonian and this conjecture was verified up to 90 vertices 
in~\cite{BGM22}, 
so there are no 3-connected cubic planar $K_2$-hypohamiltonian bipartite
graphs 
up to at least 90 vertices.

\smallskip

%\noindent \textbf{3. Line graphs.} It is well known that if $G$ is hamiltonian, then its line-graph $L(G)$ is hamiltonian---thus, if $L(G)$ is non-hamiltonian, then $G$ is non-ham. As the dodecahedron shows, it is not true that if $L(G)$ is $K_2$-hamiltonian, then $G$ is $K_1$-hamiltonian, and as the triangular prism shows, it is not true that if $L(G)$ is $K_1$-hamiltonian, then $G$ is $K_2$-hamiltonian. Is there nonetheless something interesting we can say about graphs and their line-graphs in terms of $K_1$- and $K_2$-hamiltonicity? \Carol{We can drop this one if you want.}

%\bigskip

\noindent \textbf{3. Toward Gr\"unbaum's conjecture.} Going back to 
Gr\"unbaum's problem, one goal is to find $K_2$-hypohamiltonian graphs with as 
few hamiltonian vertex-deleted subgraphs as possible. Let ${\cal G}_3$ be the 
family of all cubic $K_2$-hypohamiltonian graphs and
$$\rho_3 := \sup_{G \in {\cal G}_3} \frac{|{\rm exc}(G)|}{|V(G)|}.$$
In~\cite{Za21}, the last author showed that $\rho_3 \ge \frac14$. Any improvement of this bound would be very welcome.

\smallskip

\noindent \textbf{4. Girth.} It is natural to wonder what the smallest 
$K_2$-hypohamiltonian graph of a certain girth is. For an integer~$k \ge 3$, 
let $g_k$ denote the order of a smallest $K_2$-hypohamiltonian graph of 
girth~$k$. Our computations imply that $g_3 \ge 14$ and $g_4 \ge 17$. We do not 
know whether $K_2$-hypohamiltonian graphs of girth 3 or 4 exist. The smallest 
$K_2$-hypohamiltonian graph of girth $5$ is Petersen's, so $g_5 = 10$. There 
exists a $K_2$-hypohamiltonian graph of girth $6$ and order~$25$---it is the 
smallest hypohamiltonian graph of girth~6, first described in~\cite{GZ17}---and 
none on fewer vertices, so $g_6 = 25$. We have computed that $g_k \ge 26$ for 
all $k \ge 7$, but we do not know whether $K_2$-hypohamiltonian graphs of girth 
greater than 6~exist. The smallest cubic $K_2$-hypohamiltonian graph of girth~6 
is the flower snark $J_7$. %\Carol{What if 	we drop 'cubic'?}\Jarne{The 
%smallest $K_2$-hypohamiltonian graph of girth 	$6$ 	is the same graph as 
%the smallest hypohamiltonian graph of girth $6$ as the 	latter is 
%$K_2$-hypohamiltonian and there are no $K_2$-hypohamiltonian 	graphs of girth 
%$6$ up to and including order $24$. Maybe not mention	\textbf{the}	
%smallest but \textbf{a} smallest as there may be more on $25$ vertices}.
As mentioned in Theorem~2.9 of~\cite{GZ17}, the 28-vertex Coxeter graph is the only non-hamiltonian cubic graph with girth 7 up to at least 42 vertices (and Coxeter's graph is not $K_2$-hamiltonian), the smallest non-hamiltonian cubic graph with girth 8 has at least 50 vertices, and the smallest non-hamiltonian cubic graph with girth 9 has at least 66 vertices.

\section*{Acknowledgements}
Several of the computations for this work were carried out using the 
supercomputer infrastructure provided by the VSC (Flemish Supercomputer 
Center), funded by the Research Foundation Flanders (FWO) and the Flemish 
Government. The research of Jan Goedgebeur and Jarne Renders was supported by 
Internal Funds of KU Leuven. The research of G\'abor Wiener was supported by 
the National Research, Development and Innovation Office (NKFIH), Hungary, 
Grant no.\ K-124171 and is part of project 
no.\ BME-NVA-02, implemented with the support provided by the Ministry of 
Innovation and Technology of Hungary from the National Research, Development 
and Innovation Fund, financed under the TKP2021 funding scheme.
%was supported by the National Research, Development and Innovation Office 
%(NKFIH), Hungary, Grant no. K-124171 and... \Gabor{more to come here, 
%unfortunately}
The research of Carol T.~Zamfirescu was supported by a Postdoctoral Fellowship of the Research Foundation Flanders (FWO).

%\newpage

\appendix

\section{Appendix}

\subsection{Figures used in the proof of Lemma~\ref{lem:egy}}

\begin{figure}[H]
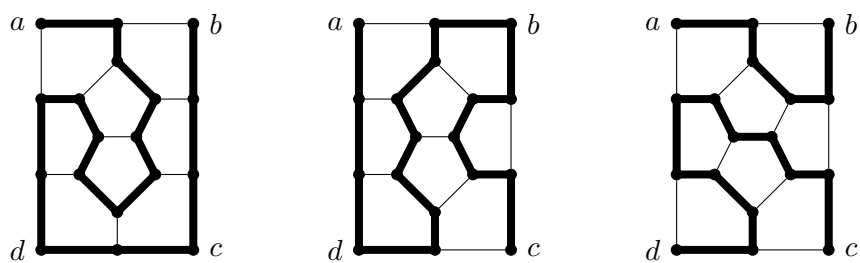

\begin{center}
\tikz[fo/.style={draw, fill=black, circle, minimum size={0.16cm}, inner sep=0cm, font=\bf, align=center, scale=0.88},
scale=0.5]
{
\node [fo, label=left:$d$] (1) at (0,0) {};
\node [fo] (2) at (2,0) {};
\node [fo, label=right:$c$] (3) at (4,0) {};
\node [fo] (4) at (0,2) {};
\node [fo] (5) at (0,4) {};
\node [fo, label=left:$a$] (6) at (0,6) {};
\node [fo] (7) at (4,2) {};
\node [fo] (8) at (4,4) {};
\node [fo, label=right:$b$] (9) at (4,6) {};
\node [fo] (10) at (2,6) {};
\node [fo] (11) at (2,1) {};
\node [fo] (12) at (2,5) {};
\node [fo] (13) at (1,2) {};
\node [fo] (14) at (1,4) {};
\node [fo] (15) at (3,2) {};
\node [fo] (16) at (3,4) {};
\node [fo] (17) at (1.5,3) {};
\node [fo] (18) at (2.5,3) {};
\draw
(1) edge node  {} (2)
    edge node {} (4)
(2) edge node {} (3)
    edge node {} (11)
(3) edge node {} (7)
(4) edge node {} (5)
    edge node {} (13)
(5) edge node {} (6)
    edge node {} (14)
(6) edge node {} (10)
(7) edge node {} (8)
    edge node {} (15)
(8) edge node {} (9)
    edge node {} (16)
(9) edge node {} (10)
(10) edge node {} (12)
(11) edge node {} (13)
     edge node {} (15)
(12) edge node {} (14)
     edge node {} (16)
(13) edge node {} (17)
(14) edge node {} (17)
(15) edge node {} (18)
(16) edge node {} (18)
(17) edge node {} (18);

\draw[line width=3pt] (6.center) -- (10.center) -- (12.center) -- (16.center) -- (18.center) -- (15.center) -- (11.center) -- (13.center) -- (17.center) -- (14.center) -- (5.center) -- (4.center) -- (1.center) -- (2.center) -- (3.center) -- (7.center) -- (8.center) -- (9.center)

;} \hskip 1cm \tikz[fo/.style={draw, fill=black, circle, minimum size={0.16cm}, inner sep=0cm, font=\bf, align=center, scale=0.88},
scale=0.5]
{
\node [fo, label=left:$d$] (1) at (0,0) {};
\node [fo] (2) at (2,0) {};
\node [fo, label=right:$c$] (3) at (4,0) {};
\node [fo] (4) at (0,2) {};
\node [fo] (5) at (0,4) {};
\node [fo, label=left:$a$] (6) at (0,6) {};
\node [fo] (7) at (4,2) {};
\node [fo] (8) at (4,4) {};
\node [fo, label=right:$b$] (9) at (4,6) {};
\node [fo] (10) at (2,6) {};
\node [fo] (11) at (2,1) {};
\node [fo] (12) at (2,5) {};
\node [fo] (13) at (1,2) {};
\node [fo] (14) at (1,4) {};
\node [fo] (15) at (3,2) {};
\node [fo] (16) at (3,4) {};
\node [fo] (17) at (1.5,3) {};
\node [fo] (18) at (2.5,3) {};
\draw
(1) edge node  {} (2)
    edge node {} (4)
(2) edge node {} (3)
    edge node {} (11)
(3) edge node {} (7)
(4) edge node {} (5)
    edge node {} (13)
(5) edge node {} (6)
    edge node {} (14)
(6) edge node {} (10)
(7) edge node {} (8)
    edge node {} (15)
(8) edge node {} (9)
    edge node {} (16)
(9) edge node {} (10)
(10) edge node {} (12)
(11) edge node {} (13)
     edge node {} (15)
(12) edge node {} (14)
     edge node {} (16)
(13) edge node {} (17)
(14) edge node {} (17)
(15) edge node {} (18)
(16) edge node {} (18)
(17) edge node {} (18);

\draw[line width=3pt] (6.center) -- (5.center) -- (4.center) -- (1.center) -- (2.center) -- (11.center) -- (13.center) -- (17.center) -- (14.center) -- (12.center) -- (10.center) -- (9.center) -- (8.center) -- (16.center) -- (18.center) -- (15.center) -- (7.center) -- (3.center)

;} \hskip 1cm \tikz[fo/.style={draw, fill=black, circle, minimum size={0.16cm}, inner sep=0cm, font=\bf, align=center, scale=0.88},
scale=0.5]
{
\node [fo, label=left:$d$] (1) at (0,0) {};
\node [fo] (2) at (2,0) {};
\node [fo, label=right:$c$] (3) at (4,0) {};
\node [fo] (4) at (0,2) {};
\node [fo] (5) at (0,4) {};
\node [fo, label=left:$a$] (6) at (0,6) {};
\node [fo] (7) at (4,2) {};
\node [fo] (8) at (4,4) {};
\node [fo, label=right:$b$] (9) at (4,6) {};
\node [fo] (10) at (2,6) {};
\node [fo] (11) at (2,1) {};
\node [fo] (12) at (2,5) {};
\node [fo] (13) at (1,2) {};
\node [fo] (14) at (1,4) {};
\node [fo] (15) at (3,2) {};
\node [fo] (16) at (3,4) {};
\node [fo] (17) at (1.5,3) {};
\node [fo] (18) at (2.5,3) {};
\draw
(1) edge node  {} (2)
    edge node {} (4)
(2) edge node {} (3)
    edge node {} (11)
(3) edge node {} (7)
(4) edge node {} (5)
    edge node {} (13)
(5) edge node {} (6)
    edge node {} (14)
(6) edge node {} (10)
(7) edge node {} (8)
    edge node {} (15)
(8) edge node {} (9)
    edge node {} (16)
(9) edge node {} (10)
(10) edge node {} (12)
(11) edge node {} (13)
     edge node {} (15)
(12) edge node {} (14)
     edge node {} (16)
(13) edge node {} (17)
(14) edge node {} (17)
(15) edge node {} (18)
(16) edge node {} (18)
(17) edge node {} (18);

\draw[line width=3pt] (6.center) -- (10.center) -- (12.center) -- (16.center) -- (8.center) -- (9.center)  (1.center) -- (2.center) -- (11.center) -- (13.center) -- (4.center) -- (5.center) -- (14.center) -- (17.center) -- (18.center) -- (15.center) -- (7.center) -- (3.center)

;}

\caption{Paths spanning $J_{18}$: a hamiltonian $ab$-path, a hamiltonian $ac$-path, and a disjoint $ab$-path and $cd$-path which together span $J_{18}$.}
\label{fig:J18spanningpaths}
\end{center}
\end{figure}

\begin{figure}[H]
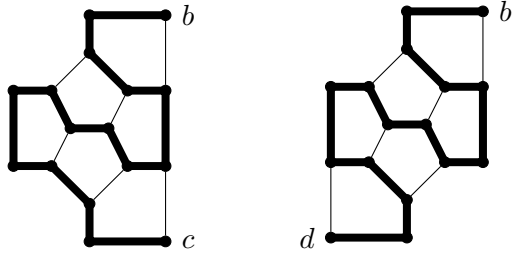

\begin{center}
\tikz[fo/.style={draw, fill=black, circle, minimum size={0.16cm}, inner sep=0cm, font=\bf, align=center, scale=0.88},
scale=0.5]
{

\node [fo] (2) at (2,0) {};
\node [fo, label=right:$c$] (3) at (4,0) {};
\node [fo] (4) at (0,2) {};
\node [fo] (5) at (0,4) {};

\node [fo] (7) at (4,2) {};
\node [fo] (8) at (4,4) {};
\node [fo, label=right:$b$] (9) at (4,6) {};
\node [fo] (10) at (2,6) {};
\node [fo] (11) at (2,1) {};
\node [fo] (12) at (2,5) {};
\node [fo] (13) at (1,2) {};
\node [fo] (14) at (1,4) {};
\node [fo] (15) at (3,2) {};
\node [fo] (16) at (3,4) {};
\node [fo] (17) at (1.5,3) {};
\node [fo] (18) at (2.5,3) {};
\draw

(2) edge node {} (3)
    edge node {} (11)
(3) edge node {} (7)
(4) edge node {} (5)
    edge node {} (13)
(5)
    edge node {} (14)

(7) edge node {} (8)
    edge node {} (15)
(8) edge node {} (9)
    edge node {} (16)
(9) edge node {} (10)
(10) edge node {} (12)
(11) edge node {} (13)
     edge node {} (15)
(12) edge node {} (14)
     edge node {} (16)
(13) edge node {} (17)
(14) edge node {} (17)
(15) edge node {} (18)
(16) edge node {} (18)
(17) edge node {} (18);

\draw[line width=3pt] (3.center) -- (2.center) -- (11.center) -- (13.center) -- (4.center) -- (5.center) -- (14.center) -- (17.center) -- (18.center) -- (15.center) -- (7.center) -- (8.center) -- (16.center) -- (12.center) -- (10.center) -- (9.center)

;} \hskip 1cm \tikz[fo/.style={draw, fill=black, circle, minimum size={0.16cm}, inner sep=0cm, font=\bf, align=center, scale=0.88},
scale=0.5]
{
\node [fo, label=left:$d$] (1) at (0,0) {};
\node [fo] (2) at (2,0) {};

\node [fo] (4) at (0,2) {};
\node [fo] (5) at (0,4) {};

\node [fo] (7) at (4,2) {};
\node [fo] (8) at (4,4) {};
\node [fo, label=right:$b$] (9) at (4,6) {};
\node [fo] (10) at (2,6) {};
\node [fo] (11) at (2,1) {};
\node [fo] (12) at (2,5) {};
\node [fo] (13) at (1,2) {};
\node [fo] (14) at (1,4) {};
\node [fo] (15) at (3,2) {};
\node [fo] (16) at (3,4) {};
\node [fo] (17) at (1.5,3) {};
\node [fo] (18) at (2.5,3) {};
\draw
(1) edge node  {} (2)
    edge node {} (4)
(2)
    edge node {} (11)

(4) edge node {} (5)
    edge node {} (13)
(5)
    edge node {} (14)

(7) edge node {} (8)
    edge node {} (15)
(8) edge node {} (9)
    edge node {} (16)
(9) edge node {} (10)
(10) edge node {} (12)
(11) edge node {} (13)
     edge node {} (15)
(12) edge node {} (14)
     edge node {} (16)
(13) edge node {} (17)
(14) edge node {} (17)
(15) edge node {} (18)
(16) edge node {} (18)
(17) edge node {} (18);

\draw[line width=3pt] (1.center) -- (2.center) -- (11.center) -- (13.center) -- (4.center) -- (5.center) -- (14.center) -- (17.center) -- (18.center) -- (15.center) -- (7.center) -- (8.center) -- (16.center) -- (12.center) -- (10.center) -- (9.center)

;}

\caption{A hamiltonian $bc$-path in $J_{18} - a - d$ and a hamiltonian $bd$-path in $J_{18} - a - c$.}
\label{fig:J18bcpath}
\end{center}
\end{figure}

%\Carol{After all this fuss I made about thicker paths in the figures, I am thinking that perhaps we should not include figs.~2 and 3 (but we should keep fig.~1), in particular if we aim at a more ``dense'' article (as these paths are not difficult to find).}
%\Jan{Personally, I would keep them in as it will save the readers some time and we already made those figures anyway. If you think it will make the main part too long, we could consider to move Figures 2+3 to the Appendix...}

\subsection{Certificates to verify the computer-aided proof of Lemma~\ref{lem:j18}}
\label{app:j18}
%\Jarne{I moved the explanation of these figures to the proof of theorem 2. Is
%this fine or should there be some lines here explaining it?} \Carol{I think
%this is fine.}

\begin{figure}[H]
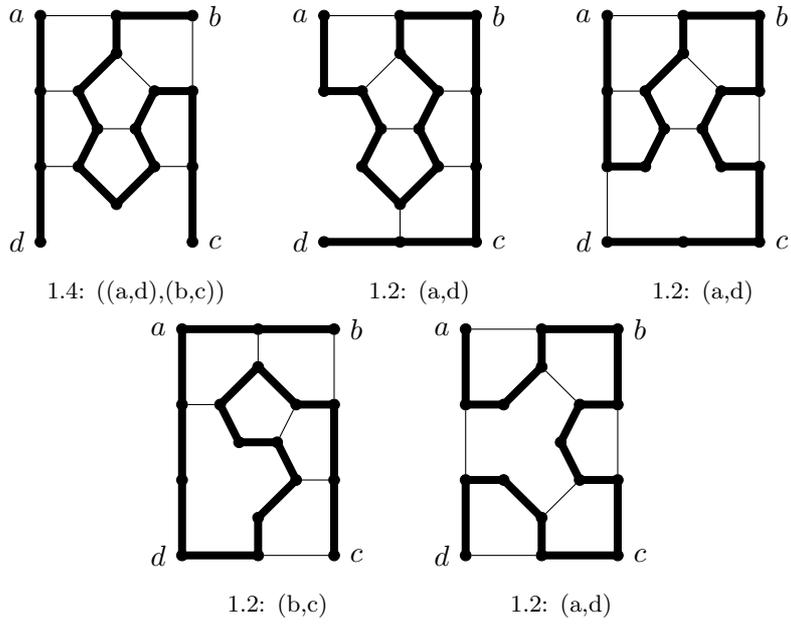

%	\centering
%	\begin{adjustbox}{width=\linewidth}
	\centering
	{\begin{subfigure}{0.24\linewidth}
\tikz[fo/.style={draw, fill=black, circle, minimum size={0.16cm}, inner 
sep=0cm, font=\bf, align=center, scale=0.88},scale=0.5] {
\node [fo, label=left:$d$] (1) at (0,0) {};
\node [fo, label=right:$c$] (3) at (4,0) {};
\node [fo] (4) at (0,2) {};
\node [fo] (5) at (0,4) {};
\node [fo, label=left:$a$] (6) at (0,6) {};
\node [fo] (7) at (4,2) {};
\node [fo] (8) at (4,4) {};
\node [fo, label=right:$b$] (9) at (4,6) {};
\node [fo] (10) at (2,6) {};
\node [fo] (11) at (2,1) {};
\node [fo] (12) at (2,5) {};
\node [fo] (13) at (1,2) {};
\node [fo] (14) at (1,4) {};
\node [fo] (15) at (3,2) {};
\node [fo] (16) at (3,4) {};
\node [fo] (17) at (1.5,3) {};
\node [fo] (18) at (2.5,3) {};

\draw

(1) edge node {} (4)

(3) edge node {} (7)

(4) edge node {} (5)
    edge node {} (13)

(5) edge node {} (6)
    edge node {} (14)
	
(6) edge node {} (10)

(7) edge node {} (8)
    edge node {} (15)
		
(8) edge node {} (9)
    edge node {} (16)
		
(9) edge node {} (10)

(10) edge node {} (12)
     		
(11) edge node {} (13)
     edge node {} (15)

(12) edge node {} (14)
     edge node {} (16)
		
(13) edge node {} (17)

(14) edge node {} (17)

(15) edge node {} (18)

(16) edge node {} (18)

(17) edge node {} (18)
;
\draw[line width = 3pt] (3.center) -- (7.center) -- (8.center) -- (16.center) 
-- (18.center) -- (15.center) -- (11.center) -- (13.center) -- (17.center) -- 
(14.center) -- (12.center) -- (10.center) -- (9.center);
;
\draw[line width = 3pt] (1.center) -- (4.center) -- (5.center) -- (6.center);
}
\subcaption*{1.4: ((a,d),(b,c))}
\end{subfigure}
\begin{subfigure}{0.24\linewidth}
\tikz[fo/.style={draw, fill=black, circle, minimum size={0.16cm}, inner 
sep=0cm, font=\bf, align=center, scale=0.88},scale=0.5] {
\node [fo, label=left:$d$] (1) at (0,0) {};
\node [fo] (2) at (2,0) {};
\node [fo, label=right:$c$] (3) at (4,0) {};
\node [fo] (5) at (0,4) {};
\node [fo, label=left:$a$] (6) at (0,6) {};
\node [fo] (7) at (4,2) {};
\node [fo] (8) at (4,4) {};
\node [fo, label=right:$b$] (9) at (4,6) {};
\node [fo] (10) at (2,6) {};
\node [fo] (11) at (2,1) {};
\node [fo] (12) at (2,5) {};
\node [fo] (13) at (1,2) {};
\node [fo] (14) at (1,4) {};
\node [fo] (15) at (3,2) {};
\node [fo] (16) at (3,4) {};
\node [fo] (17) at (1.5,3) {};
\node [fo] (18) at (2.5,3) {};

\draw

(1) edge node  {} (2)
(2) edge node {} (3)
    edge node {} (11)
		
(3) edge node {} (7)

(5) edge node {} (6)
    edge node {} (14)
	
(6) edge node {} (10)

(7) edge node {} (8)
    edge node {} (15)
		
(8) edge node {} (9)
    edge node {} (16)
		
(9) edge node {} (10)

(10) edge node {} (12)
     		
(11) edge node {} (13)
     edge node {} (15)

(12) edge node {} (14)
     edge node {} (16)
		
(13) edge node {} (17)

(14) edge node {} (17)

(15) edge node {} (18)

(16) edge node {} (18)

(17) edge node {} (18)
;
\draw[line width = 3pt] (1.center) -- (2.center) -- (3.center) -- (7.center) -- (8.center) -- (9.center) -- (10.center) -- (12.center) -- (16.center) -- (18.center) -- (15.center) -- (11.center) -- (13.center) -- (17.center) -- (14.center) -- (5.center) -- (6.center);
}
\subcaption*{1.2: (a,d)}
\end{subfigure}
\begin{subfigure}{0.24\linewidth}
\tikz[fo/.style={draw, fill=black, circle, minimum size={0.16cm}, inner 
sep=0cm, font=\bf, align=center, scale=0.88},scale=0.5] {
\node [fo, label=left:$d$] (1) at (0,0) {};
\node [fo] (2) at (2,0) {};
\node [fo, label=right:$c$] (3) at (4,0) {};
\node [fo] (4) at (0,2) {};
\node [fo] (5) at (0,4) {};
\node [fo, label=left:$a$] (6) at (0,6) {};
\node [fo] (7) at (4,2) {};
\node [fo] (8) at (4,4) {};
\node [fo, label=right:$b$] (9) at (4,6) {};
\node [fo] (10) at (2,6) {};
\node [fo] (12) at (2,5) {};
\node [fo] (13) at (1,2) {};
\node [fo] (14) at (1,4) {};
\node [fo] (15) at (3,2) {};
\node [fo] (16) at (3,4) {};
\node [fo] (17) at (1.5,3) {};
\node [fo] (18) at (2.5,3) {};

\draw

(1) edge node  {} (2)
    edge node {} (4)

(2) edge node {} (3)
(3) edge node {} (7)

(4) edge node {} (5)
    edge node {} (13)

(5) edge node {} (6)
    edge node {} (14)
	
(6) edge node {} (10)

(7) edge node {} (8)
    edge node {} (15)
		
(8) edge node {} (9)
    edge node {} (16)
		
(9) edge node {} (10)

(10) edge node {} (12)
     		
(12) edge node {} (14)
     edge node {} (16)
		
(13) edge node {} (17)

(14) edge node {} (17)

(15) edge node {} (18)

(16) edge node {} (18)

(17) edge node {} (18)
;
\draw[line width = 3pt] (1.center) -- (2.center) -- (3.center) -- (7.center) -- (15.center) -- (18.center) -- (16.center) -- (8.center) -- (9.center) -- (10.center) -- (12.center) -- (14.center) -- (17.center) -- (13.center) -- (4.center) -- (5.center) -- (6.center);
}
\subcaption*{1.2: (a,d)}
\end{subfigure}\\
\begin{subfigure}{0.24\linewidth}
\tikz[fo/.style={draw, fill=black, circle, minimum size={0.16cm}, inner 
sep=0cm, font=\bf, align=center, scale=0.88},scale=0.5] {
\node [fo, label=left:$d$] (1) at (0,0) {};
\node [fo] (2) at (2,0) {};
\node [fo, label=right:$c$] (3) at (4,0) {};
\node [fo] (4) at (0,2) {};
\node [fo] (5) at (0,4) {};
\node [fo, label=left:$a$] (6) at (0,6) {};
\node [fo] (7) at (4,2) {};
\node [fo] (8) at (4,4) {};
\node [fo, label=right:$b$] (9) at (4,6) {};
\node [fo] (10) at (2,6) {};
\node [fo] (11) at (2,1) {};
\node [fo] (12) at (2,5) {};
\node [fo] (14) at (1,4) {};
\node [fo] (15) at (3,2) {};
\node [fo] (16) at (3,4) {};
\node [fo] (17) at (1.5,3) {};
\node [fo] (18) at (2.5,3) {};

\draw

(1) edge node  {} (2)
    edge node {} (4)

(2) edge node {} (3)
    edge node {} (11)
		
(3) edge node {} (7)

(4) edge node {} (5)
(5) edge node {} (6)
    edge node {} (14)
	
(6) edge node {} (10)

(7) edge node {} (8)
    edge node {} (15)
		
(8) edge node {} (9)
    edge node {} (16)
		
(9) edge node {} (10)

(10) edge node {} (12)
     		
(11) edge node {} (15)

(12) edge node {} (14)
     edge node {} (16)
		
(14) edge node {} (17)

(15) edge node {} (18)

(16) edge node {} (18)

(17) edge node {} (18)
;
\draw[line width = 3pt] (3.center) -- (7.center) -- (8.center) -- (16.center) -- (12.center) -- (14.center) -- (17.center) -- (18.center) -- (15.center) -- (11.center) -- (2.center) -- (1.center) -- (4.center) -- (5.center) -- (6.center) -- (10.center) -- (9.center);
}
\subcaption*{1.2: (b,c)}
\end{subfigure}
\begin{subfigure}{0.24\linewidth}
\tikz[fo/.style={draw, fill=black, circle, minimum size={0.16cm}, inner 
sep=0cm, font=\bf, align=center, scale=0.88},scale=0.5] {
\node [fo, label=left:$d$] (1) at (0,0) {};
\node [fo] (2) at (2,0) {};
\node [fo, label=right:$c$] (3) at (4,0) {};
\node [fo] (4) at (0,2) {};
\node [fo] (5) at (0,4) {};
\node [fo, label=left:$a$] (6) at (0,6) {};
\node [fo] (7) at (4,2) {};
\node [fo] (8) at (4,4) {};
\node [fo, label=right:$b$] (9) at (4,6) {};
\node [fo] (10) at (2,6) {};
\node [fo] (11) at (2,1) {};
\node [fo] (12) at (2,5) {};
\node [fo] (13) at (1,2) {};
\node [fo] (14) at (1,4) {};
\node [fo] (15) at (3,2) {};
\node [fo] (16) at (3,4) {};
\node [fo] (18) at (2.5,3) {};

\draw

(1) edge node  {} (2)
    edge node {} (4)

(2) edge node {} (3)
    edge node {} (11)
		
(3) edge node {} (7)

(4) edge node {} (5)
    edge node {} (13)

(5) edge node {} (6)
    edge node {} (14)
	
(6) edge node {} (10)

(7) edge node {} (8)
    edge node {} (15)
		
(8) edge node {} (9)
    edge node {} (16)
		
(9) edge node {} (10)

(10) edge node {} (12)
     		
(11) edge node {} (13)
     edge node {} (15)

(12) edge node {} (14)
     edge node {} (16)
		
(15) edge node {} (18)

(16) edge node {} (18)

;
\draw[line width = 3pt] (1.center) -- (4.center) -- (13.center) -- (11.center) -- (2.center) -- (3.center) -- (7.center) -- (15.center) -- (18.center) -- (16.center) -- (8.center) -- (9.center) -- (10.center) -- (12.center) -- (14.center) -- (5.center) -- (6.center);
}
\subcaption*{1.2: (a,d)}
\end{subfigure}

}
%	\end{adjustbox}
	\caption{Proof that $J_{18}$ is a $K_1$-cell.}
	\label{fig:j18k1goodpairs}
\end{figure}

\begin{figure}[H]
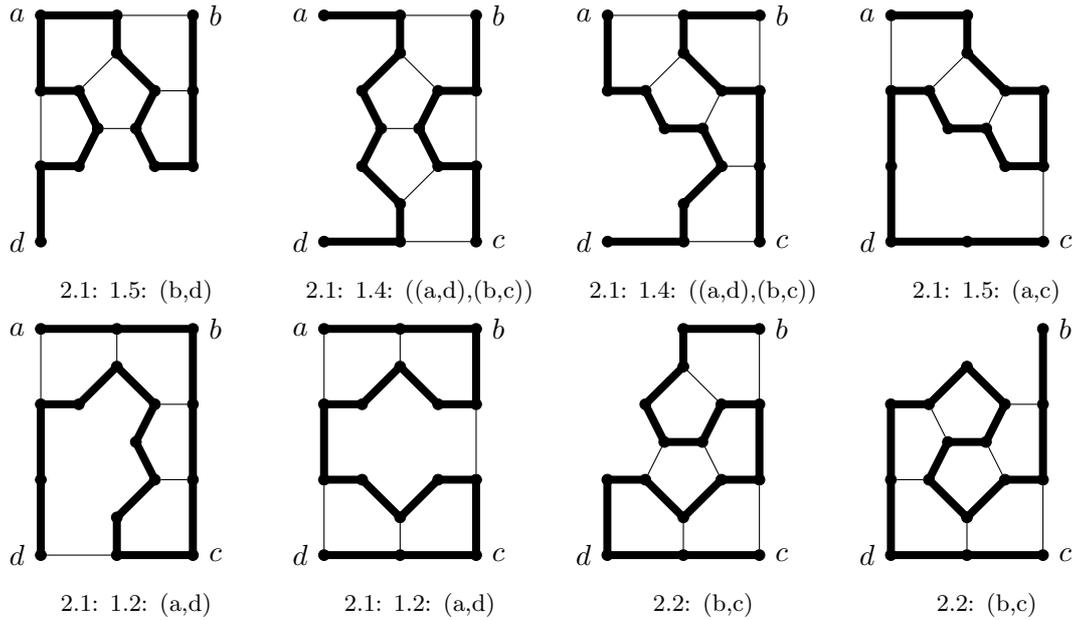

%	\centering
%	\begin{adjustbox}{minipage=\linewidth,scale=0.75}
		\centering
\begin{subfigure}{0.24\linewidth}
\tikz[fo/.style={draw, fill=black, circle, minimum size={0.16cm}, inner 
sep=0cm, font=\bf, align=center, scale=0.88},scale=0.5] {
\node [fo, label=left:$d$] (1) at (0,0) {};
\node [fo] (4) at (0,2) {};
\node [fo] (5) at (0,4) {};
\node [fo, label=left:$a$] (6) at (0,6) {};
\node [fo] (7) at (4,2) {};
\node [fo] (8) at (4,4) {};
\node [fo, label=right:$b$] (9) at (4,6) {};
\node [fo] (10) at (2,6) {};
\node [fo] (12) at (2,5) {};
\node [fo] (13) at (1,2) {};
\node [fo] (14) at (1,4) {};
\node [fo] (15) at (3,2) {};
\node [fo] (16) at (3,4) {};
\node [fo] (17) at (1.5,3) {};
\node [fo] (18) at (2.5,3) {};

\draw

(1) edge node {} (4)

(4) edge node {} (5)
    edge node {} (13)

(5) edge node {} (6)
    edge node {} (14)
	
(6) edge node {} (10)

(7) edge node {} (8)
    edge node {} (15)
		
(8) edge node {} (9)
    edge node {} (16)
		
(9) edge node {} (10)

(10) edge node {} (12)
     		
(12) edge node {} (14)
     edge node {} (16)
		
(13) edge node {} (17)

(14) edge node {} (17)

(15) edge node {} (18)

(16) edge node {} (18)

(17) edge node {} (18)
;
\draw[line width = 3pt] (1.center) -- (4.center) -- (13.center) -- (17.center) -- (14.center) -- (5.center) -- (6.center) -- (10.center) -- (12.center) -- (16.center) -- (18.center) -- (15.center) -- (7.center) -- (8.center) -- (9.center);
}
\subcaption*{2.1: 1.5: (b,d)}
\end{subfigure}
\begin{subfigure}{0.24\linewidth}
\tikz[fo/.style={draw, fill=black, circle, minimum size={0.16cm}, inner 
sep=0cm, font=\bf, align=center, scale=0.88},scale=0.5] {
\node [fo, label=left:$d$] (1) at (0,0) {};
\node [fo] (2) at (2,0) {};
\node [fo, label=right:$c$] (3) at (4,0) {};
\node [fo, label=left:$a$] (6) at (0,6) {};
\node [fo] (7) at (4,2) {};
\node [fo] (8) at (4,4) {};
\node [fo, label=right:$b$] (9) at (4,6) {};
\node [fo] (10) at (2,6) {};
\node [fo] (11) at (2,1) {};
\node [fo] (12) at (2,5) {};
\node [fo] (13) at (1,2) {};
\node [fo] (14) at (1,4) {};
\node [fo] (15) at (3,2) {};
\node [fo] (16) at (3,4) {};
\node [fo] (17) at (1.5,3) {};
\node [fo] (18) at (2.5,3) {};

\draw

(1) edge node  {} (2)
(2) edge node {} (3)
    edge node {} (11)
		
(3) edge node {} (7)

(6) edge node {} (10)

(7) edge node {} (8)
    edge node {} (15)
		
(8) edge node {} (9)
    edge node {} (16)
		
(9) edge node {} (10)

(10) edge node {} (12)
     		
(11) edge node {} (13)
     edge node {} (15)

(12) edge node {} (14)
     edge node {} (16)
		
(13) edge node {} (17)

(14) edge node {} (17)

(15) edge node {} (18)

(16) edge node {} (18)

(17) edge node {} (18)
;
\draw[line width = 3pt] (3.center) -- (7.center) -- (15.center) -- (18.center) -- (16.center) -- (8.center) -- (9.center);
;
\draw[line width = 3pt] (1.center) -- (2.center) -- (11.center) -- (13.center) -- (17.center) -- (14.center) -- (12.center) -- (10.center) -- (6.center);
}
\subcaption*{2.1: 1.4: ((a,d),(b,c))}
\end{subfigure}
\begin{subfigure}{0.24\linewidth}
\tikz[fo/.style={draw, fill=black, circle, minimum size={0.16cm}, inner 
sep=0cm, font=\bf, align=center, scale=0.88},scale=0.5] {
\node [fo, label=left:$d$] (1) at (0,0) {};
\node [fo] (2) at (2,0) {};
\node [fo, label=right:$c$] (3) at (4,0) {};
\node [fo] (5) at (0,4) {};
\node [fo, label=left:$a$] (6) at (0,6) {};
\node [fo] (7) at (4,2) {};
\node [fo] (8) at (4,4) {};
\node [fo, label=right:$b$] (9) at (4,6) {};
\node [fo] (10) at (2,6) {};
\node [fo] (11) at (2,1) {};
\node [fo] (12) at (2,5) {};
\node [fo] (14) at (1,4) {};
\node [fo] (15) at (3,2) {};
\node [fo] (16) at (3,4) {};
\node [fo] (17) at (1.5,3) {};
\node [fo] (18) at (2.5,3) {};

\draw

(1) edge node  {} (2)
(2) edge node {} (3)
    edge node {} (11)
		
(3) edge node {} (7)

(5) edge node {} (6)
    edge node {} (14)
	
(6) edge node {} (10)

(7) edge node {} (8)
    edge node {} (15)
		
(8) edge node {} (9)
    edge node {} (16)
		
(9) edge node {} (10)

(10) edge node {} (12)
     		
(11) edge node {} (15)

(12) edge node {} (14)
     edge node {} (16)
		
(14) edge node {} (17)

(15) edge node {} (18)

(16) edge node {} (18)

(17) edge node {} (18)
;
\draw[line width = 3pt] (3.center) -- (7.center) -- (8.center) -- (16.center) -- (12.center) -- (10.center) -- (9.center);
;
\draw[line width = 3pt] (1.center) -- (2.center) -- (11.center) -- (15.center) -- (18.center) -- (17.center) -- (14.center) -- (5.center) -- (6.center);
}
\subcaption*{2.1: 1.4: ((a,d),(b,c))}
\end{subfigure}
\begin{subfigure}{0.24\linewidth}
\tikz[fo/.style={draw, fill=black, circle, minimum size={0.16cm}, inner 
sep=0cm, font=\bf, align=center, scale=0.88},scale=0.5] {
\node [fo, label=left:$d$] (1) at (0,0) {};
\node [fo] (2) at (2,0) {};
\node [fo, label=right:$c$] (3) at (4,0) {};
\node [fo] (4) at (0,2) {};
\node [fo] (5) at (0,4) {};
\node [fo, label=left:$a$] (6) at (0,6) {};
\node [fo] (7) at (4,2) {};
\node [fo] (8) at (4,4) {};
\node [fo] (10) at (2,6) {};
\node [fo] (12) at (2,5) {};
\node [fo] (14) at (1,4) {};
\node [fo] (15) at (3,2) {};
\node [fo] (16) at (3,4) {};
\node [fo] (17) at (1.5,3) {};
\node [fo] (18) at (2.5,3) {};

\draw

(1) edge node  {} (2)
    edge node {} (4)

(2) edge node {} (3)
(3) edge node {} (7)

(4) edge node {} (5)
(5) edge node {} (6)
    edge node {} (14)
	
(6) edge node {} (10)

(7) edge node {} (8)
    edge node {} (15)
		
(8) edge node {} (16)
		
(10) edge node {} (12)
     		
(12) edge node {} (14)
     edge node {} (16)
		
(14) edge node {} (17)

(15) edge node {} (18)

(16) edge node {} (18)

(17) edge node {} (18)
;
\draw[line width = 3pt] (3.center) -- (2.center) -- (1.center) -- (4.center) -- (5.center) -- (14.center) -- (17.center) -- (18.center) -- (15.center) -- (7.center) -- (8.center) -- (16.center) -- (12.center) -- (10.center) -- (6.center);
}
\subcaption*{2.1: 1.5: (a,c)}
\end{subfigure}
\begin{subfigure}{0.24\linewidth}
\tikz[fo/.style={draw, fill=black, circle, minimum size={0.16cm}, inner 
sep=0cm, font=\bf, align=center, scale=0.88},scale=0.5] {
\node [fo, label=left:$d$] (1) at (0,0) {};
\node [fo] (2) at (2,0) {};
\node [fo, label=right:$c$] (3) at (4,0) {};
\node [fo] (4) at (0,2) {};
\node [fo] (5) at (0,4) {};
\node [fo, label=left:$a$] (6) at (0,6) {};
\node [fo] (7) at (4,2) {};
\node [fo] (8) at (4,4) {};
\node [fo, label=right:$b$] (9) at (4,6) {};
\node [fo] (10) at (2,6) {};
\node [fo] (11) at (2,1) {};
\node [fo] (12) at (2,5) {};
\node [fo] (14) at (1,4) {};
\node [fo] (15) at (3,2) {};
\node [fo] (16) at (3,4) {};
\node [fo] (18) at (2.5,3) {};

\draw

(1) edge node  {} (2)
    edge node {} (4)

(2) edge node {} (3)
    edge node {} (11)
		
(3) edge node {} (7)

(4) edge node {} (5)
(5) edge node {} (6)
    edge node {} (14)
	
(6) edge node {} (10)

(7) edge node {} (8)
    edge node {} (15)
		
(8) edge node {} (9)
    edge node {} (16)
		
(9) edge node {} (10)

(10) edge node {} (12)
     		
(11) edge node {} (15)

(12) edge node {} (14)
     edge node {} (16)
		
(15) edge node {} (18)

(16) edge node {} (18)

;
\draw[line width = 3pt] (1.center) -- (4.center) -- (5.center) -- (14.center) -- (12.center) -- (16.center) -- (18.center) -- (15.center) -- (11.center) -- (2.center) -- (3.center) -- (7.center) -- (8.center) -- (9.center) -- (10.center) -- (6.center);
}
\subcaption*{2.1: 1.2: (a,d)}
\end{subfigure}
\begin{subfigure}{0.24\linewidth}
\tikz[fo/.style={draw, fill=black, circle, minimum size={0.16cm}, inner 
sep=0cm, font=\bf, align=center, scale=0.88},scale=0.5] {
\node [fo, label=left:$d$] (1) at (0,0) {};
\node [fo] (2) at (2,0) {};
\node [fo, label=right:$c$] (3) at (4,0) {};
\node [fo] (4) at (0,2) {};
\node [fo] (5) at (0,4) {};
\node [fo, label=left:$a$] (6) at (0,6) {};
\node [fo] (7) at (4,2) {};
\node [fo] (8) at (4,4) {};
\node [fo, label=right:$b$] (9) at (4,6) {};
\node [fo] (10) at (2,6) {};
\node [fo] (11) at (2,1) {};
\node [fo] (12) at (2,5) {};
\node [fo] (13) at (1,2) {};
\node [fo] (14) at (1,4) {};
\node [fo] (15) at (3,2) {};
\node [fo] (16) at (3,4) {};

\draw

(1) edge node  {} (2)
    edge node {} (4)

(2) edge node {} (3)
    edge node {} (11)
		
(3) edge node {} (7)

(4) edge node {} (5)
    edge node {} (13)

(5) edge node {} (6)
    edge node {} (14)
	
(6) edge node {} (10)

(7) edge node {} (8)
    edge node {} (15)
		
(8) edge node {} (9)
    edge node {} (16)
		
(9) edge node {} (10)

(10) edge node {} (12)
     		
(11) edge node {} (13)
     edge node {} (15)

(12) edge node {} (14)
     edge node {} (16)
		
;
\draw[line width = 3pt] (1.center) -- (2.center) -- (3.center) -- (7.center) -- (15.center) -- (11.center) -- (13.center) -- (4.center) -- (5.center) -- (14.center) -- (12.center) -- (16.center) -- (8.center) -- (9.center) -- (10.center) -- (6.center);
}
\subcaption*{2.1: 1.2: (a,d)}
\end{subfigure}
\begin{subfigure}{0.24\linewidth}
\tikz[fo/.style={draw, fill=black, circle, minimum size={0.16cm}, inner 
sep=0cm, font=\bf, align=center, scale=0.88},scale=0.5] {
\node [fo, label=left:$d$] (1) at (0,0) {};
\node [fo] (2) at (2,0) {};
\node [fo, label=right:$c$] (3) at (4,0) {};
\node [fo] (4) at (0,2) {};
\node [fo] (7) at (4,2) {};
\node [fo] (8) at (4,4) {};
\node [fo, label=right:$b$] (9) at (4,6) {};
\node [fo] (10) at (2,6) {};
\node [fo] (11) at (2,1) {};
\node [fo] (12) at (2,5) {};
\node [fo] (13) at (1,2) {};
\node [fo] (14) at (1,4) {};
\node [fo] (15) at (3,2) {};
\node [fo] (16) at (3,4) {};
\node [fo] (17) at (1.5,3) {};
\node [fo] (18) at (2.5,3) {};

\draw

(1) edge node  {} (2)
    edge node {} (4)

(2) edge node {} (3)
    edge node {} (11)
		
(3) edge node {} (7)

(4) edge node {} (13)

(7) edge node {} (8)
    edge node {} (15)
		
(8) edge node {} (9)
    edge node {} (16)
		
(9) edge node {} (10)

(10) edge node {} (12)
     		
(11) edge node {} (13)
     edge node {} (15)

(12) edge node {} (14)
     edge node {} (16)
		
(13) edge node {} (17)

(14) edge node {} (17)

(15) edge node {} (18)

(16) edge node {} (18)

(17) edge node {} (18)
;
\draw[line width = 3pt] (3.center) -- (2.center) -- (1.center) -- (4.center) -- (13.center) -- (11.center) -- (15.center) -- (7.center) -- (8.center) -- (16.center) -- (18.center) -- (17.center) -- (14.center) -- (12.center) -- (10.center) -- (9.center);
;}
\subcaption*{2.2: (b,c)}
\end{subfigure}
\begin{subfigure}{0.24\linewidth}
\tikz[fo/.style={draw, fill=black, circle, minimum size={0.16cm}, inner 
sep=0cm, font=\bf, align=center, scale=0.88},scale=0.5] {
\node [fo, label=left:$d$] (1) at (0,0) {};
\node [fo] (2) at (2,0) {};
\node [fo, label=right:$c$] (3) at (4,0) {};
\node [fo] (4) at (0,2) {};
\node [fo] (5) at (0,4) {};
\node [fo] (7) at (4,2) {};
\node [fo] (8) at (4,4) {};
\node [fo, label=right:$b$] (9) at (4,6) {};
\node [fo] (11) at (2,1) {};
\node [fo] (12) at (2,5) {};
\node [fo] (13) at (1,2) {};
\node [fo] (14) at (1,4) {};
\node [fo] (15) at (3,2) {};
\node [fo] (16) at (3,4) {};
\node [fo] (17) at (1.5,3) {};
\node [fo] (18) at (2.5,3) {};

\draw

(1) edge node  {} (2)
    edge node {} (4)

(2) edge node {} (3)
    edge node {} (11)
		
(3) edge node {} (7)

(4) edge node {} (5)
    edge node {} (13)

(5) edge node {} (14)
	
(7) edge node {} (8)
    edge node {} (15)
		
(8) edge node {} (9)
    edge node {} (16)
		
(11) edge node {} (13)
     edge node {} (15)

(12) edge node {} (14)
     edge node {} (16)
		
(13) edge node {} (17)

(14) edge node {} (17)

(15) edge node {} (18)

(16) edge node {} (18)

(17) edge node {} (18)
;
\draw[line width = 3pt] (3.center) -- (2.center) -- (1.center) -- (4.center) -- (5.center) -- (14.center) -- (12.center) -- (16.center) -- (18.center) -- (17.center) -- (13.center) -- (11.center) -- (15.center) -- (7.center) -- (8.center) -- (9.center);
;}
\subcaption*{2.2: (b,c)}
\end{subfigure}
%	\end{adjustbox}
	\caption{Proof of properties 2.1 and 2.2-2.5.}
	\label{fig:j182.1-2.5}
\end{figure}

\subsection{Counts of \texorpdfstring{\boldmath${K_2}$}{K2}-hypohamiltonian graphs}
\label{app:counts_k2hypo}

%\begin{table}[ht!]
\begin{table}[H]
\centering
\begin{tabular}{c | c | r | r | c | c}
	Order & Girth & Total & Non-hamiltonian & $K_2$-hypoham. &
	hypo- and $K_2$-\\
	\hline
	$5$--$9$ & $\geq 3$ & $273\, 183$
	 & $2\, 500\, 918$
	  & $0$ & $0$\\
	  $10$ & $\geq 3$ & $11\,716\,571$ & $2\,411\,453
	  $ & $1$ & $1$\\
	  $11$ & $\geq 3$ & $1\,006\,700\,565$ & $123\,544\,541$ & $0$ & $0$\\
	  $12$ & $\geq 3$ & $164\,059\,830\,476
	  $ & $11\,537\,642\,646$ & $0$ & $0$\\
	  $13$ & $\geq 3$ & $50\,335\,907\,869\,219
	  $ & $2\,013\,389\,528\,672
	  $ & $1$ & $1$\\
	  \hline
	  $14$ & $\geq 4$ & $445\,781\,050
	  $ & $297\,831\,030
	  $ & $0$ & $0$\\
	  $15$ & $\geq 4$ & $13\,743\,625\,184
	  $ & $8\,160\,280\,174
	  $ & $1$ & $1$\\
	  $16$ & $\geq 4$ & $566\,756\,900\,370
	  $ & $290\,053\,195\,532
	  $ & $4$ & $4$\\
	  \hline
	  $17$ & $\geq 5$ & $179\,593\,823
	  $ & $167\,592\,079$ & $0$ & $0$\\
	  $18$ & $\geq 5$ & $2\,098\,423\,758
	  $ & $1\,923\,066\,160$ & $3$ & $2$\\
	  $19$ & $\geq 5$ & $28\,215\,583\,324
	  $ & $25\,294\,705\,218$ & $28$ & $28$\\
	  $20$ & $\geq 5$ & $434\,936\,005\,284
	  $ & $379\,901\,735\,283$ & $2$ & $2$\\
	  $21$ & $\geq 5$ & $7\,662\,164\,738\,118
	  $ & $6\,490\,985\,473\,090$ & $31$ & $28$\\
	  \hline
	  $22$ & $\geq 6$ & $38\,259\,063\,075
	  $ & $37\,983\,633\,685$ & $0$ & $0$\\
	  $23$ & $\geq 6$ & $342\,877\,495\,737
	  $ & $340\,026\,856\,312$ & $0$ & $0$\\
	  $24$ & $\geq 6$ & $3\,340\,738\,751\,508
	  $ & $3\,307\,741\,496\,902$ & $0$ & $0$\\
	  \hline
	  $25$ & $\geq 7$ & $91\,618\,528\,203
	  $ & $91\,585\,941\,929$ & $0$ & $0$\\
	  	
\end{tabular}
\caption{Counts of all non-hamiltonian and $K_2$-hypohamiltonian graphs of a
given order and girth. The last column lists the number of graphs that are both hypohamiltonian and $K_2$-hypohamiltonian.}
\label{table:counts_general}
\end{table}

\subsection{All cubic planar \texorpdfstring{\boldmath${K_2}$}{K2}-hypohamiltonian graphs on 76 and 
78 
vertices}
\label{app:cubic_pl}

%\begin{figure}[!htb]
\begin{figure}[H]
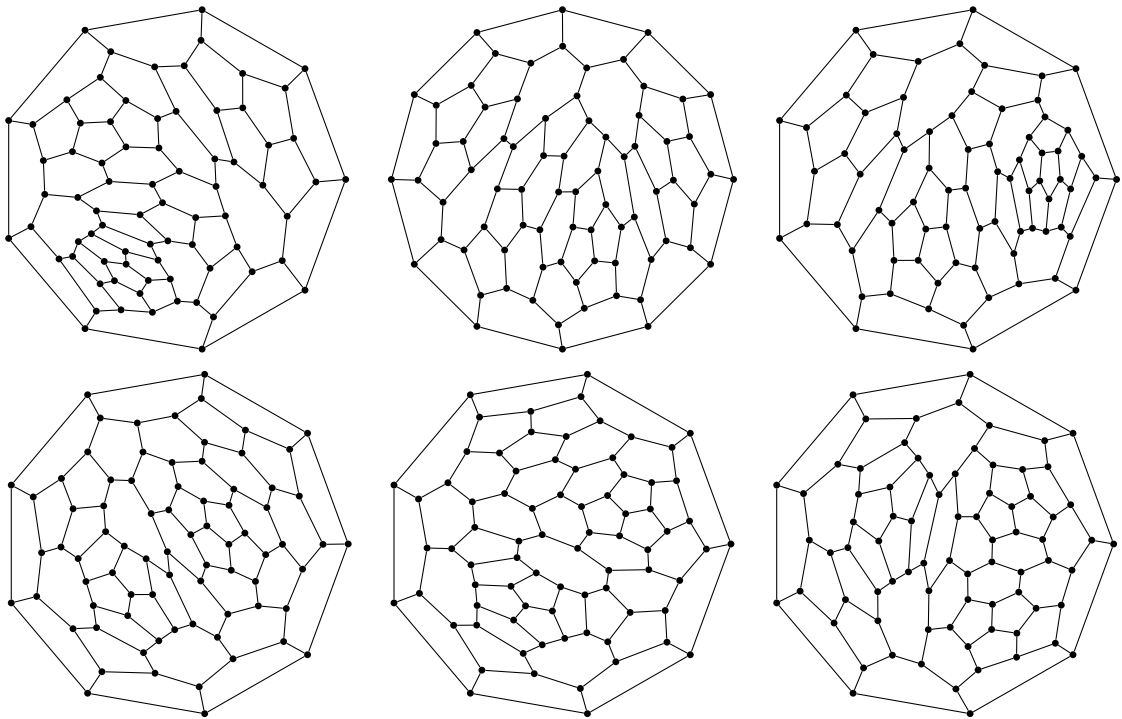

\begin{center}
%

\caption{The three cubic planar $K_2$-hypohamiltonian graphs on 76 vertices 
(top row) and the three such graphs on 78 vertices (bottom row), respectively. 
%\Carol{I guess the top row are the 76-vertex graphs and the bottom row the 78-vertex graphs? Say so.}
}
\label{fig:cubic_planar_76_78}
\end{center}
\end{figure}

\subsection{Extendable 5-cycles}

%\Jan{TODO: this subsection is incomplete. If possible, it would be nice if the pictures could use the same (tikz) style as the other figures in the manuscript.}

%\begin{figure}[!htb]
%\begin{center}
%	\begin{adjustbox}{minipage=\linewidth, scale=0.97}
%		\centering
%		\input{48-ext-cyc.tex}
%	\end{adjustbox}
%	\caption{A planar $K_2$-hypohamiltonian graph and the proof that it
%	contains an extendable 5-cycle. \Jan{Warning: We don't seem to refer to 
%this figure? Do we need it?}}
%	\label{fig:48-ext-cyc}
%\end{center}
%\end{figure}

\begin{figure}[H]
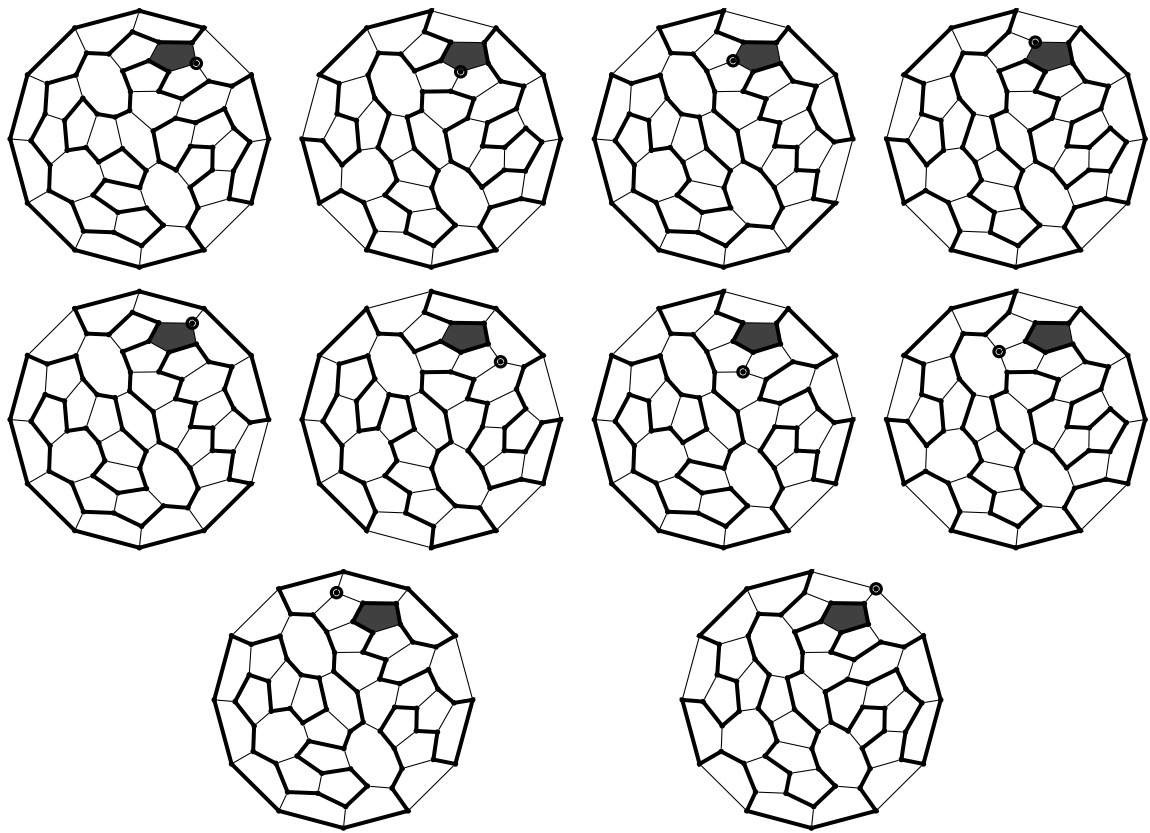

\begin{center}
{
\hfill
}
\caption{A cubic planar $K_2$-hypohamiltonian graph and the proof that it
contains an extendable 5-cycle.}
\label{fig:68-ext-cyc}
\end{center}
\end{figure}

\end{document}